\documentclass[reqno]{amsart}
\usepackage{amsfonts}
\usepackage{amsmath}
\usepackage{amsthm, amscd}
\usepackage{mathtools}
\usepackage{amssymb}
\usepackage{mathrsfs}

\usepackage[alphabetic]{amsrefs}
\usepackage{etex}
\usepackage{hyperref}
\hypersetup{
  colorlinks = true,
  linkcolor  = black
}

\usepackage{tikz}
\usepackage{xypic}

\allowdisplaybreaks

\theoremstyle{plain}\newtheorem{Theorem}{Theorem}[section]
\theoremstyle{plain}\newtheorem{Corollary}[Theorem]{Corollary}
\theoremstyle{plain}\newtheorem{Lemma}[Theorem]{Lemma}
\theoremstyle{plain}\newtheorem{Definition}[Theorem]{Definition}
\theoremstyle{plain}\newtheorem{Proposition}[Theorem]{Proposition}
\theoremstyle{plain}
\theoremstyle{plain}\newtheorem{Condition}[Theorem]{Condition}
\theoremstyle{plain}\newtheorem*{Theorem*}{Theorem}
\theoremstyle{remark}\newtheorem{remark}[Theorem]{Remark}

\theoremstyle{remark}\newtheorem{eg}[Theorem]{Example}
\theoremstyle{plain}\newtheorem{question}[Theorem]{Question}
\theoremstyle{plain}

\DeclareMathOperator{\rank}{rank}
\DeclareMathOperator{\II}{I}

\DeclareMathOperator{\Imm}{Im}

\DeclareMathOperator{\AHI}{AHI}

\DeclareMathOperator{\muu}{\mu^{orb}}

\DeclareMathOperator{\SO}{SO}

\DeclareMathOperator{\Hom}{Hom}

\DeclareMathOperator{\id}{id}
\DeclareMathOperator{\Tor}{Tor}

\DeclareMathOperator{\Kh}{Kh}
\DeclareMathOperator{\Khr}{Khr}
\DeclareMathOperator{\lk}{lk}
\DeclareMathOperator{\KHI}{KHI}
\DeclareMathOperator{\interior}{interior}
\DeclareMathOperator{\len}{length}

\DeclareMathOperator{\HFK}{\widehat{HFK}}
\DeclareMathOperator{\HFL}{\widehat{HFL}}

\newcommand{\bC}{\mathbb{C}}

\newcommand{\bQ}{\mathbb{Q}}
\newcommand{\bR}{\mathbb{R}}

\newcommand{\bZ}{\mathbb{Z}}

\newcommand{\bF}{\mathbb{F}}

\author{Yi Xie}
\address{Beijing International Center for Mathematical Research, Peking University, Beijing 100871, China}
\email{yixie@pku.edu.cn}
\author{Boyu Zhang}
\address{Department of Mathematics, Princeton University, New Jersey 08544, USA}
\email{bz@math.princeton.edu}

\title{Classification of links with Khovanov homology of minimal rank}

\begin{document}
\begin{abstract}
If $L$ is an oriented link with $n$ components, then the rank of its Khovanov homology  is at least $2^n$. We classify all the links whose Khovanov homology with $\mathbb{Z}/2$-coefficients achieves this lower bound, and show that such links can be obtained by iterated connected sums and disjoint unions of Hopf links and unknots. This gives a positive answer to a question asked by Batson and Seed \cite{Kh-unlink}. 
\end{abstract}

\maketitle 
\setcounter{tocdepth}{1}
\tableofcontents

\section{Introduction}
Let $L$ be an oriented link in $S^3$ and $R$ be a ring. Khovanov homology \cite{Kh-Jones} assigns a bi-graded $R$-module $\Kh(L;R)$ to the link $L$. In this paper, we will take the coefficient ring to be $\bZ/2$.  The Euler characteristics of $\Kh(L;\bZ/2)$ recover the coefficients of the unreduced Jones polynomial of $L$. If $L$ has $n$ components, then the value of the unreduced Jones polynomial at $t=1$ equals $(-2)^n$, therefore
\begin{equation}\label{eqn_lower_bound_Kh}
\rank_{\bZ/2} \Kh(L;\bZ/2)\ge 2^n.
\end{equation}

The rank of $\Kh(L;\bZ/2)$ is independent of the orientation of $L$. 
If $L$ is the unlink, then $\rank_{\bZ/2} \Kh(L;\bZ/2)= 2^n.$ However, the unlink is not the only case that \eqref{eqn_lower_bound_Kh} achieves equality. For example, when $n=2$,  the Khovanov homology of the Hopf link has rank $4$. More generally, for arbitrary $n$, we have the following construction. In graph theory, a \emph{forest} is a simple graph without cycles. Given a forest $G$, define a link $L_G$ by placing an unknot at each vertex of $G$ and linking two unknots as a Hopf link whenever there is an edge connecting the corresponding vertices 
(see Figures \ref{fig_tree1}, \ref{fig_tree2}, \ref{fig_tree3} for examples). The link $L_G$ is called the \emph{forest of unknots} defined by $G$. Every forest of unknots can be obtained by iterated connected sums and disjoint unions of  Hopf links and unknots.
\begin{figure}  
  \includegraphics[width=0.8\linewidth]{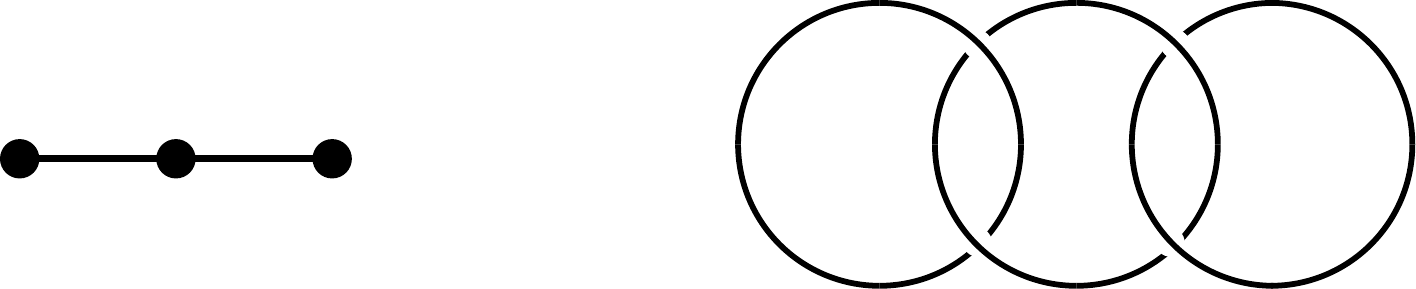}
  \caption{Example of a forest of unknots.}
  \label{fig_tree1}
  \vspace{\baselineskip}
  \includegraphics[width=0.8\linewidth]{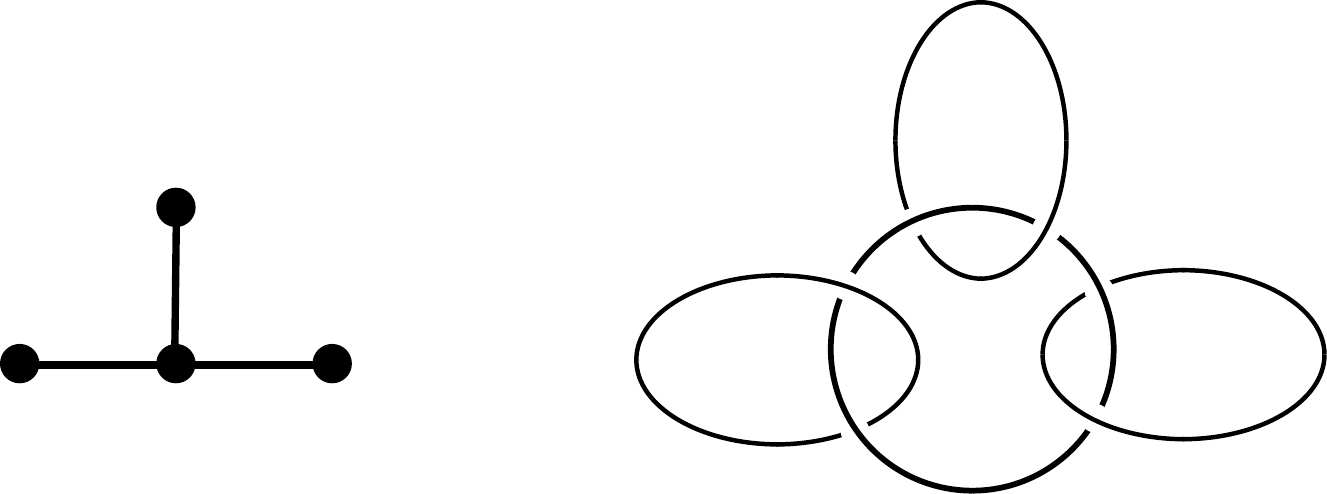}
  \caption{Example of a forest of unknots.}
  \label{fig_tree2}
  \vspace{\baselineskip}
  \includegraphics[width=0.8\linewidth]{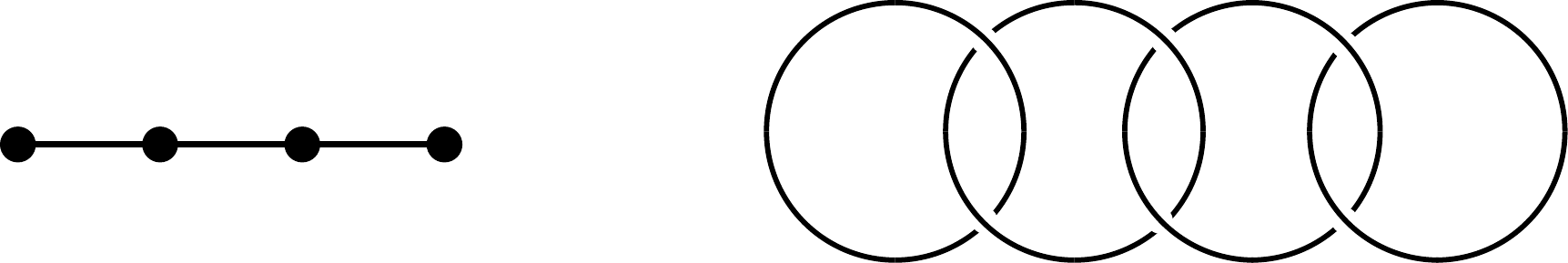}
  \caption{Example of a forest of unknots.}
  \label{fig_tree3}
\end{figure}
By \cite[Corollary 6.6]{AP:Kh_torsion} and the K\"unneth formula, if $L_G$ is a forest of unknots with $n$ components, then $\rank_{\bZ/2} \Kh(L_G;\bZ/2)= 2^n$.

The following question was asked by Batson and Seed:
\begin{question}[{\cite[Question 7.2]{Kh-unlink}}]\label{question-Batson-Seed}
Are forests of unknots the only $n$-component links with Khovanov homology of rank
$2^n$ over $\bZ/2$\emph{?}
\end{question}

Our main result gives an affirmative answer to the above question.
\begin{Theorem}\label{The_main_theorem}
If $L$ is an $n$-component link such that $\rank_{\bZ/2} \Kh(L;\bZ/2)= 2^n$, then $L$ is a forest of unknots.
\end{Theorem}

In the special case that $L$ is an alternating link, Theorem \ref{The_main_theorem} was proved by Shumakovitch \cite[Lemma 3.3.C]{Shu:torsion_Kh}.
In fact, let $d(L)$ be the sum of the absolute values of the coefficients of the (reduced) Jones polynomial of $L$, 
then Shumakovitch \cite[Lemma 3.3.B]{Shu:torsion_Kh} proved that every  $n$-component alternating link $L$ with $d(L)=2^{n-1}$ is a forest of unknots. Notice that the rank of the reduced Khovanov homology $\Khr(L)$ is at least $d(L)$, and by \cite[Corollary 3.2.C]{Shu:torsion_Kh} we have 
\begin{equation}\label{eqn_relation_Kh_Khr}
\rank_{\bZ/2}\Kh(L;\bZ/2)=2\rank_{\bZ/2}\Khr(L;\bZ/2),
\end{equation}
therefore the statement above implies
Theorem \ref{The_main_theorem} for alternating links.
On the other hand, if $L$ is not required to be alternating, then it is possible to have $d(L)=2^{n-1}$ with $L$ not being a forest of unknots. An infinite family of such examples are constructed in \cite{EKT:trivial_Jones}.

Theorem \ref{The_main_theorem} is only stated for $\bZ/2$-coefficients because the proof relies on equation \eqref{eqn_relation_Kh_Khr}, which is only valid in characteristic 2. The result of Theorem \ref{The_main_theorem} would still hold for an arbitrary field $\bF$ by the same proof, if we further assumed that $\rank_\bF\Khr(L;\bF)=2^{n-1}$ for every base point. At the moment of writing, it is not clear to the authors whether the same result still holds if we only assume $\rank_{\mathbb{F}} \Kh(L;\mathbb{F})= 2^n$ for a field 
$\mathbb{F}\neq \bZ/2$.

The detection properties of Khovanov homology have been studied intensively in the past decade. The first breakthrough of this field was the landmark paper by Kronheimer and Mrowka \cite{KM:Kh-unknot}, which
proved that Khovanov homology detects the unknot. Since then, several other 
detection results have been proved.  The following list is a summary of related results in chronological order:
\begin{enumerate}
\item \cite{KM:Kh-unknot} The rank of the reduced Khovanov homology with $\bQ$-coefficients detects the unknot;
\item \cite{HN-unlink} The Khovanov homology with $\bZ/2$-coefficients, together with an extra module structure, detects the unlink;
\item \cite{Kh-unlink} The Khovanov homology with $\bZ/2$-coefficients detects the unlink;
\item \cite{BS} The rank of the reduced Khovanov homology with $\bQ$-coefficients detects the trefoil;
\item \cite{BSX} Both the Khovanov homology and the reduced Khovanov homology, with coefficients in either $\bZ$ or $\bZ/2$, detect the Hopf link.
\end{enumerate}

Theorem \ref{The_main_theorem} is a generalization of (2), (3) and (5) above. In fact, we have the following two corollaries of 
Theorem \ref{The_main_theorem}.

\begin{Corollary}
Suppose $L_1$ and $L_2$ are two oriented links with $n$ components and $L_1$ is a forest of unknots. If there exist $e,f\in\bZ$ such that
$$
\Kh^{i,j}(L_1;\bZ/2)\cong \Kh^{i+e,j+f}(L_2;\bZ/2)
$$
for all $i,j\in \bZ$,
then $L_2$ is isotopic to a forest of unknots whose graph has the same number of edges as the graph of  $L_1$. In particular, if
$$
\Kh^{i,j}(L_1;\bZ/2)\cong \Kh^{i,j}(L_2;\bZ/2)
$$
for all $i,j\in \bZ$, and if
$L_1$ is a Hopf link, or
the connected sum of two Hopf links, or the disjoint union of a Hopf link and the unlink, or the unlink, then
$L_2$ is isotopic to $L_1$. 
\end{Corollary}
\begin{proof}
Suppose $T$ is a tree with
$k$ vertices and let $L_T$ be the forest of unknots given by $T$. If $L_T$ is oriented such that the linking numbers between the components of $L_T$ are all non-negative, then by \cite[Corollary 6.6]{AP:Kh_torsion}
we have
$$
P(L_T):=\sum_{i,j}t^iq^j\rank_{\bZ/2}\Kh(L_T;\bZ/2)=t^{k-1}q^{3(k-1)}(q+q^{-1})(tq^2+t^{-1}q^{-2})^{k-1}.
$$

Theorem \ref{The_main_theorem} implies $L_2$ is a forest of unknots. 
Let $G=T_1\sqcup\cdots\sqcup T_l$ be the graph of $L_2$, where $T_i$ is a tree with $k_i$ vertices and $n=\sum k_i$. Then the K\"unneth formula shows that
$$
P(L_G)=t^aq^b (q+q^{-1})^l(tq^2+t^{-1}q^{-2})^{n-l},
$$
where $a$ and $b$ are integers depending on the orientaion of $L_2$.
Since $(q+q^{-1})$ and $(tq^2+t^{-1}q^{-2})$ are irreducible polynomials in the unique factorization domain $\bZ[t,t^{-1},q,q^{-1}]$,
the value of $n-l$ is determined by $P(L_G)$. Therefore the first part of the corollary is proved. For the second part, notice that
in these four cases the graph for $L_1$ is uniquely determined by the number of edges. 
\end{proof}
 
 Given a link $L$ with $n$ components, one can equip $\Kh(L;\bZ/2)$ with a module structure over the ring
 $$
 R_n:=(\bZ/2)[X_1,\cdots, X_n]/(X_1^2,\cdots,X_n^2).
 $$
 For the definition of the module structure,
 the reader may refer to \cite[Section 2]{HN-unlink} and \cite[Section 3]{Khovanov-pattern}.   
\begin{Corollary}
Suppose $L_1$ and $L_2$ are two links with $n$ components, and suppose $L_1$ is a forest of unknots with graph $G_1$. 
If $\Kh(L_1;\bZ/2)$ is isomorphic 
to $\Kh(L_2;\bZ/2)$ as $R_n$-modules, then $L_2$ is isotopic to a forest of unknots with graph 
$G_2$ such that
\begin{itemize}
\item there is an one-to-one correspondence between the connected components of $G_1$ and the connected components of $G_2$;
\item the corresponding components of $G_1$ and $G_2$ have the same number of vertices.
\end{itemize}
If we further assume
the number of vertices of every connected component
of $G_1$ is less than or equal to $3$, then $L_2$ is isotopic to $L_1$ as unoriented links.
\end{Corollary}
\begin{proof}
Let $U_n$ be the $n$-component unlink and $H$ be the Hopf link, it is known that $\Kh(U_n;\bZ/2)\cong R_n$ and $\Kh(H;\bZ/2)\cong R_2/(X_1-X_2)\oplus R_2/(X_1-X_2) $. 
For $k\in \bZ^+$, let $H_{k-1}$ be a forest of unknots with $k$ components whose graph is a tree, then for $k\ge 2$, the link $H_{k-1}$ is given by a connected sum of $k-1$ Hopf links.
We use induction on $k$ to show that the Khovanov module of $H_{k-1}$ is given by
\begin{equation}\label{eq_Kh_module_Hk-1}
\Kh(H_{k-1};\bZ/2)\cong[R_{k}/(X_1=X_2=\cdots=X_k)]^{\oplus 2^k}.
\end{equation}
The cases of $k=1, 2$ follow from the formulas above.
 Suppose \eqref{eq_Kh_module_Hk-1} holds for $k=l$ with $l\ge 2$. To show that \eqref{eq_Kh_module_Hk-1} holds for $k=l+1$, write $H_l$ as a connected sum of $H_{l-1}$ and the Hopf link $H$, where $H_{l-1}$ is a connected sum of $l-1$ Hopf links. Label the components of $H_{l-1}$ by $1,2,\cdots,l$ and the components
of $H$ by $l,l+1$, such that $H_l$ is the connected sum along the components with label $l$. By the induction hypothesis, we have
\begin{align*}
\Kh(H_{l-1};\bZ/2)&\cong[R_{l}/(X_1=X_2=\cdots=X_l)]^{\oplus 2^l},
\\
\Kh(H;\bZ/2)&\cong[R'_{2}/(X_l=X_{l+1})]^{\oplus 2},
\end{align*}
where $R'_2:=(\bZ/2)[X_l,X_{l+1}]/(X_l^2,X_{l+1}^2)$. According to \cite[Proposition 3.3]{Khovanov-pattern}, we have
$$
C(H_l)=C(H_{l-1})\otimes_{R_1'} C(H),
$$
where $C(H_l), C(H_{l-1}), C(H)$ are the Khovanov chain complexes which are free modules over $R'_1:=(\bZ/2)[X_l]/(X_l^2)$. 
Moreover, the equation above respects
the actions of $X_1,\cdots, X_{l-1},X_{l+1}$ on both sides. 
Since $R'_1$ is not a principal ideal domain, we do not have a K\"unneth formula, but there is a K\"unneth spectral sequence
$$
E^2=\bigoplus_{j=0}^{\infty}\Tor_{R'_1}^j(\Kh(H_{l-1};\bZ/2), \Kh(H;\bZ/2))\Rightarrow \Kh(H_{l};\bZ/2).
$$ 
Since $\Kh(H_{l-1};\bZ/2)$ and $\Kh(H;\bZ/2)$ are free $R'_1$-modules, we have
$$
E^2=\Kh(H_{l-1};\bZ/2)\otimes_{R'_1} \Kh(H;\bZ/2).
$$
The knowledge of the $(\bZ/2)$-ranks of Khovanov homology shows that the spectral sequence collapses at the $E^2$-page.
Hence we have 
\begin{align*}
  \Kh(H_l;\bZ/2) &\cong \Kh(H_{l-1};\bZ/2)\otimes_{R'_1} \Kh(H;\bZ/2)
 \\
 & \cong  [R_{l+1}/(X_1=X_2=\cdots=X_{l+1})]^{\oplus 2^{l+1}}
\end{align*}
as $R_{l+1}$-modules. This completes the induction. The Khovanov module of the disjoint union of two links is the tensor product
of the Khovanov modules of the two links over $\bZ/2$. Theorem \ref{The_main_theorem} implies $L_2$ is a forest of unknots
with a graph $G_2$.
It is clear from the above discussion that  the module structure of $\Kh(L_2;\bZ/2)$ determines the number of vertices in each component of $G_2$,
hence the corollary is proved.
\end{proof}
\begin{remark}
Neither the bi-grading nor the module structure of Khovanov homology can distinguish the links given by Figure \ref{fig_tree2} 
and Figure \ref{fig_tree3}. 
\end{remark}

The proof of Theorem \ref{The_main_theorem} relies on Kronheimer-Mrowka's spectral sequence \cite{KM:Kh-unknot} 
and Batson-Seed's inequality comparing the ranks of the Khovanov homology of a link and its sublinks \cite{Kh-unlink}. Under the assumption of Theorem \ref{The_main_theorem}, Batson-Seed's inequality and Kronheimer-Mrowka's unknot detection theorem imply that all the components of $L$ are unknots. On the other hand,
Kronheimer-Mrowka's spectral sequence gives an upper bound on the rank of the instanton link invariant $\II^\natural(L)$. Since all the components of $L$ are unknots, $\II^\natural(L)$ is 
isomorphic to an annular instanton Floer homology introduced by the first author in \cite{AHI}. The annular 
instanton Floer homology carries a $\bZ$-grading (we call it the f-grading). In \cite{XZ:excision}, the authors showed that
this grading detects the generalized Thurston norm of surfaces with a meridian boundary, which allows us to extract topological information 
from $\II^\natural(L)$. The topological properties imply that either $L$ is a forest of unknots, or $L$ contains a sublink with a very specific configuration. We then use a 
direct computation of Khovanov homology 
and Jones polynomial to rule out the latter case.  
\\

Our method can also be used to prove new dectection results for some links with small Khovanov homology.
\begin{figure}  
  \includegraphics[width=0.4\linewidth]{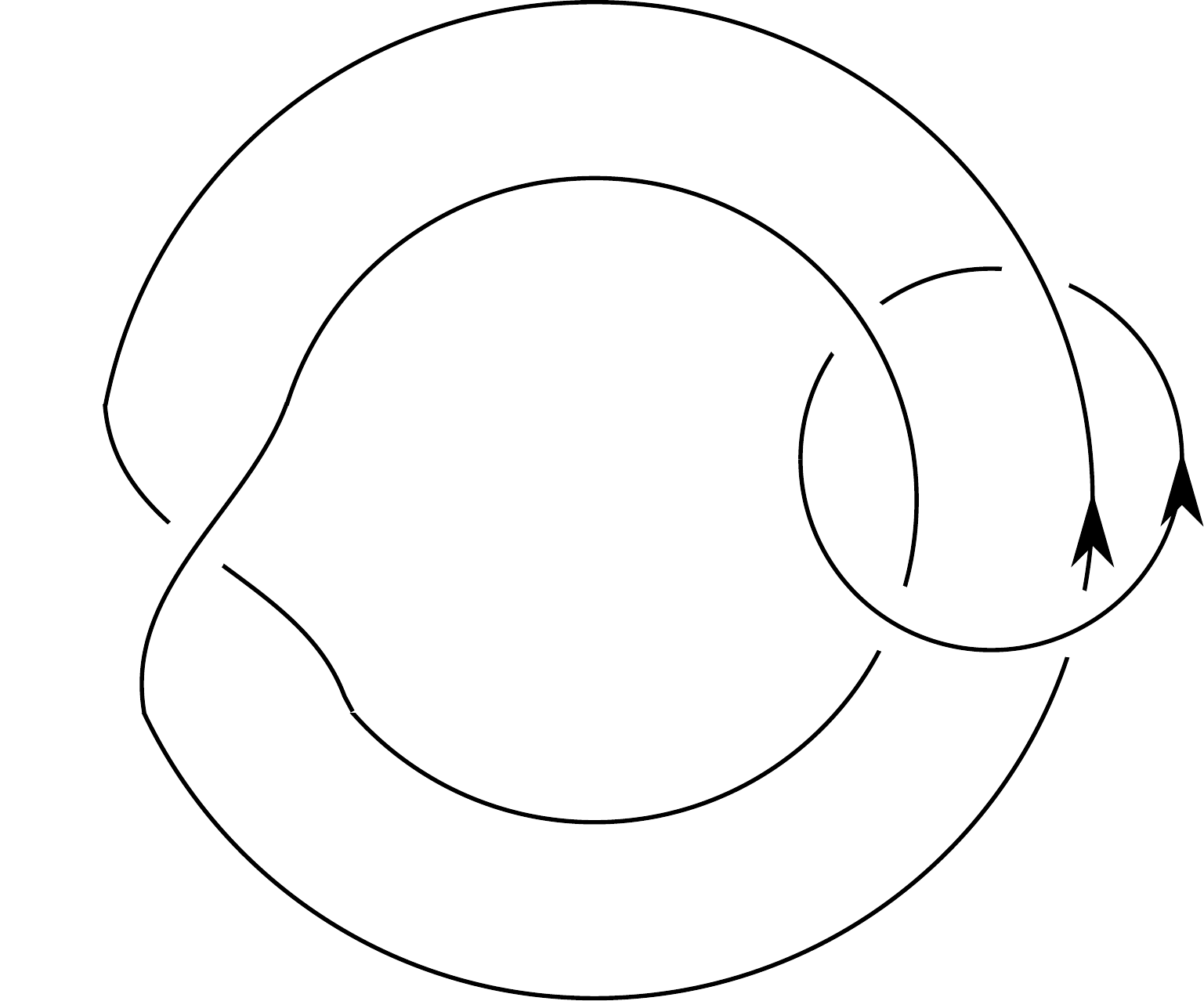}
  \caption{The link $L_1$.}
  \label{fig_braid_1}
 \end{figure}
\begin{Theorem}\label{thm_small_rank_detection_intro}
Let $L_1$ be the oriented link given by Figure \ref{fig_braid_1}, and $L_2$ be the disjoint union of a trefoil and an unknot. 
Let $L=K_1\cup K_2$ be a 2-component oriented link. Then
\begin{itemize}
\item[(a)]
If
$$
\Kh(L;\bZ/2)\cong \Kh(L_1;\bZ/2)
$$
as bi-graded abelian groups, then
$L$ is isotopic to $L_1$.

\item[(b)]
Let $q\in K_2$ and $p$ be a base point on the unknotted component of $L_2$. If
$$
\Khr(L,q;\bZ)\cong \Khr(L_2,p;\bZ)
$$
as bi-graded abelian groups, then
 $L$ splits into the disjoint union of a trefoil $K_1$ and an unknot $K_2$.
\end{itemize}
\end{Theorem}
By the trefoil detection result of Baldwin and Sivek \cite{BS}, the essential content of Part (b) is the splitness of $L$. Recently, Lipshitz and Sarkar proved that  
the splitness of a general link can be detected by the module structure of Khovanov homology \cite{LS-split}.

\subsection*{Acknowledgements}
Part of this work was done during 
the 2019 Summer Program \emph{Quantum Field Theory and Manifold Invariants} at PCMI. The authors would like to thank the organizers
for providing us with such a great environment to carry on this work. We would like to thank John Baldwin and Zhenkun Li for many helpful 
conversations. We also want to thank Nathan Dowlin, Peter Ozsv\'ath, and Zolt\'an Szab\'o for their help with knot Floer homology in the proof of Lemma \ref{Lemma-Kh_bound_linking_number}.

\section{Annular intanton Floer homology}

The singular instanton Floer homology theory was introduced by Kronheimer and Mrowka in \cite{KM:Kh-unknot,KM:YAFT}. Let 
$(Y,L,\omega)$ be a triple 
where $Y$ is a closed oriented 3-manifold, $L\subset Y$ is a link and  
$\omega\subset Y$ is an embedded 1-manifold such that $ \partial \omega=\omega\cap L$.
The triple $(Y,L,\omega)$ is called \emph{admissible} if there is 
an embedded closed surface $\Sigma\subset Y$ satisfying either one of the following conditions:
\begin{itemize}
  \item $\Sigma$ is disjoint from $L$ and the intersection number of $\omega$ and $\Sigma$ is odd,
  \item  The intersection number of $L$ and $\Sigma$ is odd. 
\end{itemize}
If $(Y,L,\omega)$ is admissible, the instanton Floer homology $\II(Y,L,\omega)$ is defined to be a Morse homology of the Chern-Simons functional on a certain space of
orbifold $\SO(3)$-connections over $Y$, where $Y$ is equipped with an orbifold structure with cone angle $\pi$ along $L$, and $\omega$
represents the second Stiefel-Whitney class of the $\SO(3)$-bundle. In this article, we will always take $\bC$-coefficients for instanton Floer 
homologies.

The homology group $\II(Y,L,\omega)$ carries a relative $\mathbb{Z}/4$-homological grading. Given an embedded closed surface $F\subset Y$, there is an operator $\muu(F)$ defined on $\II(Y,L,\omega)$ with degree $2$. For more details the reader may refer to, for example, \cite[Section 2]{XZ:excision}.

The rest of this section gives a brief review of the annular instanton Floer homology introduced in \cite{AHI}.
Let $L$ be a link in the solid torus $S^1\times D^2$. 
The annular instanton Floer homology $\AHI(L)$ is defined by the following procedure:
\begin{enumerate}
\item Let $\mathcal{K}_2$ be the product link $S^1\times \{p_1,p_2\}$ in $S^1\times D^2$, and let $u$ be an arc in $S^1\times D^2$ connecting $S^1\times\{p_1\}$ and $S^1\times\{p_2\}$;

\item Form the new link $L\cup \mathcal{K}_2$ in 
      \begin{equation*}
      S^1\times S^2=S^1\times D^2\cup_{S^1\times S^1} S^1\times D^2,
      \end{equation*}
     where $L$ lies in the first copy of $S^1\times D^2$, and $\mathcal{K}_2$ lies in the second copy. 
\item Define
     \begin{equation*}
      \AHI(L):=\II(S^1\times S^2, L\cup \mathcal{K}_2,u).
     \end{equation*}
\end{enumerate}
The vector space 
$\AHI(L)$ is equipped with an absolute $\mathbb{Z}$-grading (called the f-grading). 
By definition, the component of $\AHI(L)$ with f-degree $i$ is given by the generalized eigenspace of $\muu(S^2)$ for the eigenvalue $i$, and is denoted by $\AHI(L,i)$.
Since $\muu(S^2)$ has degree 2 with respect to the $\mathbb{Z}/4$-homological grading of $\AHI(L)$, the subspace
$\AHI(L,i)$ carries a $\mathbb{Z}/2$-homological grading, and we have
\begin{equation}\label{AHI-symmetry}
\AHI(L,i)\cong \AHI(L,-i).
\end{equation}
There is a product formula for split links in $S^1\times D^2$.
\begin{Proposition}[{\cite[Proposition 4.3]{AHI}}]\label{product-formula}
Suppose $L_1$ and $L_2$ are two links in $S^1\times D_1$ and $S^1\times D_2$ respectively, where $D_1$ and $D_2$ are disjoint sub-disks 
of $D^2$. Then we have
$$
\AHI(L_1\cup L_2)\cong \AHI(L_1)\otimes \AHI(L_2).
$$ 
Moreover, the isomorphism above is compatible with the f-gradings.
\end{Proposition}

In the following, we will use $\mathcal{U}_n$ to denote the unlink with $n$ components in $S^1\times D^2$, and use $\mathcal{K}_n$ to denote the closure
of the trivial braid with $n$ strands in $S^1\times D^2$. We will use $\mathcal{U}_k\cup \mathcal{K}_l$ to denote the union of $\mathcal{U}_k$ and $\mathcal{K}_l$ such that $\mathcal{U}_k$ is included in a solid $3$-ball disjoint from $\mathcal{K}_l$.
\begin{eg}[{\cite[Example 4.2]{AHI}}]\label{eg-Uk-Kl}
The critical set of the unperturbed Chern-Simons functional for $\AHI(\mathcal{U}_1)$ 
is diffeomorphic to $S^2$, and after perturbation the critical set consists of two points whose homological degrees differ by 2. 
Therefore there is no differential, and we have
$$
\AHI(\mathcal{U}_1)\cong \bC\oplus \bC.
$$
The vector space $\AHI(\mathcal{U}_1)$ is supported at the f-grading $0$.

The critical set for $\AHI(\mathcal{K}_1)$ consists of two points whose homological degrees differ by 2, 
hence there is no differential and 
$$
\AHI(\mathcal{K}_1)\cong  \bC\oplus \bC.
$$
The vector space $\AHI(\mathcal{K}_1)$ is supported at f-gradings $\pm 1$.

By Proposition \ref{product-formula}, we have
$$
\AHI(\mathcal{U}_k\cup \mathcal{K}_l)\cong \AHI(\mathcal{U}_1)^{\otimes k}\otimes  \AHI(\mathcal{K}_1)^{\otimes l},
$$
and the above isomorphism preserves the $f$-gradings. 

We also have $\AHI(\emptyset)\cong \mathbb{C}$,
because the critical set consists of a single point.
\end{eg}

\begin{Definition}\label{def_meridional_surface}
A properly embedded, connected, oriented 
surface $S\subset S^1\times D^2$ is called a \emph{meridional surface} if $\partial S$ is a meridian of $S^1\times D^2$. 
\end{Definition}
The annular instanton Floer homology detects the generalized Thurston norm of meridional surfaces.
\begin{Theorem}[{\cite[Theorem 8.2]{XZ:excision}}]\label{Theorem-2g+n}
Given a link $L$ in $S^1\times D^2$ and suppose $S$ is a meridional surface that intersects $L$ transversely. Let $g$ be the genus of $S$ 
and let $n:=|S\cap L|$. Suppose $S$ minimizes the value of $2g+n$ among meridional surfaces,
then we have
\begin{equation*}
\AHI(L,i)= 0
\end{equation*}
for all $|i|> 2g+n$, and
\begin{equation*}
\AHI(L,\pm(2g+n))\neq 0.
\end{equation*}
\end{Theorem}
We also need the following result.
\begin{Proposition}[{\cite[Corollary 8.4]{XZ:excision}}]\label{AHI-braid-detection}
Let $L$ be a link in $S^1\times D^2$.
Then $L$ is isotopic to the closure of 
a braid with $n$ strands
if and only if the top f-grading of $\AHI(L)$ is $n$, and $\AHI(L,n)\cong \bC.$
\end{Proposition}

Although the annular instanton Floer homology is defined for links in the solid torus, it can be used to study links in $S^3$.
Let $L$ be a link in $S^3$ and let $p$ be a base point on $L$. In \cite{KM:Kh-unknot}, Kronheimer and Mrowka defined the link invariant
$$
\II^\natural(L,p):=\II(S^3,L\cup m,u),
$$
where $m$ is a small meridian circle of $L$ around $p$ and $u$ is an arc joining $m$ and $p$. 
The following result is a consequence of the excision property of instanton Floer homology.
\begin{Proposition}[{\cite[Section 4.3]{AHI}}]\label{Prop-AHI=I^natural}
Suppose $L$ has an unknotted component $U$ and let $p\in U$. Let $N(U)$ be a tubular neighborhood of $U$,
then $L_0:=L-U$ is a link in the solid torus $S^3-N(U)$. We have
\begin{equation}
\AHI(L_0)\cong \II^\natural(L,p).
\end{equation}
\end{Proposition}
The above isomorphism does not preserve the f-grading of $\AHI(L_0)$ since there is no such grading on $\II^\natural(L,p)$. 
Notice that a meridional surface in the solid torus $S^3-N(U)$ is a Seifert surface of $U$. 

\section{Local coefficients}
\label{sec_local_coef}

This section reviews the singular instanton Floer homology theory with local coefficients, which was introduced in \cite[Section 3.9]{KM:YAFT} (see also \cite[Section 3]{KM-Ras}). Let $\mathcal{B}(Y,L,\omega)$ be the space of gauge-equivalence classes of orbifold connections over $(Y,L,\omega)$.
Let $\mathcal{R}$ be the ring
$$
\mathcal{R}:=\mathbb{C}[t,t^{-1}],
$$
and suppose 
$$
\mu : \mathcal{B}(Y,L,\omega) \to \mathbb{R}/\mathbb{Z}
$$
is a continuous function. For each $a\in \mathcal{B}(Y,L,\omega)$, define a rank-$1$ free $\mathcal{R}$-module by the formal multiplication
$$
\Gamma^\mu_a:=t^{\tilde\mu(a)}\cdot\mathcal{R},
$$
where $\tilde\mu(a)$ is a lift of $\mu(a)$ in $\bR$.
Let $\text{Crit}(CS)$ be the set of critical points of the (perturbed) Chern-Simons functional $CS$, 
define a free $\mathcal{R}$-module $\mathbf{C}^\mu$ by
$$
\mathbf{C}^\mu:=\bigoplus_{\alpha\in \text{Crit}(CS)} \Gamma^\mu_\alpha.
$$
To make $\mathbf{C}^\mu$ a chain complex, we need to define a differential on it. 
For $\alpha,\beta\in \text{Crit}(CS)$, let $M_d(\alpha,\beta)$ be the $d$-dimensional moduli space of trajectories of $CS$ 
from $\alpha$ to $\beta$.  This space carries an $\mathbb{R}$-action and we denote the quotient space by
$$
\breve{M}_d(\alpha,\beta):=M_d(\alpha,\beta)/\mathbb{R}.
$$
A trajectory $z\in M_d(\alpha,\beta)$ determines
a path $p_z$ in $\mathcal{B}(Y,L,\omega)$ from $\alpha$ to $\beta$. The map
$$
\mu\circ p_z: [0,1]\to \mathbb{R}/\mathbb{Z}
$$
can be lifted to a map
$$
\widetilde{\mu\circ p_z}: [0,1]\to \mathbb{R}.
$$
Although $\widetilde{\mu\circ p_z}$ is not unique, the difference 
$$
\nu(z):=\widetilde{\mu\circ p_z}(1)-\widetilde{\mu\circ p_z}(0)
$$ is well-defined. 
 We define an $\mathcal{R}$-module homomorphism by
\begin{align*}
d^\alpha_\beta :\Gamma^\mu_\alpha&\to \Gamma^\mu_\beta \\
     t^s &\mapsto \sum_{[z]\in \breve{M}_1(\alpha,\beta)} \text{sign(z)}\cdot t^{s+\nu(z)} 
\end{align*}
The differential $D$ on $\mathbf{C}^\mu$ is then given by
$$
D:=\bigoplus_{\alpha, \beta\in \text{Crit}(CS) } d^\alpha_\beta,
$$ 
and the instanton Floer homology with local coefficients is defined by
\begin{equation}
\II(Y,L,\omega;\Gamma^\mu):=H^\ast(\mathbf{C}^\mu,D).
\end{equation}
If $F\subset Y$ is an embedded closed surface, 
the operator $\muu(F)$ can be defined in the setting with local coefficients. 
Roughly speaking, 
the surface $F$ defines a two dimensional cohomology class on $\mathcal{B}(Y,L,\omega)$ (see \cite[Section 2]{XZ:excision}), and its Poincar\'e dual is given by a linear combination of divisors on $\mathcal{B}(Y,L,\omega)$ as
$$
\sum a_i V_i, ~a_i\in \mathbb{Q}.
$$
There is a map from $M_d(\alpha,\beta)$ to $\mathcal{B}(Y,L,\omega)$ by restricting the trajectories at time $0$.
The divisors $V_i$'s are generic in the sense that they are transverse to 
the restriction map $M_d(\alpha,\beta)\to \mathcal{B}(Y,L,\omega) $ for all $\alpha,\beta\in \text{Crit}(CS)$ and $d\in \mathbb{N}$.
 We define an $\mathcal{R}$-module homomorphism by
\begin{align*}
f^\alpha_\beta : \Gamma^\mu_\alpha&\to \Gamma^\mu_\beta \\
  t^s &\mapsto \sum_i a_i\sum_{z\in M_2(\alpha,\beta)\cap V_i} \text{sign}(z) \cdot  t^{s+\nu(z)} 
\end{align*}
A standard argument shows that the map
$$
H:=\bigoplus_{\alpha, \beta\in \text{Crit}(CS) }f^\alpha_\beta :\mathbf{C}^\mu \to \mathbf{C}^\mu
$$
is a chain map. The map $H$ induces the operator $\muu(F)$ on  $\II(Y,L,\omega;\Gamma^\mu)$.

The tensor products
$$
\mathbf{C}^\mu \otimes_{\mathcal{R}} \mathcal{R}/(t-1), ~ D\otimes_{\mathcal{R}} \mathcal{R}/(t-1), ~ H\otimes_{\mathcal{R}} \mathcal{R}/(t-1)
$$
recover the ordinary Floer chain complex $(C,d)$ and the ordinary operator $\muu(F)$ on $\II(Y,L,\omega)$ with $\bC$-coefficients. 

Suppose there is a component $K\subset L$ such that $K\cap \omega=\emptyset$. Fix an orientation and a framing of $K$, we can define a continuous 
map
$$
\mu_K:\mathcal{B}\to U(1)=\mathbb{R}/\mathbb{Z}
$$
by taking the limit holonomy of the orbifold connections along the longitude of $K$. The map $\mu_K$ then gives a local system.
 The local systems defined by different framings of $K$ are isomorphic via multiplications by powers of $t$, therefore the choice of the framing is not important.
 More generally, suppose there is a sublink $L'=K_1\cup\cdots\cup K_l$ of $L$ such that $\omega\cap L'=\emptyset$, we can choose a framing for each $K_j$ and define the map $\mu_{K_j}$ as above, hence 
we obtain a local system $\Gamma$ associated with $L'$ defined by
$$
\mu_{L'}:=\mu_{K_1}\mu_{K_2}\cdots\mu_{K_l}.
$$
If $L'$ is the empty link, then $\mu_{L'}=1$, thus the local system $\Gamma$ is the trivial system with coefficient $\mathcal{R}$. In this case, we have
\begin{equation}\label{eq-constant-coefficients}
\II(Y,L,\omega;\Gamma)=\II(Y,L,\omega)\otimes_{\mathbb{C}}\mathcal{R}.
\end{equation}

Suppose $(Y_0,L_0,\omega_0)$ and $(Y_1,L_1,\omega_1)$ are two admissible triples with local systems $\Gamma_0$ and $\Gamma_1$ 
associated with oriented sublinks $L_0'\subset L_0$ and $L_1'\subset L_1$ respectively. Let 
$$
(W,S,\eta)=(W,S_0\sqcup S_1,\eta):(Y_0,L_0,\omega_0)\to (Y_1,L_1,\omega_1)
$$ 
be a cobordism such that $\partial S_0=L_0'\cup L_1'$ and $\eta\cap S_0=\emptyset$. 
Then $(W,S,\eta)$ induces a map
$$
\II(W,S,\eta): \II(Y_0,L_0,\omega_0;\Gamma_0) \to \II(Y_1,L_1,\omega_1;\Gamma_1).
$$
This makes the instanton Floer homology with local coefficients a functor.
By the definition of cobordism of triples, $S$ and $\eta$ are required to be embedded surfaces in $W$. We can also consider the situation where
$S$ is an immersed surface with transverse double points, as discussed in \cite[Section 5]{KM:YAFT} and \cite{Kr-ob}. In this situation, 
one can blow up $W$ at the self-intersection points of $S$ to  resolve the double points and 
obtain an ordinary cobordism  $(\widetilde{W}, \widetilde{S},\eta)$, and then define
$$
\II(W,S,\eta):= \II (\widetilde{W}, \widetilde{S},\eta).
$$ 

\begin{figure}  
\centering
  \includegraphics[width=0.8\linewidth]{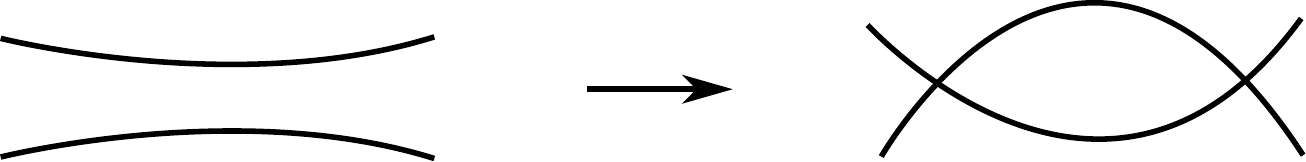}
  \caption{A schematic picture of the finger move.}
  \label{finger_move} 
\end{figure}

Now suppose $S=S_0\sqcup S_1$ and $S'=S_0'\sqcup S_1'$ are two immersed surfaces with transverse double points in $W$ such that $\eta\cap S=\eta\cap S'=\emptyset$, $\partial S = \partial S' = L_0\cup L_1$, and $\partial S_0=\partial S_0'=L_0'\cup L_1'$.
We consider the following 5 situations:
\begin{itemize}
\item[(i)]   $S'$ is obtained from $S$ by an ambient isotopy;

\item[(ii)]  $S_1'=S_1$, and $S_0'$ is obtained from $S_0$ by a twist move introducing a positive double point (see \cite[Section 1.3] {freedman1990Topology} for the definition of twist move);

\item[(iii)] $S_1'=S_1$, and $S_0'$ is obtained from $S_0$ by a twist move introducing a negative double point;

\item[(iv)]  $S_1'=S_1$, and $S_0'$ is obtained from $S_0$ by a finger move introducing two double points of opposite signs 
(see Figure \ref{finger_move} for a schematic picture and see \cite[Section 1.5] {freedman1990Topology} for the precise definition of finger move);

\item[(v)]  $S'$ is obtained from $S$ by a finger move introducing two double points of opposite signs in $S'_0\cap S'_1$.
\end{itemize}

\begin{Proposition}\label{Prop-moves-inst}
Let $(Y_0,L_0,\omega_0), (Y_1,L_1,\omega_1), W,S,S',\eta, \Gamma_0,\Gamma_1$ be as above, let
$$
\II(W,S,\eta): \II(Y_0,L_0,\omega_0;\Gamma_0) \to \II(Y_1,L_1,\omega_1;\Gamma_1)
$$
$$
\II(W,S',\eta): \II(Y_0,L_0,\omega_0;\Gamma_0) \to \II(Y_1,L_1,\omega_1;\Gamma_1)
$$
be the induced cobordism maps.
For the 5 cases listed above, the following equations hold respectively:
\begin{itemize}
\item[(i)]  $\II(W,S',\eta)= \II(W,S,\eta)$;

\item[(ii)]  $\II(W,S',\eta)= (1-t^2)\II(W,S,\eta)$;

\item[(iii)] $\II(W,S',\eta)= \II(W,S,\eta)$;

\item[(iv)]  $\II(W,S',\eta)=(1-t^2) \II(W,S,\eta)$;

\item[(v)]   $\II(W,S',\eta)=\theta(t) \II(W,S,\eta)$ for a universal non-zero polynomial $\theta(t)\in \mathcal{R}$.
\end{itemize}
\end{Proposition}
\begin{proof}
Part (i) is trivial. (ii), (iii) and (iv) are from \cite[Proposition 3.1]{KM-Ras}. The proof of (v) is similar to (iv). We first
review Kronheimer and Mrowka's proof of (iv) briefly. For simplicity, consider a special case that $(W,S,\eta)$ 
is closed, thus it can be viewed as a cobordism from the empty set to the empty set. We also assume $S_1=S_1'=\emptyset$.
In this case, 
$$ \II(W,S,\eta)\in \Hom_{\mathcal{R}}(\mathcal{R},\mathcal{R})=\mathcal{R}$$ is the singular Donaldson invariants (from 0-dimensional moduli spaces) 
introduced by Kronheimer and Mrowka in 
\cite{KM:embedded-surfaces-II}. More precisely, we have
\begin{equation}\label{eq_singular_Donaldson_polynomial}
\II(W,S,\eta)=\sum_{k,l} q_{k,l}(W,S,\eta) t^{-l},
\end{equation}
where $q_{k,l}$ denotes the singular Donaldson invariant defined by
counting the number of points in the 0-dimensional moduli spaces over all orbifold bundles
 with instanton number $k$ and monopole number $l$. The restriction 
of an orbifold $\SO(3)$-bundle to $S$
has a reduction to $K\oplus \underline{\bR}$ where $K$ is an $\SO(2)$-bundle. By definition, the monopole number $l$ is given by
$$
l=-\frac12 e(K)[S].
$$
If $(W',S')$ is obtained from $(W,S)$ by a finger move, \cite[Proposition 3.1]{Kr-ob} used a gluing argument to prove that 
\begin{equation}\label{eq_qkl_finger_move}
q_{k,l}(W',S',\eta)= q_{k,l}(W,S,\eta)-q_{k-1,l+2}(W,S,\eta).
\end{equation}
 When $(W,S,\eta)$ is closed and $S_1=S_1'=\emptyset$, Part (iv) follows immediately from \eqref{eq_singular_Donaldson_polynomial} and \eqref{eq_qkl_finger_move}. 
Since the gluing argument only depends on the local structure of the finger move, it is straightforward to extend the argument to the general (relative) case. 

To prove Part (v), we first assume $(W,S,\eta)$ is closed and $S_1\neq \emptyset$. We give a refined definition of the monopole number $l$ by taking
$$
l_0:=-\frac12 e(K)[S_0], ~l_1:=-\frac12 e(K)[S_1].
$$
It is clear from the definition that $l=l_0+l_1$. With this definition, we have a refined singular Donaldson invariant
$q_{k,l_0,l_1}$. 
Similar to \eqref{eq_singular_Donaldson_polynomial}, we define the polynomial 
\begin{equation*}
Q(W,S,\eta)(t_0,t_1):=\sum_{k,l_0,l_1} q_{k,l_0,l_1}(W,S,\eta) t_0^{-l_0}t_1^{-l_1}.
\end{equation*}
If $S'$ is obtained by a finger move that introduces intersection points between $S_0$ and $S_1$, then the proof of \cite[Equation (23)]{Kr-ob}
shows that there exist universal constants $a_{i,j}\in \bZ$ such that
\begin{equation*}
 q_{k,l_0,l_1}(W',S',\eta)=\sum_{2|i+j}a_{i,j} q_{k-\frac{i+j}{2},l_0+i,l_1+j}(W,S,\eta).
\end{equation*}
By the Uhlenbeck compactness theorem, only finitely many $a_{i,j}$'s are non-zero. 
This implies that
\begin{equation}\label{eq_finger_move_Qt0t1}
Q(W',S',\eta)(t_0,t_1)=P(t_0,t_1)Q(W,S,\eta)(t_0,t_1)
\end{equation}
for a universal polynomial $P(t_0,t_1)\in \bC[t_0,t_0^{-1},t_1,t_1^{-1}]$. Notice that 
$$
q_{k,l}=\sum_{l_0+l_1=l}q_{k,l_0,l_1},
$$
therefore \eqref{eq_qkl_finger_move} implies 
\begin{equation}\label{eqn_P(t,t)}
P(t,t)=(1-t^2).
\end{equation}
We also have $P(t_0,t_1)=P(t_1,t_0)$ because there is no difference between the roles of $S_0$ and $S_1$ in the finger move.

We claim that
\begin{equation}\label{eqn_P(1,t)neq_0}
P(t,1)=P(1,t)\neq 0.
\end{equation}
In fact, suppose the contrary, then 
$$
(t_0-1)| P(t_0,t_1),~(t_1-1)|P(t_0,t_1),
$$
thus we have
$$
(t_0-1)(t_1-1)|P(t_0,t_1),
$$
therefore
$$
(t-1)(t-1)|P(t,t),
$$
which contradicts \eqref{eqn_P(t,t)}, hence the claim is proved. 

Now let $\theta(t):=P(t,1)$,
we have
$$
\II(W,S,\eta)(t)=Q(W,S,\eta)(t,1),
$$
therefore in the closed case, Part (v) of the proposition follows from \eqref{eq_finger_move_Qt0t1} and \eqref{eqn_P(1,t)neq_0}. 
By the gluing argument, the same result holds for the non-closed case.
\end{proof}

Suppose $(Y,L_0,\omega)$ is an admissible triple, and let $L_0'$ be a sublink of $L_0$ such that $L_0'\cap \omega = \emptyset$. Fix an orientation of $L_0'$. By the previous discussion, $L_0'$ defines a local system $\Gamma_0$ with coefficient $\mathcal{R}$.
Suppose $L_1$ is obtained from $L_0$ by a local crossing change in $Y-\omega$, where the crossing is either within $L_0'$
or between $L_0'$ and $L_0-L_0'$, and let $\Gamma_1$ be the local system of $(Y,L_1,\omega)$ associated with the image of $L_0'$ after the crossing change. 

The crossing change induces an immersed cobordism 
$S:L_0\to L_1$,
where $S$ is an immersed surface in $ [0,1]\times Y$ with one double point.
Reversing $S$, we obtain an immersed cobordism $\overline{S}:L_1\to L_0$ with one double point. The composition 
$S\cup \overline{S}\subset [0,2]\times Y$ can be obtained from the product cobordism $[0,2]\times L_0$ by a finger move decribed by case (iv) or
case (v) of Proposition \ref{Prop-moves-inst}. Therefore by Proposition \ref{Prop-moves-inst}, the map
\begin{equation}\label{Equation-L0-L0}
\II([0,2]\times Y, S\cup \overline{S}, [0,2]\times\omega):\II(Y,L_0,\omega_0;\Gamma_0)\to \II(Y,L_0,\omega_0;\Gamma_0)
\end{equation}
is equal to $(1-t^2)\id$ or $\theta(t)\id$. Similarly, the map
\begin{equation}\label{Equation-L1-L1}
\II([0,2]\times Y, \overline{S}\cup {S}, [0,2]\times\omega):\II(Y,L_1,\omega;\Gamma_1)\to \II(Y,L_1,\omega;\Gamma_1)
\end{equation}
is equal to $(1-t^2)\id$ or $\theta(t)\id$. As a consequence, we have the following result.
\begin{Proposition}\label{Prop-homotopy-inst}
Suppose $(Y,L_0,\omega)$ is an admissible triple and $L_0'\subset L_0$ is a sublink with $L_0\cap \omega=\emptyset$.
Fix an orientation on $L_0'$ and let $\Gamma_0$ be the local system of $(Y,L_0,\omega)$ defined by $L_0'$. Suppose $L_1'$ is an oriented link that is homotopic to $L_0'$ in $Y-\omega$ and is disjoint from $L_0-L_0'$. Let $L_1 := L_1' \cup  (L_0-L_0')$. Let $\Gamma_1$ be the local system of $(Y,L_1,\omega)$ defined by $L_1'$. Then we have
$$
T^{-1}\II(Y,L_0,\omega;\Gamma_0)\cong T^{-1}\II(Y,L_1,\omega;\Gamma_1),
$$
where $T$ is the multiplicative system generated by $(1-t^2)\theta(t)$.
\end{Proposition}
\begin{proof}
Since $L_0'$ is homotopic to $L_1'$ in $Y-\omega$, the link $L_1$  can be obtained from $L_0$ by a finite sequence of crossing changes in $Y-\omega$, such that the sublink $L_0-L_0'$ remains fixed.
 Without loss of generality, we may assume that
$L_1$ is obtained from $L_0$ by one such crossing change. Let $S\subset Y\times [0,1]$ be the immersed cobordism from $L_0$ to $L_1$ given by the crossing change, and let $\overline S$ be the reverse of $S$.
By \eqref{Equation-L0-L0}, \eqref{Equation-L1-L1}, and the functoriality of the instanton Floer homology with local coefficients,
we have
$$
\II(Y\times [0,1],\overline{S}, \omega\times [0,1]) \circ \II(Y\times [0,1],S,\omega\times [0,1])= (1-t^2)\id ~\text{or}~ \theta(t)\id
$$
on $\II(Y,L_0,\omega;\Gamma_0)$,
and
$$
\II(Y\times [0,1],{S}, \omega\times [0,1]) \circ \II(Y\times [0,1],\overline{S},\omega\times [0,1])= (1-t^2)\id ~\text{or}~ \theta(t)\id.
$$
on $\II(Y,L_1,\omega;\Gamma_1)$. Therefore $T^{-1}\II(Y,L_0,\omega;\Gamma_0)$ and
$T^{-1}\II(Y,L_1,\omega;\Gamma_1)$ are isomorphic. 
\end{proof}

\begin{Corollary}\label{LocallyFreePart-inst}
Let $Y,\omega,L_0,L_1,\Gamma_0,\Gamma_1$ be as in Proposition \ref{Prop-homotopy-inst}, we have
$$
\rank_{\mathcal{R}}\II(Y,L_0,\omega;\Gamma_0)=\rank_{\mathcal{R}}\II(Y,L_1,\omega;\Gamma_1).\makeatletter\displaymath@qed
$$
\end{Corollary}

Given an oriented link $L$ in $S^1\times D^2$, we define the annular instanton Floer homology with local coefficients by
$$
\AHI(L;\Gamma):=\II(S^1\times S^2,L\cup \mathcal{K}_2,u;\Gamma),
$$
where $\Gamma$ is the local system associated with $L$. The operator $\muu(S^2)$ on $\AHI(L;\Gamma)$ is now an $\mathcal{R}$-module
homomorphism instead of a $\bC$-linear map, therefore $\AHI(L;\Gamma)$ no longer carries the f-grading. 
The torus excision theorem (\cite[Theorem 5.6]{KM:Kh-unknot}) still holds for instanton Floer homology with local coefficients, as long as the exicion surface is disjoint from the sublink defining the local system. Therefore,
Proposition \ref{product-formula} still holds for the annular instanton Floer homology with local coefficients, except that there is no f-grading anymore. 

\begin{eg}\label{eg-Uk-Kl-Gamma}
By Example \ref{eg-Uk-Kl}, the critical points of the (perturbed) Chern-Simons functional for 
$\AHI(\mathcal{U}_1)$ (or $\AHI(\mathcal{K}_1)$) consist of two points whose homological degrees differ by $2$. Therefore there are no differentials in the Floer chain complex, and we
have
$$
\AHI(\mathcal{U}_1;\Gamma)\cong \mathcal{R}\oplus \mathcal{R}, ~ \AHI(\mathcal{K}_1;\Gamma)\cong \mathcal{R}\oplus \mathcal{R}.
$$
By Proposition \ref{product-formula}, we have
$$
\AHI(\mathcal{U}_k\cup\mathcal{K}_l;\Gamma)\cong \mathcal{R}^{2^{k+l}}.
$$
\end{eg}

Proposition \ref{Prop-AHI=I^natural} follows from the torus excision theorem, therefore it also works in the case with local coefficients.  
Let $L\subset S^1\times D^2$ be a link with $n$ components and view $S^1\times D^2$ as the complement of a neighborhood of the unknot $U$ in $S^3$, let $p\in U$ and let $\Gamma_L$ be the local system associated with $L$, we have
$$
\AHI(L;\Gamma)\cong \II^\natural (L\cup U,p;\Gamma_L).
$$
Suppose the annular link $L$ has $n$ components, then the embedded image of $L$ in $S^3$ is homotopic to the embedded  image of $\mathcal{U}_n$ in $S^3$.
By Corollary \ref{LocallyFreePart-inst}, we have
$$
\rank_{\mathcal{R}}\II^\natural (L\cup U,p;\Gamma_L)=\rank_{\mathcal{R}}\II^\natural (\mathcal{U}_n\cup U,p;\Gamma_{\mathcal{U}_n}).
$$
By Proposition \ref{Prop-AHI=I^natural} again, we have
$$
\II^\natural (U_n\cup U,p;\Gamma)\cong \AHI(\mathcal{U}_n;\Gamma)\cong \mathcal{R}^{ 2^n}.
$$
In conclusion, we obtain
\begin{equation}
\rank_{\mathcal{R}}\AHI(L;\Gamma)=2^n.
\end{equation}
By the universal coefficient theorem, we have
\begin{equation}\label{eq-rank_lower_bound}
\rank_{\mathbb{C}}\II^\natural (L\cup U,p)  = \rank_{\mathbb{C}}\AHI(L)\ge \rank_{\mathcal{R}}\AHI(L;\Gamma)=2^n.
\end{equation}

\section{Limits of chain complexes}
This section discusses a simple observation from linear algebra and its consequences in instanton Floer homology.

Suppose $\{C_n\}_{n\in\bZ}$ is a sequence of finite dimensional complex vector spaces. For each $n\in \bZ$ and $k\ge 0$, let $\partial_n^{(k)}$ be a $\bC$-linear map from $C_n$ to $C_{n-1}$, let $f_n^{(k)}$ be an endomorphism of $C_n$. Suppose for each pair $(n,k)$, we have $\partial_n^{(k)}\circ \partial_{n+1}^{(k)} = 0$ and $\partial_n^{(k)}\circ f_n^{(k)} = f_{n-1}^{(k)} \circ \partial_{n}^{(k)}$. Moreover, suppose for each $n$, the limits 
$$\lim_{k\to\infty}f_n^{(k)}\text{ and }\lim_{k\to\infty} \partial_n^{(k)}$$
are convergent. Let 
\begin{align*}
\partial_n:= \lim_{k\to\infty} \partial_n^{(k)}, 
\qquad & 
f_n:= \lim_{k\to\infty} f_n^{(k)},
\\
H_n^{(k)}:=\ker \partial_n^{(k)} / \Imm\partial_{n+1}^{(k)}, 
\qquad & 
H_n:=\ker \partial_n/\Imm\partial_{n+1}. 
\end{align*}
The maps $f_n^{(k)}$ and $f_n$ induce maps on $H_n^{(k)}$ and $H_n$ respectively. For $\Lambda\subset\bC$, define $E_{n,\Lambda}^{(k)}\subset H_{n}^{(k)}$ to be the direct sum of the generalized eigenspaces of $f_n^{(k)}$ with eigenvalues in $\Lambda$. Similarly, define $E_{n,\Lambda}\subset H_{n}$ to be the direct sum of the generalized eigenspaces of $f_n$ with eigenvalues in $\Lambda$. 
\begin{Lemma}\label{chain-complex-limit}
Let $C_n$, $\partial_n^{(k)}$, $\partial_n$, $f_n^{(k)}$, $f_n$, $H_n^{(k)}$, $H_n$, $E_{n,\Lambda}^{(k)}$, and $E_{n,\Lambda}$  be as above.
\begin{enumerate}
\item If $\Lambda\subset \bC$ is a closed subset, then
$$
\dim E_{n,\Lambda} \ge \limsup_{k\to\infty} \dim E_{n,\Lambda}^{(k)}.
$$
\item For $\epsilon>0$, let $N(\Lambda,\epsilon)$ be the closed $\epsilon$-neighborhood of $\Lambda$. If $\Lambda$ is closed, then
$$
\dim E_{n,\Lambda} \ge \lim_{\epsilon\to 0}\limsup_{k\to\infty} \dim E_{n,N(\Lambda,\epsilon)}^{(k)}.
$$
\item
If we further assume that $\dim H_n^{(k)} = \dim H_n$ for all $k$, then 
$$
\dim E_{n,\Lambda} = \lim_{\epsilon\to 0}\limsup_{k\to\infty} \dim E_{n,N(\Lambda,\epsilon)}^{(k)}.
$$
\end{enumerate}

\end{Lemma}

\begin{proof}
(1)
Let $Z_n:= \ker\partial_n^{(k)}$, $B_{n}^{(k)} := \Imm \partial_{n+1}^{(k)}$, $Z_n := \ker \partial_n$,  $B_n:=\Imm \partial_{n+1}$. 
Suppose there exists a closed set $\Lambda\subset \bC$ such that the statement of Part (1) does not hold.
After taking a subsequence, we may assume that the dimensions of $Z_n^{(k)}$ and $B_n^{(k)}$ are independent of $k$, and that they are convergent in the corresponding Grassmannians as $k\to\infty$.  
The spectrum of $f_n$ (with multiplicities) on 
\begin{equation}
\label{eqn_quotient_of_limits}
(\lim_{k\to\infty} Z_{n}^{(k)})/(\lim_{k\to\infty} B_{n}^{(k)})
\end{equation}
is the limit of the spectra of $f_n^{(k)}$ on $Z_{n}^{(k)}/B_{n}^{(k)}$ as $k\to \infty$. Let 
$$E_{n,\Lambda}'\subset (\lim_{k\to\infty} Z_{n}^{(k)})/(\lim_{k\to\infty} B_{n}^{(k)})$$
  be the direct sum of the generalized eigenspaces of $f_n^{(k)}$ for the eigenvalues in $\Lambda$. Since $\Lambda$ is closed, the previous argument implies
  $$
  \dim E_{n,\Lambda}' \ge \limsup_{k\to\infty} \dim E_{n,\Lambda}^{(k)}.
  $$  
  On the other hand, we have
\begin{align*}
Z_{n} &\supset \lim_{k\to\infty} Z_{n}^{(k)},
\\
B_{n} &\subset \lim_{k\to\infty} B_{n}^{(k)},
\end{align*}
therefore \eqref{eqn_quotient_of_limits} is a sub-quotient of $H_n = Z_n/B_n$, hence 
$$\dim E_{n,\Lambda} \ge \dim E_{n,\Lambda}' \ge \limsup_{k\to\infty} \dim E_{n,\Lambda}^{(k)},$$
 contradicting the assumption.

(2) Let $\epsilon_0$ be sufficiently small such that
$E_{n,N(\Lambda,\epsilon_0)} = E_{n,\Lambda}$. By Part (1), we have
\begin{align*}
& \dim E_{n,\Lambda} = \dim E_{n,N(\Lambda,\epsilon_0)} 
\\
\ge 
&  \limsup_{k\to\infty} \dim E_{n,N(\Lambda,\epsilon_0)}^{(k)}
\ge \lim_{\epsilon\to 0} \limsup_{k\to\infty} \dim E_{n,N(\Lambda,\epsilon)}^{(k)}.
\end{align*}

(3) Let $\epsilon_0$ be sufficiently small such that
$E_{n,N(\Lambda,\epsilon_0)} = E_{n,\Lambda}$. Suppose $\epsilon<\epsilon_0$, then $E_{n,N(\Lambda,\epsilon)} = E_{n,\Lambda}$. Let $\Lambda_1 := \overline{\bC-N(\Lambda,\epsilon)}$. By the condition on $\epsilon$, 
$$
E_{n,\partial N(\Lambda,\epsilon)} = E_{n,\partial \Lambda_1} = \{0\}.
$$
Hence by Part (1), for $k$ sufficiently large we have
$$
\dim E_{n,\partial N(\Lambda,\epsilon)}^{(k)}=\dim E_{n,\partial \Lambda_1}^{(k)}=0.
$$ 
Apply Part (1) again on $N(\Lambda,\epsilon)$ and $\Lambda_1$, we deduce that if $k$ is sufficiently large, then $\dim E_{n, N(\Lambda,\epsilon)} \ge \dim E_{n, N(\Lambda,\epsilon)}^{(k)}$, $\dim E_{n,\Lambda_1} \ge \dim E_{n,\Lambda_1}^{(k)}$. Therefore
\begin{align*}
\dim H_n & = \dim E_{n, N(\Lambda,\epsilon)} + \dim E_{n,\Lambda_1} 
\\
& \ge 
\dim E_{n, N(\Lambda,\epsilon)}^{(k)} + \dim E_{n,\Lambda_1}^{(k)}
\\
& = \dim H_n^{(k)}= \dim H_n.
\end{align*}
As a consequence, for $k$ sufficiently large,
$\dim E_{n, N(\Lambda,\epsilon)} = \dim E_{n, N(\Lambda,\epsilon)}^{(k)}$, hence 
$$
\dim E_{n,\Lambda} = \dim E_{n, N(\Lambda,\epsilon)} = \limsup_{k\to\infty} \dim E_{n,N(\Lambda,\epsilon)}^{(k)}.
$$
Since the above equation holds for all $\epsilon<\epsilon_0$, Part (3) of the lemma is proved.
\end{proof}

Recall that given an admissible triple $(Y,L,\omega)$ and a continuous function $\mu:\mathcal{B}(Y,L,\omega)\to \bR/\bZ$, there is a local system $\Gamma^\mu$ on $\mathcal{B}(Y,L,\omega)$ defined by $\mu$. The Floer chain complex 
$\mathbf{C}^\mu$ is a finitely generated free $\mathcal{R}$-module, where $\mathcal{R} = \bC[t,t^{-1}]$. The differential $D$ is an $\mathcal{R}$-endomorphism of $\mathbf{C}^\mu$.
For $h\in\bC-\{0\}$, define 
$$
(C_h,d_h):=(\mathbf{C}^\mu\otimes_\mathcal{R} \mathcal{R}/(t-h), D\otimes_\mathcal{R} \id_{\mathcal{R}/(t-h)} ).
$$
Notice that $\mathcal{R}/(t-h) \cong \bC$ via the map $t\mapsto h$, hence $(C_h,d_h)$ is a finite dimensional chain complex over $\bC$. Let $C:=\bC^{\rank_\mathcal{R}\mathbf{C}^\mu}$, we identify $C_h$ with $C$ using the above isomorphism. The differentials $d_h$ become a continuous family of linear maps on $C$.
Given an embedded surface $F\subset Y$,
define
$$
\muu(F)_h:=\muu(F)\otimes_{\mathcal{R}} \mathcal{R}/(t-h),
$$
then $\muu(F)_h$ is continuous with respect to $h$ and is a chain map on $(C,d_h)$. 
Therefore, for each $h\in \mathbb{C}-\{0\}$, the map $\muu(F)_h$ induces a map on the Floer homology 
$$
\II(Y,L,\omega;\Gamma^\mu\otimes_{\mathcal{R}}\mathcal{R}/(t-h))=H^\ast(C,d_h).
$$
To simplify notations, we will use $\Gamma^\mu(h)$ to denote $\Gamma^\mu\otimes_{\mathcal{R}}\mathcal{R}/(t-h)$ for the rest of this article. 
If $h=1$, then $\II(Y,L,\omega;\Gamma^\mu(1))$ is the ordinary instanton Floer homology without local coefficients, and $\muu(F)_1$ coincides with the ordinary $\mu$ map.

\begin{Proposition}\label{Prop-Gammah-homotopy_invariance}
Let $Y,\omega,L_0,L_1,L_0',L_1',\Gamma_0,\Gamma_1$ be as in Proposition \ref{Prop-homotopy-inst}. Let $\theta(t)$ be the polynomial given by Part (v) of Proposition \ref{Prop-moves-inst}. Suppose $h\in \bC-\{0\}$ satisfies 
\begin{equation}\label{eq_root_O}
(1-h^2)\theta(h)\neq 0,
\end{equation}
then we have
\begin{equation}\label{eqn_Gammah-homotopy_invariance}
\II(Y,L_0,\omega;\Gamma_0(h))\cong \II(Y,L_1,\omega;\Gamma_1(h)).
\end{equation}
Moreover, if $F\subset Y$ is a closed embedded surface in $Y$, then the isomorphism \eqref{eqn_Gammah-homotopy_invariance} intertwines with $\muu(F)_h$.
\end{Proposition}
\begin{proof}
Let $T\subset\mathcal{R}$ be the multiplicative system generated by $(1-t^2)\theta(t)$ as in Proposition \ref{Prop-homotopy-inst}.
By \eqref{eq_root_O}, the elements of $T$ have non-zero images in $\mathcal{R}/(t-h)\cong \bC$, hence $\mathcal{R}/(t-h)$ is isomorphic to $(T^{-1}\mathcal{R})/(t-h)$. Therefore, for $i=0,1$, we have
\begin{equation}\label{eqn_change_of_coef_by_tensor_prod}
\II(Y,L_i,\omega;\Gamma_i(h))\cong\II(Y,L_i,\omega;\Gamma_i\otimes_\mathcal{R} T^{-1}\mathcal{R}\otimes_{T^{-1}\mathcal{R}} T^{-1}\mathcal{R}/(t-h)  ).
\end{equation}
On the other hand, since localization is an exact functor, we have
$$
\II(Y,L_i,\omega;\Gamma_i\otimes_\mathcal{R} T^{-1}\mathcal{R})\cong T^{-1}\II(Y,L_i,\omega;\Gamma_i).
$$
Therefore by Proposition \ref{Prop-homotopy-inst},
\begin{equation}\label{eqn_inst_Floer_iso_local_coef_tensor_T^-1}
\II(Y,L_0,\omega;\Gamma_0\otimes_\mathcal{R} T^{-1}\mathcal{R})\cong \II(Y,L_1,\omega;\Gamma_1\otimes_\mathcal{R} T^{-1}\mathcal{R}).
\end{equation}
Since $\mathcal{R}$ is a principal ideal domain, the localization $T^{-1}\mathcal{R}$ is also a principal ideal domain, hence \eqref{eqn_Gammah-homotopy_invariance} follows from the universal coefficient theorem and the isomorphisms \eqref{eqn_change_of_coef_by_tensor_prod} and \eqref{eqn_inst_Floer_iso_local_coef_tensor_T^-1}.

It remains to prove that \eqref{eqn_Gammah-homotopy_invariance} intertwines with with $\muu(F)_h$. Since the isomorphism \eqref{eqn_inst_Floer_iso_local_coef_tensor_T^-1} is induced by a cobordism in which the two copies of the surface $F$ on the two ends are homologous, it intertwines the $\muu(F)_h$ on the in-coming end with the
$\muu(F)_h$ on the out-going end, hence the statement is proved.
\end{proof}

 Lemma \ref{chain-complex-limit} and Proposition \ref{Prop-Gammah-homotopy_invariance} have the following application.  

\begin{Proposition}\label{Prop-AHI-rank-inequality}
Suppose $L\subset S^1\times D^2$ is an oriented link such that every component of $L$
has winding number $0$ or $\pm 1$. Assume there are $k$ components with winding number $0$ and $l$ components
with winding number $\pm 1$,  then we have
\begin{equation*}
\dim_\bC\AHI(L,i)\ge \dim_{\bC} \AHI(\mathcal{U}_k\cup\mathcal{K}_l,i)
\end{equation*}
for all $i\in\mathbb{Z}$.
\end{Proposition}
\begin{proof}

For $\lambda\in \bC$, let $N(\lambda,\epsilon)$ be the closed $\epsilon$-neighborhood of $\lambda$ in $\bC$. Given a vector space $V$ over $\bC$, a linear map $f:V\to V$, and a subset $\Lambda\subset \mathbb{C}$, we use
$E(V,f,\Lambda)$ to denote the direct sum of the generalized eigenspaces of $f$ with eigenvalues in $\Lambda$.

Recall that $\AHI(L;\Gamma)$ is defined to be the instanton Floer homology 
$$\II(S^1\times S^2, L\cup \mathcal{K}_2,u;\Gamma),$$
where $\Gamma$ is the local coefficient system associated with $L$. For $h\in\bC-\{0\}$, recall that $\Gamma(h)$ is the local system over $\bC$ given by $\Gamma\otimes_\mathcal{R}\mathcal{R}/(t-h)$.
For every $i\in \bZ$, Part (2) of Lemma \ref{chain-complex-limit} and Proposition \ref{Prop-Gammah-homotopy_invariance} gives
\begin{align*}
\dim\AHI(L,i)&\ge \lim_{\epsilon\to 0}\limsup_{h\to 1} E(\AHI(L;\Gamma(h)), \muu(S^2)_h,N(i,\epsilon)) \\
&= \lim_{\epsilon\to 0}\limsup_{h\to 1} E(\AHI(\mathcal{U}_k\cup\mathcal{K}_l;\Gamma(h)), \muu(S^2)_h,N(i,\epsilon)).
\end{align*}

According to Example \ref{eg-Uk-Kl-Gamma}, $\AHI(\mathcal{U}_k\cup\mathcal{K}_l;\Gamma)$ is a free $\mathcal{R}$-module of rank $2^{k+l}$.
By the universal coefficient theorem, $\dim_\bC \AHI(\mathcal{U}_k\cup\mathcal{K}_l;\Gamma(h))=2^{k+l}$ for all $h\in \bC-\{0\}$. Therefore Part (c) of Proposition \ref{chain-complex-limit} gives 
\begin{equation*}
 \lim_{\epsilon\to 0}\limsup_{h\to 1} E(\AHI(\mathcal{U}_k\cup\mathcal{K}_l;\Gamma(h)), \muu(F)_h,N(i,\epsilon))=
 \dim_\bC \AHI(\mathcal{U}_k\cup\mathcal{K}_l,i),
\end{equation*}
and the proposition is proved.
\end{proof}

\begin{Corollary}\label{Cor-minimal_rank-meridional_surface}
Suppose $L\subset S^1\times D^2$ is an oriented link such that every component of $L$
has winding number $0$ or $\pm 1$. Assume there are $k$ components with winding number $0$ and $l$ components
with winding number $\pm 1$. Moreover, assume 
\begin{equation}\label{eqn_assume_AHI_L_minimal_rank}
\dim_\bC \AHI(L)=2^{k+l}.
\end{equation}
 Then there exists a meridional disk $S$ in $S^1\times D^2$, such that $S$  intersects every component of $L$ with  winding number $\pm 1$ transeversely at one point, and $S$ is disjoint from every  component of $L$ with
winding number 0.
\end{Corollary}
\begin{proof}
By Example \ref{eg-Uk-Kl-Gamma}, $\dim_\bC \AHI(\mathcal{U}_k\cup\mathcal{K}_l)=2^{k+l}$, therefore \eqref{eqn_assume_AHI_L_minimal_rank} and Proposition \ref{Prop-AHI-rank-inequality} imply
\begin{equation*}
\dim_\bC\AHI(L,i)= \dim_\bC \AHI(\mathcal{U}_k\cup\mathcal{K}_l,i)
\end{equation*}
for all $i\in\mathbb{Z}$.
The top f-grading of $\AHI(\mathcal{U}_k\cup\mathcal{K}_l,i)$ is $l$.
By Theorem \ref{Theorem-2g+n}, there exists a meridional surface $S$ with genus $g$, such that $S$ intersects $L$ transversely at $n$ points, and $2g+n=l$.
On the other hand, every component with a non-zero winding number must intersect $S$, therefore we have $g=0$ and $n=l$,
and the surface $S$ is the desired meridional disk.
\end{proof}

\section{Linking numbers and forests of unknots}



This section proves a weaker version of Theorem \ref{The_main_theorem}:
\begin{Theorem}\label{Thm-forest_detection_linking_number}
Suppose $L=K_1\cup\cdots \cup K_n$ is an oriented link with $n$ components in $S^3$ such that
\begin{enumerate}
\item $\rank_{{\mathbb{Z}/2}} \Kh(L;{\mathbb{Z}/2})=2^n$;
\item there exists a forest of unknots $L_G=K_1'\cup\cdots \cup K_n'$ so that
$$
\lk(K_i,K_j)=\lk(K_i',K_j')~\text{for all}~ i\neq j.
$$
\end{enumerate}
Then $L$ is isotopic to $L_G$.
\end{Theorem}
Before starting the proof, we need to make some preparations. Notice that
Batson and Seed's work \cite{Kh-unlink} implies the following useful result. 
\begin{Proposition}[\cite{Kh-unlink}]\label{Prop_Kh_rank_sublink}
Suppose $L$ is a link in $S^3$ with $n$ components and 
$$\rank_{{\mathbb{Z}/2}}\Kh(L;{\mathbb{Z}/2})=2^n,$$ 
then we have
$\rank_{{\mathbb{Z}/2}}\Kh(L_0;{\mathbb{Z}/2})=2^{|L_0|}$
for every sub-link $L_0$ of $L$, where $|L_0|$ is the number of components of $L_0$.
\end{Proposition}
\begin{proof}
Suppose $L=K_1\cup\cdots\cup K_n$.
Let $I$ be a subset of $\{1,\cdots,n\}$ with $|I|$ components.
By \cite[Theorem 1.1]{Kh-unlink} (cf. the proof of \cite[Proposition 7.1]{Kh-unlink}), we have
\begin{align*}
2^n&=\rank_{\mathbb{Z}/2} \Kh(K;{\mathbb{Z}/2})
\\
&\ge \rank_{\mathbb{Z}/2} \Kh(\bigcup_{i\notin I}K_i;{\mathbb{Z}/2})\cdot \rank_{\mathbb{Z}/2} \Kh(\bigcup_{i\in I}K_i;{\mathbb{Z}/2})\\
&\ge 2^{n-|I|}\cdot\rank_{\mathbb{Z}/2}\Kh(\bigcup_{i\in I}K_i;{\mathbb{Z}/2}) \ge 2^n.
\end{align*}
Hence the inequalities above achieve equality, and we have
$$
 \rank_{\mathbb{Z}/2}\Kh(\bigcup_{i\in I}K_i;{\mathbb{Z}/2})=2^{|I|}. \phantom\qedhere\makeatletter\displaymath@qed
$$
\end{proof}
The above result together with Kronheimer-Mrowka's unknot detection theorem in \cite{KM:Kh-unknot} imply the following proposition.
\begin{Proposition}[{\cite[Proposition 7.1]{Kh-unlink}}]\label{Prop-miminal_rank_unknot}
Suppose $L$ is a link in $S^3$ with $n$ components and 
$$\rank_{{\mathbb{Z}/2}}\Kh(L;{\mathbb{Z}/2})=2^n,
$$ 
then each component of $L$
is an unknot. \qed
\end{Proposition}
\begin{Proposition}\label{Prop-rank_I^natural}
Suppose $L$ is a link in $S^3$ with $n$ components and 
$$\rank_{{\mathbb{Z}/2}}\Kh(L;{\mathbb{Z}/2})=2^n,$$ then for every point $p\in L$, we have
$
\dim_\mathbb{C}\II^\natural(L,p)=2^{n-1}.
$
\end{Proposition}
\begin{proof}
Given a point $p\in L$, we use $\Khr(L,p)$ to denote the reduced Khovanov homology with base point $p$. By \cite[Corollary 3.2.C]{Shu:torsion_Kh},
$$
\rank_{\mathbb{Z}/2}\Khr(L,p;{\mathbb{Z}/2})=\frac12 \rank_{{\mathbb{Z}/2}}\Kh(L;{\mathbb{Z}/2})=2^{n-1}.
$$
By the universal coefficient theorem, 
$$
\rank_\mathbb{Q}\Khr(L,p;\mathbb{Q})\le \rank_{\mathbb{Z}/2}\Khr(L,p;{\mathbb{Z}/2})=2^{n-1}
$$
Let $\overline{L}$ be the mirror image of $L$. By \cite[Corollary 11]{Kh-Jones},
$$
\rank_\mathbb{Q}\Khr(\overline{L},p;\mathbb{Q})=\rank_\mathbb{Q}\Khr(L,p;\mathbb{Q})=2^{n-1}
$$
Using Kronheimer-Mrowka's spectral sequence (\cite[Theorem 8.2]{KM:Kh-unknot}) whose $E_2$-page is $\Khr(\overline{L},p;\mathbb{Z})$
and which converges to $\II^\natural(L,p;\mathbb{Z})$, we obtain 
$$
\dim_\mathbb{C}\II^\natural(L,p)=\rank_\mathbb{Z}\II^\natural(L,p;\mathbb{Z})\le \rank_\mathbb{Z}\Khr(\overline{L},p;\mathbb{Z})=2^{n-1}.
$$
On the other hand, Proposition \ref{Prop-miminal_rank_unknot} and \eqref{eq-rank_lower_bound} imply
 that $\dim_\mathbb{C}\II^\natural(L,p)\ge 2^{n-1}$.
Therefore we obtain $\dim_\mathbb{C}\II^\natural(L,p)= 2^{n-1}$.
\end{proof}

\begin{proof}[Proof of Theorem \ref{Thm-forest_detection_linking_number}]
We prove the theorem by induction on $n$. 
When $n=1$, it is the unknot-detection theorem of Kronheimer and Mrowka \cite{KM:Kh-unknot}. 

Assume the theorem holds if the number of components is smaller than $n$. 
Since $G$ is a forest, we can find a vertex of $G$ with degree less than or equal to $1$. We discuss two cases.

Case 1: There is a vertex of $G$ with degree $1$.
Without loss of generality, assume this vertex corresponds to the component $K'_n$ of $L_G$. 
By the assumption of Theorem \ref{Thm-forest_detection_linking_number}, there exists $i\in\{1,\cdots,n-1\}$ such that
$\lk(K_i,K_n)=\pm 1$ and $\lk(K_j,K_n)=0$ when $1\le j \le n-1, j\neq i$. 

Pick a base point $p\in K_n$ and use $L'$ to denote $K_1\cup\cdots \cup K_{n-1}$.
According to Proposition \ref{Prop-miminal_rank_unknot}, $K_n$ is an unknot. Let $N(K_n)$ be a tubular neighborhood of $K_n$, then 
$L'$ can be viewed as a link in the solid torus $S^3-N(K_n)$. By Proposition \ref{Prop-AHI=I^natural} and Proposition \ref{Prop-rank_I^natural}
 we have
$$
\AHI(L')\cong \II^\natural(L,p)\cong \mathbb{C}^{2^{n-1}}
$$
According to Proposition \ref{Cor-minimal_rank-meridional_surface}, we can find a meridional disk $S$ in the solid torus 
$S^3-N(K_n)$ which intersects $K_i$ at a single point and is disjoint from the other components. The meridional disk $S$ is a Seifert disk of $K_n$. 
By the induction hypothesis, $L'$ is a forest of unknots. We can shrink $K_n$
via $S$ to a small meridian of $K_i$. Therefore $L$ is also a forest of unknots. Since the  linking numbers uniquely
determine a forest of unknots, we conclude that $L$ is isotopic to $L_G$.

Case 2: There is a vertex of $G$ with degree $0$. Without loss of generality, assume this vertex corresponds to the component $K'_n$ of $L_G$. 
By the assumption of Theorem \ref{Thm-forest_detection_linking_number}, we have $\lk(K_j,K_n)=0$ for all $1\le j \le n-1$. Let $L' := K_1\cup\cdots\cup K_{n-1}$, let $N(K_n)$ be a tubular neighborhood of $K_n$. We can view $L'$ as a link in the solid torus $S^3-N(K_n)$, and the same argument as above gives $\AHI(L')\cong \bC^{2^{n-1}}$. By Proposition \ref{Cor-minimal_rank-meridional_surface}, we can find a meridional disk $S$ in the solid torus 
$S^3-N(K_n)$ which is disjoint from $L'$. Therefore $L$ is the disjoint union of $L'$ and the unknot, and the result follows from the induction hypothesis on $L'$.
\end{proof}

\section{The case of 2-component links}
\subsection{The linking numbers of 2-component links}
This subsection proves that Condition (2) of Theorem \ref{Thm-forest_detection_linking_number} is implied by Condition (1) when $n=2$. The main result of this subsection is the following lemma.
\begin{Lemma}\label{Lemma-Kh_bound_linking_number}
Suppose $L=K_1\cup K_2$ is a link with 2 components such that 
$$\rank_{\mathbb{Z}/2}\Kh(L;{\mathbb{Z}/2})=4,$$ then
$|\lk(K_1,K_2)|\le 1.$
\end{Lemma}
Combining this lemma with Theorem \ref{Thm-forest_detection_linking_number}, we have the following  corollary.
\begin{Corollary}
Suppose $L$ is a link with 2 components and $\rank_{\mathbb{Z}/2}\Kh(L;{\mathbb{Z}/2})=4$, then
$L$ is either the 2-component unlink or the Hopf link.
\end{Corollary}
We start the proof of Lemma \ref{Lemma-Kh_bound_linking_number} with the following lemma.
\begin{Lemma}\label{lem_exchangeable_braid}
Suppose $L=K_1\cup K_2$ satisfies the assumption of Lemma \ref{Lemma-Kh_bound_linking_number}, and suppose $L$ is not the unlink. Then $K_1$ is an unknot, and $K_2$ is a braid closure with axis $K_1$. Similarly, $K_2$ is an unknot, and $K_1$ is a braid closure with axis $K_2$.
\end{Lemma}
\begin{proof}
Proposition \ref{Prop-miminal_rank_unknot} implies that
$K_1$ and $K_2$ are both unknots. Let $N(K_1)$ be a tubular neighborhood of $K_1$, then $K_2$ is a knot in the solid torus $S^3-N(K_1)$. 
Proposition \ref{Prop-rank_I^natural} yields
$$
\dim_\bC\II^\natural(L,p)=2 
$$
for every $p\in L$.
By Proposition \ref{Prop-AHI=I^natural},
\begin{equation}\label{AHI_rank_2}
\dim_\bC\AHI(K_2)=\dim_\bC\II^\natural(L,p)=2 .
\end{equation}
If $\AHI(K_2)$ is supported at f-degree $0$, then by Theorem \ref{Theorem-2g+n}, 
there exists a meridional disk which is disjoint from $K_i$. This means $K_i$ is included in a 3-ball in the solid torus $S^3-N(K_1)$, hence $K_1$ and
$K_2$ are split, therefore the link $L$ is the unlink, contradicting the assumption. 
Therefore $\AHI(K_2)$ is supported at f-degrees $\pm l$ for $l>0$. By \eqref{AHI_rank_2}, we have $\AHI(K_2,\pm l)\cong\mathbb{C}$ and $\AHI(K_2)$ vanishes at
all the other f-degrees. According to Proposition \ref{AHI-braid-detection}, $K_2$ is the closure of an $l$-braid in $S^3-N(K_{1})$.

The same argument for $S^3-N(K_2)$ proves the second half of the lemma.
\end{proof}
\begin{remark}
 A link described by the conclusion of Lemma \ref{lem_exchangeable_braid} is called an 
 \emph{exchangeably braided link}. This concept was first introduced and studied by
  Morton in \cite{Mor-braid}. 
\end{remark}

Let $l>1$ be an integer. Recall that the braid group $B_l$ is given by
\begin{equation*}
B_l=\langle \sigma_1,\cdots,\sigma_{l-1}|\sigma_i\sigma_{i+1}\sigma_i=  \sigma_{i+1}\sigma_{i}\sigma_{i+1},~
      \sigma_i\sigma_j=\sigma_j\sigma_i~(j-i\ge 2)~           \rangle.
\end{equation*}

The reduced Burau representation (see \cite{Birman}) is a group homomorphism
\begin{equation*}
\rho:B_l\to GL(l-1,\mathbb{Z}[t,t^{-1}])
\end{equation*}
defined by
\begin{equation*}
\rho(\sigma_i):=
\begin{pmatrix}
I_{i-2} &   &    &   &\\
        & 1  &  0  & 0  &\\
        & t  & -t  & 1  & \\
        &  0 &  0  &  1 & \\
        &   &    &   & I_{l-i-2}
\end{pmatrix}, ~ 2\le i\le l-2,
\end{equation*}
\begin{equation*}
\rho(\sigma_1):=
\begin{pmatrix}
 -t  & 1  & \\
   0  &  1 & \\
    &   & I_{l-3}
\end{pmatrix}, ~
\rho(\sigma_{n-1}):=
\begin{pmatrix}
 I_{l-3} &   & \\
     & 1  &0 \\
    &  t & -t
\end{pmatrix} 
\end{equation*}
for $l>2$, while for $l=2$ it is defined by $\rho(\sigma_1):= (-t)$. Notice that for every $\beta\in B_l$, there exists an integer $a$ such that 
\begin{equation}\label{eqn_det_rho_beta_power_of_t}
\det(\rho(\beta))=\pm t^a.
\end{equation}
We also need the follow result by Morton.
\begin{Theorem}[{\cite[Theorem 3]{Mor-braid}}]\label{Morton-Alexander-braid}
Suppose $L=U\cup \widehat{\beta}$ is a 2-component link where $U$ is the unknot and
$\widehat{\beta}$ is the closure of a braid $\beta\in B_l$ with axis $U$.
Then the multi-variable Alexander polynomial $\Delta_L(x,t)$ of $L$ is
given by
\begin{equation*}
\Delta(x,t)\doteq\det(xI-\rho(\beta)(t))
\end{equation*}
where $x$ and $t$ are variables corresponding to $U$ and $\hat{\beta}$ respectively.
\end{Theorem}
\begin{remark}
The sign ``$\doteq$'' in Theorem \ref{Morton-Alexander-braid} means the two sides are equal up to a multiplication by $\pm x^a t^b$. This is necessary because
 the multi-variable Alexander polynomial is only defined up to a multiplication by $\pm x^a t^b$.
\end{remark}

\begin{Lemma}\label{Lemma_Delta_xt_lower_bound}
Suppose $L=K_1\cup K_2$ is an exchangeably braided link with linking number $l\ge 2$. Let $\Delta_L(x,y)$ be the multi-variable Alexander polynomial of $L$. Then the expansion of the Laurent polynomial 
$(x-1)(y-1)\Delta_L(x,y)$
has (strictly) more than $4$ terms.
\end{Lemma}
\begin{proof}
Without loss of generality, assume  $x$ and $y$ are the variables corresponding to $K_1$ and $K_2$ respectively.
Let $\beta\in B_l$ be the braid whose closure is isotopic to $K_2$ as a link in the solid torus $S^3-K_1$.
By \eqref{eqn_det_rho_beta_power_of_t} and Theorem \ref{Morton-Alexander-braid}, we have
    \begin{align}
    \Delta_L(x,y)&\doteq (-1)^{l-1}\det(\rho(\beta)(y))+ f_1(y)x+\cdots+f_{l-2}(y)x^{l-2}+ x^{l-1}  \nonumber \\ 
                 &= \pm y^a + f_1(y)x+\cdots+f_{l-2}(y)x^{l-2}+ x^{l-1}  \label{eq-Delta-xt}
    \end{align} 
  for some $a\in \mathbb{Z},f_i\in \mathbb{Z}[y,y^{-1}]$.

Switching the roles of $K_1$ and $K_2$, we have
\begin{equation}\label{eq-Delta-tx}
 \Delta_L(x,y)\doteq \pm x^b + g_1(x)y+\cdots+g_{l-2}(x)y^{l-2}+ y^{l-1}
 \end{equation}
 for some $b\in \mathbb{Z},g_i(x)\in \mathbb{Z}[x,x^{-1}]$.

By \eqref{eq-Delta-xt}, we have
\begin{equation*}
(y-1)\Delta_L(x,y)\doteq \pm(y-1)y^a + (y-1)f_1(y)x+\cdots+(y-1)f_{l-2}(y)x^{l-2}+ (y-1)x^{l-1},
\end{equation*}
hence we have the following expansion in order of increasing powers of $x$:
\begin{equation*}
(x-1)(y-1)\Delta_L(x,y)\doteq \pm(y-1)y^a + h_1(y)x+\cdots+ h_{l-1}(y)x^{l-1} + (y-1)x^l.
\end{equation*}
The right-hand-side has at least $4$ terms after expansion, which come from the lowest and highest powers of $x$.
Suppose it has only $4$ terms in  total, then all the terms in between must vanish, thus we have
\begin{equation} \label{eqn_(x-1)(y-1)Delta_four_terms}
(x-1)(y-1)\Delta_L(x,y)\doteq \pm(y-1)y^a + (y-1)x^l.
\end{equation}
Plugging in $x=1$, we have
\begin{equation*}
0= \pm(y-1)y^a + (y-1),
\end{equation*}
therefore $a=0$, and \eqref{eqn_(x-1)(y-1)Delta_four_terms} gives
\begin{equation*}
\Delta_L(x,y)\doteq \frac{ -(y-1) + (y-1)x^l}{(x-1)(y-1)}=1+x+\cdots +x^{l-1}
\end{equation*}
which contradicts \eqref{eq-Delta-tx} when $l\ge 2$.
\end{proof}

\begin{proof}[Proof of Lemma \ref{Lemma-Kh_bound_linking_number}]
Suppose $l\ge 2$.
We use $\HFK$ and $\HFL$ to denote 
the Heegaard knot Floer homology \cite{OS:HFK,Ras:HFK} and link Floer homology \cite{OS:HFL} respectively. 
The link Floer homology was originally defined only for ${\mathbb{Z}/2}$-coefficients, and was generalized to $\bZ$-coefficients in \cite{Sar-HFL}.
It is known that 
$$
\rank_{\bQ} \HFK(L;\bQ)=\rank_{\bQ}\HFL(L;\bQ),
$$ 
but $\HFL$ carries more refined gradings.

By \cite[Corollary 1.7]{Dowlin}, we have
\begin{equation}\label{eq_HFK<Kh}
\rank_\bQ \HFK(L;\bQ)\le 2\rank \Khr(L;\bQ)\le 2\rank \Khr(L;{\mathbb{Z}/2})=4.
\end{equation}
On the other hand, let $\Delta_L(x,y)$ be the multi-variable Alexander polynomial of $L$, it was proved in \cite{OS:HFL} that the graded Euler characteristic of $\HFL(L;\bQ)$ satisfies 
$$
\chi(\HFL(L;\bQ))\doteq (x-1)(y-1)\Delta_L(x,y).
$$ 
By Lemma \ref{Lemma_Delta_xt_lower_bound}, we have
$$
\rank_\bQ \HFK(L;\bQ)=\rank_\bQ \HFL(L;\bQ)> 4
$$
which contradictions \eqref{eq_HFK<Kh}.
\end{proof}

We introduce the following condition on a link $L\subset S^3$:
\begin{Condition}
\label{cond_cycle}
~

\begin{itemize}
\item[(1)] $L$ has $n\ge 3$ connected components,
\item[(2)] the rank of $\Kh(L;{\mathbb{Z}/2})$ is $2^n$,
\item[(3)] the components of $L$ can be arranged as a sequence $K_1,\cdots,K_n$, 
such that the linking number of $K_i$ and $K_j$ ($i\neq j$) is $\pm 1$ when $|i-j|=1$ or $n-1$, and is zero otherwise.
\end{itemize}
\end{Condition}
 
Theorem \ref{Thm-forest_detection_linking_number} and Lemma \ref{Lemma-Kh_bound_linking_number} have the following consequence.

\begin{Lemma} \label{thm_modulo_cond_cycle}
If $L_0$ is an $m$-component link with $\rank_{\bZ/2} \Kh(L_0;\bZ/2)= 2^m$, then either $L_0$ is a forest of unknots, or $L_0$ contains a sublink $L$ satisfying Condition \ref{cond_cycle}.
\end{Lemma}
 
\begin{proof}
Let $K_1,\cdots,K_m$ be the components of $L_0$. By Proposition \ref{Prop_Kh_rank_sublink} and Lemma \ref{Lemma-Kh_bound_linking_number}, for each pair $i\neq j$, the linking number of $K_i$ and $K_j$ is equal to $0$ or $\pm1$. Let $G$ be a simple graph with $m$ vertices $p_1,\cdots,p_m$, such that $p_i$ and $p_j$ are connected by an edge if and only if $|\lk(K_i,K_j)|=\pm 1$. If $G$ is a forest, then Theorem \ref{Thm-forest_detection_linking_number} implies that $L_0$ is a forest of unknots. If $G$ contains a cycle, then the vertices of the shortest cycle of $G$ corresponds to a sublink of $L_0$ satisfying Condition \ref{cond_cycle}.
\end{proof}

The next subsection gives a proof of Theorem \ref{thm_small_rank_detection_intro}.  
After that, the rest of this article is devoted to proving that there is no link satisfying Condition \ref{cond_cycle}, therefore Theorem \ref{The_main_theorem} will follow from Lemma \ref{thm_modulo_cond_cycle}.

\subsection{Some 2-component links with small Khovanov homology}
This subsection gives a proof of Theorem \ref{thm_small_rank_detection_intro}, and shows that the 
bi-graded Khovanov homology detects some simple 2-component links other than the unlink and the Hopf link. The results of this subsection will not be used in the proof of Theorem \ref{The_main_theorem}.

Recall that the internal grading of the Khovanov homology of a link $L$ is introduced in \cite[Section 2]{Kh-unlink} as $h-q$, where 
$h$ is the homological grading and $q$ is the  quantum grading. 
 The following is a special case of a more general result due to Batson and Seed.
\begin{Theorem}[{\cite[Corollary 4.4]{Kh-unlink}}]\label{Theorem_Kh_linking_number}
Suppose $L=K_1\cup K_2$ is a 2-component oriented link. Then we have
$$
\rank_{\mathbb{F}}^{l}\Kh(L;\mathbb{F})\ge \rank_{\mathbb{F}}^{l+2\lk(K_1,K_2)} (\Kh(K_1;\mathbb{F})\otimes \Kh(K_2;\mathbb{F}))
$$
where $\mathbb{F}$ is an arbitrary field and $\rank^k$ denotes the rank of the summand with internal grading $k$.
\end{Theorem}

Let $L_1$ be the oriented link given in Figure \ref{fig_braid_1}, then $L_1$  
is isotopic to the link L4a1 in the Thistlethwaite link table. 
 Its Khovanov homology is given by 
\begin{equation*}
\Kh(L_1;\mathbb{Z})=\bZ_{(0)}\oplus \bZ_{(1)}\oplus \bZ_{(2)}\oplus (\bZ/2)_{(3)}\oplus \bZ_{(4)}^2\oplus \bZ_{(6)},
\end{equation*}
where the subscripts represent the internal gradings. 
\begin{Theorem}\label{Theorem_Kh_detect_L4a1}
Let $L=K_1\cup K_2$ be a 2-component oriented link. Suppose we have 
$$
\Kh(L;\bZ/2)\cong \Kh(L_1;\bZ/2)
$$
as bi-graded abelian groups, then
$L$ is isotopic to $L_1$.
\end{Theorem} 
\begin{proof}
To simplify notation, we use $\bF$ to denote the field $\bZ/2$.
By \eqref{eq_Kh_L4a1} and the universal coefficient theorem, we have
\begin{equation}\label{eq_Kh_L4a1}
\Kh(L;\bF)=\bF_{(0)}\oplus \bF_{(1)}\oplus \bF_{(2)}^2\oplus \bF_{(3)}\oplus \bF_{(4)}^2\oplus \bF_{(6)},
\end{equation}
where the subscripts represent the internal grading.
By Theorem \ref{Theorem_Kh_linking_number}, we have
$$
8=\rank_\bF\Kh(L;\bF)\ge \rank_\bF\Kh(K_1;\bF) \cdot \rank_\bF\Kh(K_2;\bF),
$$
hence we have $\rank_\bF\Kh(K_i;\bF)\le 4$. On the other hand, we have $\rank_\bF\Kh(K_i;\bF)=2 \rank_\bF\Khr(K_i;\bF)$, and
$\rank_\bF\Khr(K_i;\bF)$ is always odd for knots. Therefore 
\begin{align*}
 \rank_\bF\Kh(K_1;\bF) &=\rank_\bF\Kh(K_2;\bF)=2,
 \\
 \rank_\bF\Khr(K_1;\bF) &=\rank_\bF\Khr(K_2;\bF)=1.
\end{align*}
By Kronheimer-Mrowka's unknot detection theorem, both $K_1$ and $K_2$ are unknots. Hence
$$
\Kh(K_1;\bF)\otimes \Kh(K_2;\bF)=\bF_{(-2)}\oplus \bF_{(0)}^2\oplus \bF_{(2)}.
$$
By Theorem \ref{Theorem_Kh_linking_number} and \eqref{eq_Kh_L4a1}, we have $\lk(K_1,K_2)=1~\text{or}~2$,
hence $K_1$ is homotopic to the closure of a 1-braid or a 2-braid in the solid torus $S^3-N(K_2)$. 
Let $l:=\lk(K_1,K_2)$, and let $\hat{\beta}_l$ be the closure of the corresponding 1-braid or 2-braid.
By \cite[Section 4.4]{AHI}, 
\begin{equation}\label{eq_parity_AHI_odd}
\dim_\bC\AHI(K_1,l)\equiv\dim_\bC\AHI(\hat{\beta}_l,l)= 1 \mod~2 ,
\end{equation}
and
\begin{equation}\label{eq_parity_AHI_even}
\dim_\bC\AHI(K_1,j)\equiv\dim_\bC\AHI(\hat{\beta}_i,j)= 0 \mod~2 ~\text{ if}~j>l.
\end{equation}
By Proposition \ref{Prop-AHI=I^natural} and Kronheimer-Mrowka's spectral sequence (\cite[Theorem 8.2]{KM:Kh-unknot}), we have 
\begin{equation}\label{eq_rank_inequal_AHI}
4=\rank_\bF\Khr(L_1,p;\bF) \ge \dim_\bC\AHI(K_1),
\end{equation}
where $p\in K_2$.

If $l=1$, we have $\AHI(K_1,\pm 1)=\bC$ by \eqref{eq_parity_AHI_odd} and \eqref{eq_rank_inequal_AHI}. By \eqref{eq_parity_AHI_even} and \eqref{eq_rank_inequal_AHI}, we have $\AHI(K_1,j)= 0$ for all $j>l$. Hence the top f-grading of $\AHI(K_1)$ is 1 and $K_1$ is the closure of a 1-braid by
Theorem \ref{AHI-braid-detection}.
This implies that $L$ is a Hopf link, whose Khovanov homology is different from $\Kh(L_1)$, which yields a contradiction. 
Hence we must have $l=2$.
By  similar arguments as above, we obtain
 $\AHI(K_1,\pm 2)=\bC$ and the top f-grading of $\AHI(K_1)$ is 2. Therefore $K_1$ is the closure of a 2-braid. Since $K_1$ is an unknot,
this 2-braid must be given by a generator of the 2-braid group. The proof is then completed by directly checking all the possiblities. 
\end{proof}

Now we prove the second part of Theorem \ref{thm_small_rank_detection_intro}.
Let $T$ be the left-handed trefoil, we have
$$
\Kh(T;\bZ)\cong \bZ_{(-3,-9)}\oplus\bZ_{(-2,-5)}\oplus \bZ_{(0,-3)}\oplus \bZ_{(0,-1)} \oplus (\bZ/2)_{(-2,-7)},
$$
where the subscripts represent the $(h,q)$-bigrading.
Let $L_2$ be the disjoint union of $T$ and an unknot $U$, then we have
$$
\Khr(L_2,p;\bZ)\cong \Kh(T;\bZ),
$$
where the base point $p\in U$. 
\begin{Theorem}
Let $L=K_1\cup K_2$ be a 2-component link with a base point $q\in K_2$. Suppose we have 
$$
\Khr(L,q;\bZ)\cong \Khr(L_2,p;\bZ)
$$
as bi-graded abelian groups. Then the link
 $L$ splits into a left-handed trefoil $K_1$ and an unknot $K_2$.
\end{Theorem}
\begin{proof}
By Kronheimer-Mrowka's spectral sequence (\cite[Theorem 8.2]{KM:Kh-unknot}), we have
\begin{equation}
\label{eqn_Khr_rank_no_more_than_4_trefoil_unknot}
4=\rank_{\bZ}\Khr(L,q;\bZ)\ge \dim_\bC\II^\natural(L,q).
\end{equation}
Let $\Gamma$ be the local system associated with $K_1$. We have
$$
\dim_\bC\II^\natural(L,q)\ge \rank_{\mathcal{R}}\II^\natural(L,q;\Gamma).
$$
By Proposition \ref{LocallyFreePart-inst}, we have
$$
\rank_{\mathcal{R}}\II^\natural(L,q;\Gamma)=\rank_{\mathcal{R}}\II^\natural(K_2\cup U,q;\Gamma).
$$
By excision, we have
$$
\rank_{\mathcal{R}}\II^\natural(K_2\cup U,q;\Gamma)=2\dim_\bC\II^\natural(K_2,q).
$$
Hence $\dim_\bC\II^\natural(K_2,q)\le 2 $. We have $\dim_\bC\II^\natural(K_2,q)$ is always odd since crossing-changes do not change
the parity of $\II^\natural$ and $\II^\natural(\text{unknot})\cong\bC$. Hence $K_2$ is the unknot by \cite[Proposition 1.4]{KM:Kh-unknot}.
Let $\bF=\bZ/2$.
The universal coefficient theorem implies
\begin{align*}
\Khr(L,q;\bF)&\cong \Khr(L_2,p;\bF)\\
&\cong \bF_{(-3,-9)}\oplus\bF_{(-2,-5)}\oplus \bF_{(0,-3)}\oplus \bF_{(0,-1)} \oplus \bF_{(-2,-7)}\oplus \bF_{(-3,-7)} \\
&\cong \bF_{(6)}\oplus \bF_{(5)}\oplus \bF_{(4)}\oplus\bF_{(3)}^2\oplus \bF_{(1)},
\end{align*}
where the subscripts in the second last row represent the $(h,q)$-bigrading, and the subscripts in the  last row represent the internal grading. 
According to \cite[Corollary 3.2.C]{Shu:torsion_Kh}, we have
$$
\Kh^{i,j}(L;\bF)\cong  \Khr^{i,j-1}(L,q;\bF)\oplus \Khr^{i,j+1}(L,q;\bF).
$$
Hence
\begin{equation}
\label{eqn_Kho_internal_grading_Z/2_trefoil_unknot}
\Kh(L;\bF)\cong \bF_{(0)}\oplus \bF_{(2)}^2\oplus \bF_{(3)}\oplus \bF_{(4)}^3\oplus \bF_{(5)}^2\oplus \bF_{(6)}\oplus \bF_{(7)}.
\end{equation}
By Theorem \ref{Theorem_Kh_linking_number}, we have 
$$
12=\rank_{\bF}\Kh(L;\bF)\ge \rank_{\bF}\Kh(K_1;\bF) \rank_{\bF}\Kh(K_2;\bF)=2\rank_{\bF}\Kh(K_1;\bF).
$$
Hence $\rank_{\bF}\Kh(K_1;\bF)\le 6$ and $\rank_{\bF}\Khr(K_1;\bF)\le 3$.
Kronheimer-Mrowka's unknot detection theorem \cite{KM:Kh-unknot} and Baldwin-Sivek's trefoil detection theorem \cite{BS} imply $K_1$
is either an unknot or a trefoil. 

Suppose $K_1$ is an unknot. Then we have $\lk(K_1,K_2)=1~\text{or}~2$ by Theorem \ref{Theorem_Kh_linking_number} and \eqref{eqn_Kho_internal_grading_Z/2_trefoil_unknot}.
Now $K_1$ is a link in the solid torus $S^3-N(K_2)$ with winding number 1 or 2,
and by Proposition \ref{Prop-AHI=I^natural},
$$
4\ge \dim_\bC \AHI(K_1).
$$
 The argument in the proof of Theorem \ref{Theorem_Kh_detect_L4a1} shows that
$K_1$ is either the closure of the 1-braid or the closure of a generator of the 2-braid group. In either case, $\Kh(L)$ is
not isomorphic to $\Kh(L_2)$, which is a contradiction.

By the discussion above, $K_1$ must be a trefoil, hence Theorem \ref{Theorem_Kh_linking_number} and \eqref{eqn_Kho_internal_grading_Z/2_trefoil_unknot} implies
$\lk(K_1,K_2)=0$. Hence $K_1$ is homotopic to the unknot in $S^3-N(K_2)$. By \cite[Section 4.4]{AHI}, we have
$\dim_\bC\AHI(K_1,j)$ is even for all $|j|>0$. 
We also have $\dim_\bC\AHI(K_1,0)\ge \dim_\bC\AHI(\mathcal{U}_1,0)=2$ by Proposition \ref{Prop-AHI-rank-inequality}. 
By \eqref{eqn_Khr_rank_no_more_than_4_trefoil_unknot}, we have $\dim_\bC\AHI(K_1)=\dim_\bC\II^\natural(L,q) \le 4$, hence the argument above implies that the top f-grading of $\AHI(K_1)$ is $0$. By Theorem \ref{Theorem-2g+n},
$K_2$ has a Seifert disk which is disjoint from $K_1$, hence $L$ is a split link. 
\end{proof}

\section{Topological properties from instanton Floer homology}\label{sec_top_properties}

From now on, let $L=K_1\cup\cdots\cup K_n$ be a hypothetical link that satisfies Condition \ref{cond_cycle}. The goal is to deduce a contradiction from Condition \ref{cond_cycle}. This section derives several topological properties of $L$ using instanton Floer homology.

\subsection{Seifert surfaces of $K_i$}

\begin{Proposition}
\label{prop_embedded_disks}
For each $K_i$, there exists an embedded disk $D_i$ such that 
\begin{enumerate}
\item $\partial D_i = K_i$, 
\item for each $j\neq i$, if $|i-j|=1$ or $n-1$, then $D_i$ intersects $K_j$ transversely at one point; otherwise, the disk $D_i$ is disjoint from $K_j$.
\end{enumerate}
\end{Proposition}
\begin{proof}
Pick a base point $p\in K_i$. By Proposition \ref{Prop-rank_I^natural}, we have $\dim\II^\natural(L,p)=2^{n-1}$. By Proposition 
\ref{Prop-miminal_rank_unknot}, $K_i$ is an unknot. Let $N(K_i)$ be a tubular neighborhood of $K_i$, and view $L-K_i$ as a link in the solid torus $S^3-N(K_i)$.
By Proposition \ref{Prop-AHI=I^natural},
$$
\dim \AHI(L-K_i)=\II^\natural(L,p)=2^{n-1}.
$$
By Corollary \ref{Cor-minimal_rank-meridional_surface}, there exists a meridional disk $\hat D_i$ which 
is disjoint from $K_j$ if $|i-j|\neq 1$ or $n-1$ and  intersects $K_j$ transversely at one point if $|i-j|= 1$ or $n-1$.
The meridional disk $\hat D_i$ extends to the desired Seifert disk of $K_i$.
\end{proof}

\begin{Definition}
\label{def_admissible_D}
Let $D_1,\cdots,D_n$ be a sequence of immersed disks in $\bR^3$ such that $\partial D_i = K_i$ for all $i$. We say that the sequence $D_1,\cdots,D_n$ is \emph{generic}, if every self-intersection point of $\sqcup D_i$ is locally diffeomorphic to one of the following models in $\bR^3$ at $(0,0,0)$: 
\begin{enumerate}
\item the intersection of $\{(x,y,z)|z=0, y\ge 0\}$ and the $yz$-plane,
\item the intersection of the $xy$-plane and the $yz$-plane, 
\item the intersection of the $xy$-, $yz$-, and $xz$-planes. 
\end{enumerate}

If $D=(D_1,\cdots,D_n)$ is generic, let $\Sigma_1(D)$, $\Sigma_2(D)$, $\Sigma_3(D)$ be the set of self-intersection points described by (1), (2), (3) above respectively.
\end{Definition}

\begin{Definition}
If $D=(D_1,\cdots,D_n)$ is generic, define the \emph{complexity} of $D$ to be the number of components of $\Sigma_2(D)$.
\end{Definition}

Notice that if $D$ is generic, then
the complexity of $D$ is greater than or equal to $\frac12 \#\Sigma_1(D)$ which is at least $n$.

\begin{Definition}
We say that the sequence $D = (D_1,\cdots, D_n)$ is \emph{admissible}, if 
\begin{enumerate}
\item  $D$ is generic,
\item $\#\Sigma_1(D)=2n$,
\item every point in $\Sigma_3(D)$ is contained in at least two different disks in $D$.
\end{enumerate}

\end{Definition}

\begin{remark}
In the definitions above, the disks $\{D_i\}$ are only required to be immersed. 
Condition (2) in the definition above is equivalent to the following statement: for each $j\neq i$, if $|i-j|=1$ or $n-1$, then $D_i$ intersects $K_j$ transversely at one point; otherwise, the immersed disk $D_i$ is disjoint from $K_j$. Moreover, the interior of $D_i$ is disjoint from $K_i$.
\end{remark}

\begin{Proposition}
\label{prop_complexity_n}
There exists a sequence of disks $D=(D_1,\cdots,D_n)$, such that $\partial D_i = K_i$ for all $i$, and $D$ is admissible with complexity $n$. 
\end{Proposition}

\begin{proof}
By Proposition \ref{prop_embedded_disks}, there exists a sequence of disks $\hat D=(\hat D_1,\cdots \hat D_n)$ such that for all $i$, $\hat D_i$ is embedded, $\partial \hat D_i = K_i$, and $\#\Sigma_1(\hat D) = 2n$. Perturb $\hat D_1,\cdots \hat D_n$ such that they are generic. Since all the disks are embedded, every point in $\Sigma_3(\hat D)$ is contained in three different disks. 
Therefore $\hat D$ is admissible, hence admissible configurations exist.
Let $D=(D_1,\cdots,D_n)$ be an admissible configuration with minimal complexity. 

We first show that all the $D_i$'s are embedded. Suppose there exists $i$ such that $D_i$ is not embedded, then by admissibility, $D_i$ does not have triple self-intersections, and the self-intersection locus of $D_i$ is a disjoint union of circles. Let $\gamma\subset D_i$ be a circle in the self-intersection of $D_i$.

Let $B^2$ be the unit disk in $\bR^2$, and let $f_i : B^2\to \bR^3$ be an immersion that parametrizes $D_i$. Then $f_i^{-1}(\gamma)$ is a double cover of $\gamma$. There are three possibilities:

Case 1. $f_i^{-1}(\gamma)$ is a disjoint union of two circles, and they bound disjoint disks $B_1$ and $B_2$. In this case, take a diffeomorphism $\iota$ from $B_1$ to $B_2$, such that 
$$
(f_i\circ \iota)|_{\partial B_1} = f_i|_{\partial B_1}.
$$
Define 
$$
f_i'(p) := \begin{cases}
     f_i(p) & \text{ if } p\notin B_1\cup B_2 ,\\
     f_i(\iota(p)) & \text{ if } p\in B_1 ,\\
     f_i(\iota^{-1}(p)) & \text{ if } p\in B_2 .\\         
     \end{cases}
$$
By smoothing $f_i'$, we obtain an immersed disk with the same boundary as $D_i$ but has fewer self-intersection components. Figure \ref{fig_local_surgery} shows a local picture of $f_i'$ after the smoothing. Replacing $D_i$ by the image of the smoothed $f_i'$ decreases the complexity of $D$ and preserves the admissibility condition. 
\begin{figure}  
  \includegraphics[width=0.8\linewidth]{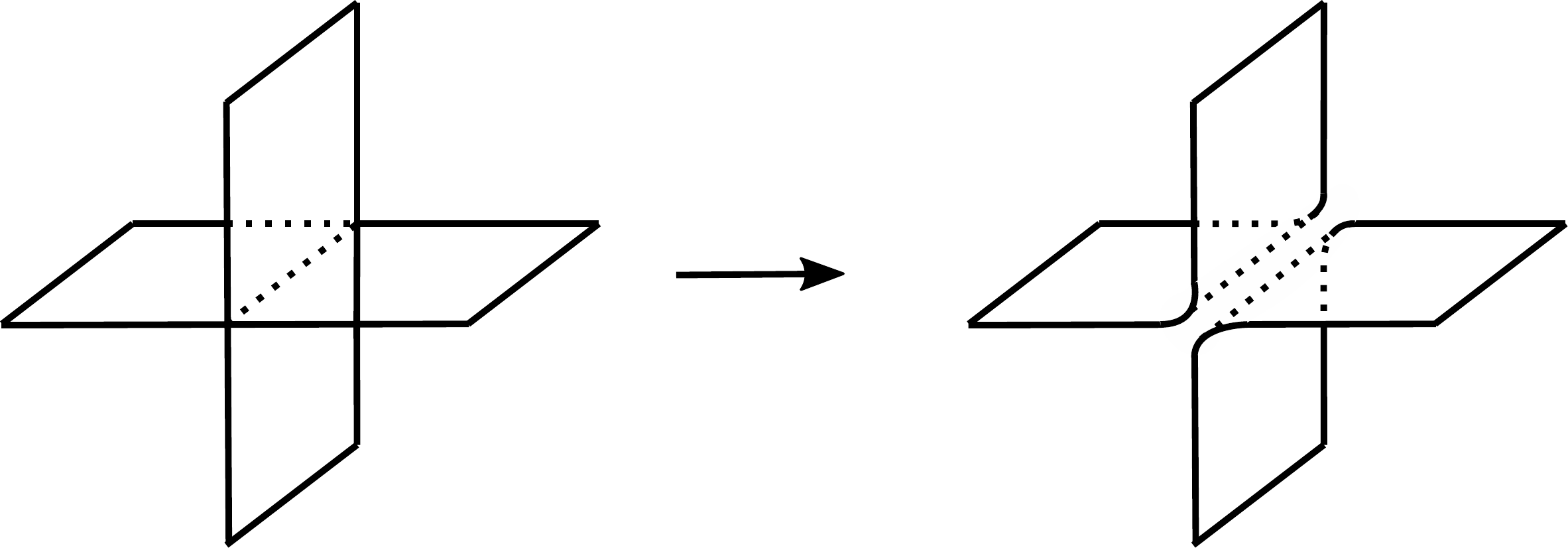}
  \caption{The local construction of $f_i'$ after smoothing.}
  \label{fig_local_surgery}
\end{figure}

Case 2. $f_i^{-1}(\gamma)$ is a disjoint union of two circles, and they bound disks $B_1$ and $B_2$ with $B_1\supset B_2$. In this case, take a diffeomorphism $\iota$ from $B_1$ to $B_2$, such that 
$$
(f_i\circ \iota)|_{\partial B_1} = f_i|_{\partial B_1}.
$$
Define 
$$
f_i'(p) := \begin{cases}
     f_i(p) & \text{ if } p\notin B_1,\\
     f_i(\iota(p)) & \text{ if } p\in B_1 .\\         
     \end{cases}
$$
Replacing $f_i$ by the smoothed version of $f_i'$ gives an admissible configuration with smaller complexity.

Case 3. $f_i^{-1}(\gamma)$ is one circle, and it bounds a disk $B_1$. In this case, $f_i|_{\partial B_1}$ is a non-trivial covering map. Take a diffeomorphism $\iota$ from $B_1$ to $B_1$, such that its restriction to $\partial B_1$ is the deck transformation.
Define 
$$
f_i'(p) := \begin{cases}
     f_i(p) & \text{ if } p\notin B_1,\\
     f_i(\iota(p)) & \text{ if } p\in B_1 .\\         
     \end{cases}
$$
Replacing $f_i$ by the smoothed version of $f_i'$ gives an admissible configuration with smaller complexity.

Since $D$ is assumed to have minimal complexity among admissible configurations, we conclude that $D_i$ has to be embedded.

Now we show that the complexity of $D$ is $n$. In fact, since all the disks $D_i$ are embedded, the intersection of $D_i$ and $D_j$ $(i\neq j)$ is a disjoint union of compact $1$-manifolds possibly with boundary. 
If the complexity of $D$ is greater than $n$, then there exists $i\neq j$ such that the intersection of $D_i$ and $D_j$ contains a circle $\gamma$. The circle $\gamma$ bounds a disk $B_1$ in $D_i$, and bounds a disk $B_2$ in $D_j$. Let
\begin{align*}
D_i' &:= (D_i - B_i) \cup B_j,\\
D_j' &:= (D_j - B_j) \cup B_i.
\end{align*}
Replace $D_i$, $D_j$ by $D_i'$ and $D_j'$ and smooth the corners, we obtain a generic configuration with smaller complexity. Since neither $D_i'$ nor $D_j'$ has triple self-intersection points, the new configuration is still admissible, contradicting the definition of $D$.
In conclusion, the complexity of $D$ is $n$.
\end{proof}

Let $L':=K_1\cup\cdots \cup K_{n-1}$. 
Proposition \ref{prop_complexity_n} has the following corollary.

\begin{figure}  
  \includegraphics[width=0.8\linewidth]{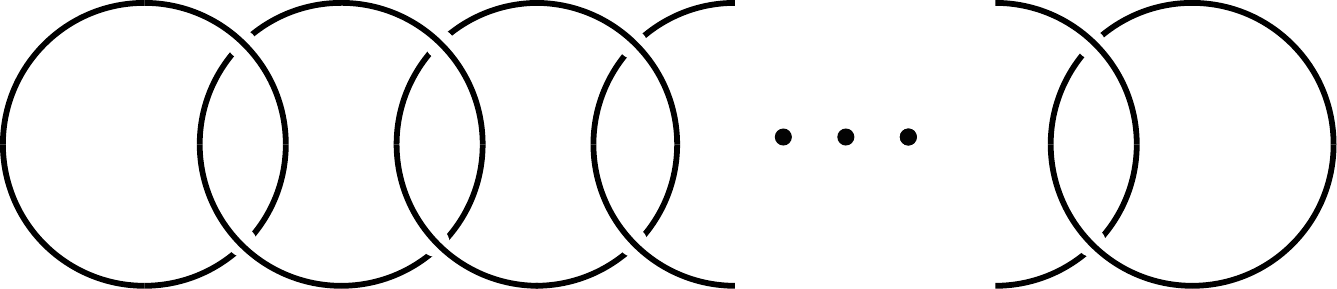}
  \caption{The link $L'$.}
  \label{fig_standard_link} 
  \vspace{\baselineskip}
  \includegraphics[width=0.8\linewidth]{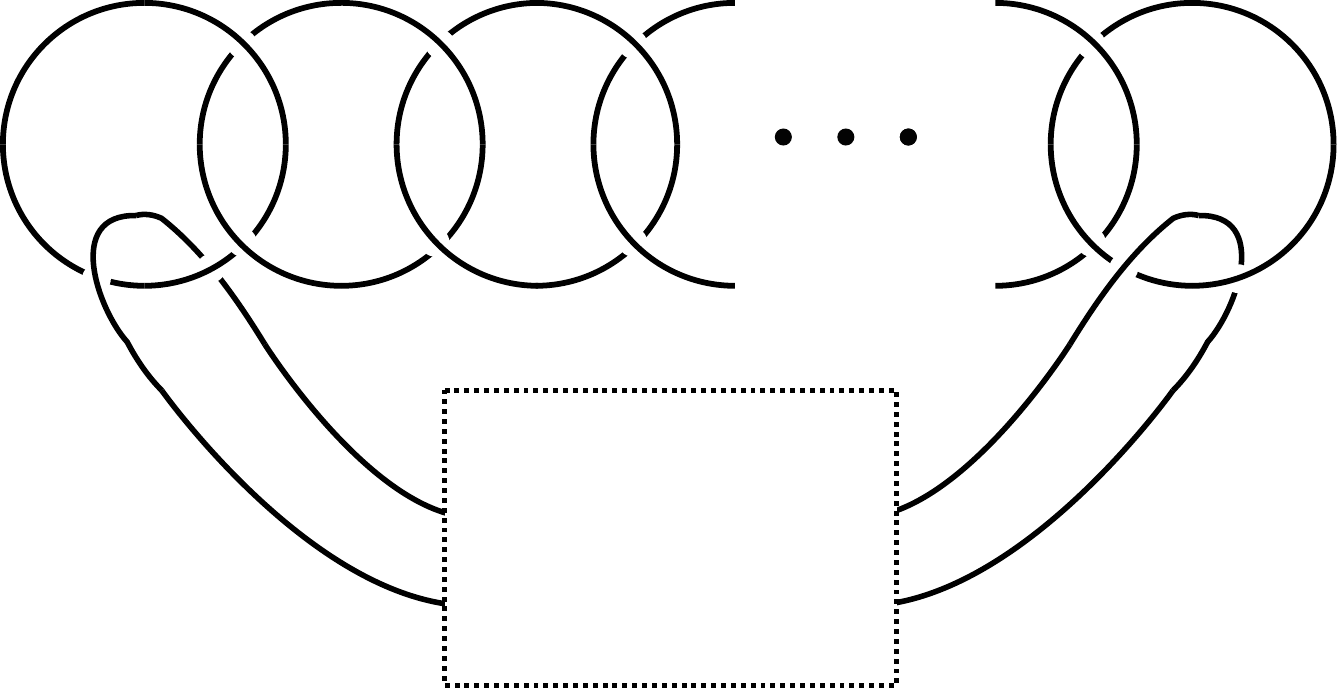}
  \caption{The link $L$.}
  \label{fig_cable}
\end{figure}

\begin{Corollary}
\label{cor_cable}
The link $L'$ is a connected sum of $n-2$ Hopf links as given by Figure \ref{fig_standard_link}. 
The link $L$ has a diagram described by Figure \ref{fig_cable}.
\end{Corollary}

\begin{proof}
By Proposition \ref{prop_complexity_n}, there exists a sequence of disks $D_1,\cdots,D_{n-1}$, such that (1) $D_i$ is embedded and $\partial D_i = K_i$ for $i=1,\cdots,n-1$, (2) if $|i-j|=1$, the disks $D_i$ and $D_j$ intersect at an arc, (3) if $i\neq j$ and $|i-j|\neq 1$, the disks $D_i$ and $D_j$ are disjoint.  It follows that $L'$ is isotopic to a connected sum of $n-2$ Hopf links as given by Figure \ref{fig_standard_link}. Moreover, we may choose the sequence of disks $D_1,\cdots,D_{n-1}$ in such a way that $K_n$ bounds an embedded disk $D_n$ which intersects $D_1$ and $D_{n-1}$ respectively at an arc, and is disjoint from $D_2\cup\cdots \cup D_{n-2}$, hence $L$ is isotopic to a diagram described by Figure \ref{fig_cable}.
\end{proof}

\subsection{Seifert surfaces of $L'$}

We recall the following property of fibered links.

\begin{Lemma}
\label{lem_connect_sum_fiber}
Suppose $L_1$ and $L_2$ are two oriented fibered links with oriented Seifert surfaces $S_1$ and $S_2$ respectively. Let $f_1:S_1\to S_1$ and $f_2:S_2\to S_2$ be the monodromies. Take $p_1\in L_1$, $p_2\in L_2$, and form the connected sum $L_1\# L_2$ and the boundary connected sum $S_1\#_b S_2$ with respect to $(p_1,p_2)$. Then $L_1\# L_2$ is a fibered link with  Seifert surface $S_1\#_b S_2$ and monodromy $f_1\#_b f_2$.
\end{Lemma}

\begin{proof}
Given a compact surface $S$ with boundary, and given a diffeomorphism $f:S\to S$ that restricts to the identity on a neighborhood of $\partial S$, define
$$
\mathcal{M}_f := S\times [0,1] / \sim,
$$
where $\sim$ is defined by 
$(x,0)\sim (f(x),1)$ for $x\in S$, and $(x,t_1)\sim (x,t_2)$ for $x\in \partial S$, $t_1, t_2\in [0,1]$. By the assumptions of the lemma,
\begin{align*}
\mathcal{M}_{f_1} & \cong S^3, \\
\mathcal{M}_{f_2} & \cong S^3,
\end{align*}
and the images of $\partial S_1$ and $\partial S_2$ are isotopic to $L_1$ and $L_2$ respectively.
Therefore,
$$
\mathcal{M}_{f_1\#_b f_2} \cong \mathcal{M}_{f_1} \# \mathcal{M}_{f_2}\cong S^3,
$$
and the image of $\partial (S_1\#_b S_2)$ is isotopic to $L_1\# L_2$.
\end{proof}

\begin{figure}  
  \includegraphics[width=0.3\linewidth]{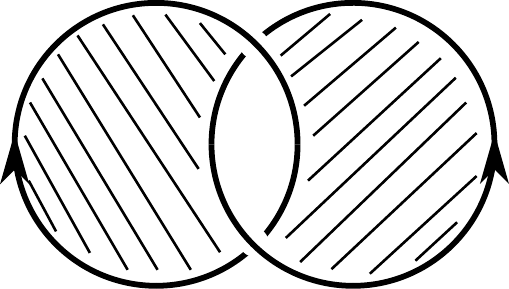}
  \caption{Seifert surface of the Hopf link with linking number $-1$.}
  \label{fig_hopf1} 
  \vspace{\baselineskip}
  \includegraphics[width=0.3\linewidth]{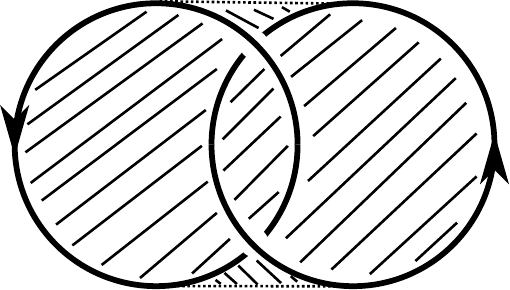}
  \caption{Seifert surface of the Hopf link with linking number $1$.}
  \label{fig_hopf2}
\end{figure}

Notice that the Hopf link is fibered. Depending on the orientations of the components, the corresponding Seifert surface is given by Figure \ref{fig_hopf1} or Figure \ref{fig_hopf2}. Both Seifert surfaces are diffeomorphic to the annulus, and the monodromies are  Dehn twists along the core circles. 

\begin{figure}  
  \includegraphics[width=0.8\linewidth]{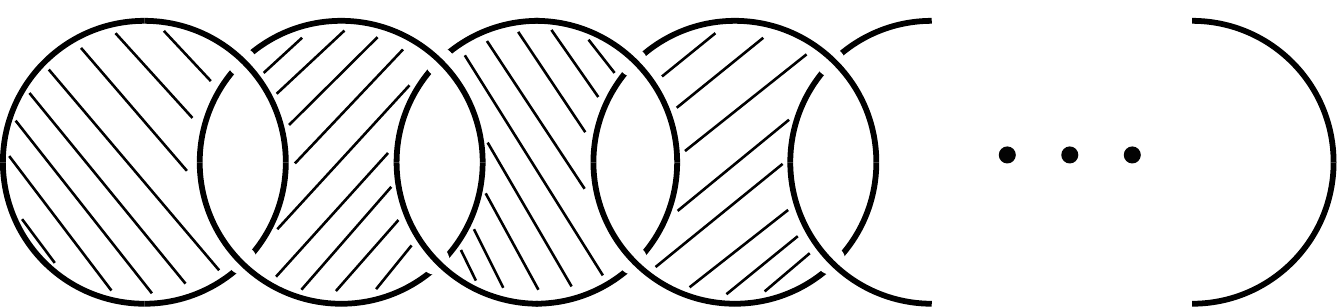}
  \caption{The Seifert surface $S_1$.}
  \label{fig_fiber1}
  \vspace{\baselineskip}
  \includegraphics[width=0.8\linewidth]{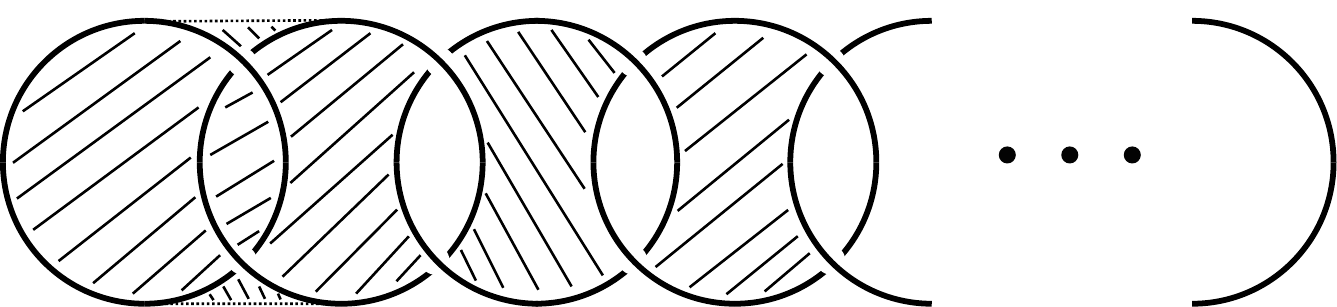}
  \caption{The Seifert surface $S_2$.}
  \label{fig_fiber2}
\end{figure}

Let $S_1$ and $S_2$ be the Seifert surfaces of $L'$ given by Figure \ref{fig_fiber1} and Figure \ref{fig_fiber2} respectively.
By Lemma \ref{lem_connect_sum_fiber}, the link $L'$ is fibered with respect to both $S_1$ and $S_2$. For each $j=1,2$, endow the components $K_1,\cdots,K_{n-1}$ with the boundary orientation of $S_j$ and choose an arbitrary orientation for $K_n$, then the algebraic intersection number of $K_n$ and $S_j$ is equal to the sum of linking numbers $\sum_{i=1}^{n-1} \lk(K_n, K_i)$. Therefore, part (3) of Condition \ref{cond_cycle} implies that there exists exactly one element $j\in \{1,2\}$ such that the algebraic intersection number of $S_j$ and $K_n$ is zero. The main result of this subsection is the following proposition.

\begin{Proposition}\label{Prop_disjoint_Seifert_surface}
Suppose $j\in\{1,2\}$ and the algebraic intersection number of $K_n$ with $S_j$ is zero. Then there exists a knot $K_n'$, such that $K_n'$ is disjoint from $S_j$, and $K_n$ is isotopic to $K_n'$ in $\bR^3-L'$. 
\end{Proposition}

Before proving Proposition \ref{Prop_disjoint_Seifert_surface}, we need to prove some results on instanton Floer homology.
Let $U$ be an unknot included in a 3-ball which is disjoint from $L'$,
let $m_i$ be a small meridian circle around $K_i$ ($1\le i\le n-1$) and $u_i$ be a small arc joining $K_i$ and $m_i$.  
\begin{Lemma}\label{lem_dim_inst_L'_earrings}
We have
\begin{equation}\label{eq_rank_I_L'_earring}
\dim_\bC \II(S^3,L'\cup\bigcup_{i=1}^{n-1}m_i,\sum_{i=1}^{n-1}u_i)=2^{2n-4},
\end{equation}
\begin{equation}\label{eq_rank_I_L_earring}
\dim_\bC  \II(S^3,L\cup\bigcup_{i=1}^{n-1}m_i,\sum_{i=1}^{n-1}u_i)=2^{2n-3},
\end{equation}
and
\begin{equation}\label{eq_Rrank_I_L'U}
 \II(S^3,L'\cup\bigcup_{i=1}^{n-1}m_i\cup U,\sum_{i=1}^{n-1}u_i;\Gamma_U)\cong\mathcal{R}^{2^{2n-3}},
\end{equation}
where $\Gamma_U$ is the local system associated with $U$.
\end{Lemma}
\begin{proof}
Pick a crossing between $m_1$ and $K_1$ and apply Kronheimer-Mrowka's unoriented skein exact triangle in \cite[Section 6]{KM:Kh-unknot}, we obtain a 3-cyclic exact sequence
\begin{align*}
 \cdots&\to  \II(S^3,L'\cup\bigcup_{i=1}^{n-1}m_i,\sum_{i=1}^{n-1}u_i)\to \II(S^3,L'\cup\bigcup_{i=2}^{n-1}m_i,\sum_{i=2}^{n-1}u_i) \\
       &\to \II(S^3,L'\cup\bigcup_{i=2}^{n-1}m_i,\sum_{i=2}^{n-1}u_i)\to \cdots
\end{align*}
See \cite[Section 3]{Xie-earring} for more details. The above exact triangle implies
$$
\dim_\bC  \II(S^3,L'\cup\bigcup_{i=1}^{n-1}m_i,\sum_{i=1}^{n-1}u_i) \le 2
\dim_\bC   \II(S^3,L'\cup\bigcup_{i=2}^{n-1}m_i,\sum_{i=2}^{n-1}u_i).
$$ 
Repeating this argument to the other meridians, we obtain
\begin{align*}
& \dim_\bC  \II(S^3,L'\cup\bigcup_{i=1}^{n-1}m_i,\sum_{i=1}^{n-1}u_i)
\\
\le ~ 
& 2^{n-2}
 \dim_\bC  \II(S^3,L'\cup m_{n-1},u_{n-1})=2^{n-2}\dim_\bC \II^\natural(L', p).
\end{align*}
By Proposition \ref{Prop_Kh_rank_sublink} and Proposition \ref{Prop-rank_I^natural},
$$
\dim_\bC \II^\natural(L', p)=2^{n-2},
$$
therefore
\begin{equation}\label{eq_upper_bound_I_L'_earring}
\dim_\bC  \II(S^3,L'\cup\bigcup_{i=1}^{n-1}m_i,\sum_{i=1}^{n-1}u_i)\le 2^{2(n-2)}.
\end{equation}
A similar earrings-removal argument yields 
\begin{equation}\label{eq_upper_bound_I_L_earring}
\dim_\bC  \II(S^3,L\cup\bigcup_{i=1}^{n-1}m_i,\sum_{i=1}^{n-1}u_i)\le 2^{n-2}\dim_\bC \II^\natural(L, p)= 2^{2n-3}.
\end{equation}
We recall some properties of the instanton knot Floer homology $\KHI$ for oriented links, which was introduced in \cite[Definition 2.4]{KM:Alexander}. Given an oriented link $M\subset S^3$, the homology group $\KHI(M)$
carries an Alexander $\bZ$-grading and a homological $\bZ/2$-grading. The rank of $\KHI(M)$ does not depend on the orientation of $M$.
 We use $\KHI(M,i)$ to denote the summand of $\KHI(M)$ with Alexander degree $i$, and use
$\chi(\KHI(M,i))$ to denote its Euler characteristic with respect to the homological grading. Recall that we always take coefficients in $\bC$ for instanton Floer homology in this article.
According to \cite[Theorem 3.6 and (14)]{KM:Alexander}, we have 
$$
\sum_i \chi(\KHI(M,i))t^i=\pm (t^{1/2}-t^{-1/2})^{|M|-1}\Delta_{M}(t)
$$
where $\Delta_{M}(t)$ denotes the single-variable Alexander polynomial of $M$. 
Notice that the Alexander polynomial for $L'$ satisfies $|\Delta_{L'}(-1)|=2^{n-2}$ for every orientation of $L'$. Therefore, taking $M=L'$, we have
$$
\dim_\bC\KHI(L')\ge 2^{n-2} |\Delta_{L'}(-1)| = 2^{2n-4}.
$$
By \cite[Proposition 5.1]{Xie-earring}, we have
\begin{equation}\label{eq_lower_bound_I_L'_earring}
\dim_\bC \II(S^3,L'\cup\bigcup_{i=1}^{n-1}m_i,\sum_{i=1}^{n-1}u_i)=\dim_\bC \KHI(L')\ge 2^{2n-4}.
\end{equation}
Equations \eqref{eq_upper_bound_I_L'_earring} and \eqref{eq_lower_bound_I_L'_earring} imply \eqref{eq_rank_I_L'_earring}.

Consider the two admissible triples 
$$
(S^3,L'\cup\bigcup_{i=1}^{n-1}m_i\cup U,\sum_{i=1}^{n-1}u_i)
,~(S^1\times S^2, S^1\times \{p_1,p_2\},v)
$$ 
where $v$ is an arc joining the two components of $S^1\times \{p_1,p_2\}$.
Let $N(K_1)$ be a small tubular neighborhood of $K_1$ in the first triple, and deform $U$ into $N(K_1)$ by an isotopy. Let $N(S^1\times \{p_1\})$ be a small tubular neighborhood of $S^1\times \{p_1\}$ in the second triple.
Cut out $ N(K_1)$ and $ N(S^1\times \{p_1\})$, exchange them, and glue back, we obtain two new 
triples
$$
(S^3,L'\cup\bigcup_{i=1}^{n-1}m_i,\sum_{i=1}^{n-1}u_i)
,~(S^1\times S^2, S^1\times \{p_1,p_2\}\cup U',v),
$$ 
where $U'$ is an unknot included in a $3$-ball disjoint from $S^1\times\{p_1,p_2\}$.
By the torus excision theorem and
and the definition of $\AHI$, we have
\begin{align}
& \II(S^3,L'\cup\bigcup_{i=1}^{n-1}m_i\cup U,\sum_{i=1}^{n-1}u_i;\Gamma_U)\otimes_\mathcal{R}\AHI(\emptyset;\Gamma)
\nonumber
\\ \cong & \II(S^3,L'\cup\bigcup_{i=1}^{n-1}m_i,\sum_{i=1}^{n-1}u_i;\Gamma_\emptyset)\otimes_\mathcal{R}\AHI(U';\Gamma),
\label{eqn_L'_U_local_system_torus_excision}
\end{align}
where  $\Gamma_\emptyset$ is the trivial local system with coefficient $\mathcal{R}$.
By \eqref{eq-constant-coefficients} and Example \ref{eg-Uk-Kl-Gamma}, $\AHI(\emptyset;\Gamma)\cong \mathcal{R}$ and $\AHI(U';\Gamma)\cong \mathcal{R}^2$.
By \eqref{eq-constant-coefficients} and \eqref{eq_rank_I_L'_earring},
$$
\II(S^3,L'\cup\bigcup_{i=1}^{n-1}m_i,\sum_{i=1}^{n-1}u_i;\Gamma_\emptyset)\cong 
\II(S^3,L'\cup\bigcup_{i=1}^{n-1}m_i,\sum_{i=1}^{n-1}u_i)\otimes_{\mathbb{C}}\mathcal{R}\cong
\mathcal{R}^{ 2^{2n-4}}.
$$
Therefore by \eqref{eqn_L'_U_local_system_torus_excision}, we have
$$
\II(S^3,L'\cup\bigcup_{i=1}^{n-1}m_i\cup U,\sum_{i=1}^{n-1}u_i;\Gamma_U)\cong 
\mathcal{R}^{ 2^{2n-4}}\otimes_\mathcal{R}\mathcal{R}^{ 2}\cong \mathcal{R}^{ 2^{2n-3}}
$$
This completes the proof of \eqref{eq_Rrank_I_L'U}.

Let $\Gamma_{K_n}$ be the local system on 
$\mathcal{R}(S^3,L\cup\bigcup_{i=1}^{n-1}m_i,\sum_{i=1}^{n-1}u_i)$ associated with $K_n$.
By Corollary \ref{LocallyFreePart-inst} and the universal coefficient theorem, we have
\begin{align*}
\dim_\bC \II(S^3,L\cup\bigcup_{i=1}^{n-1}m_i,\sum_{i=1}^{n-1}u_i)&\ge
\rank_\mathcal{R} \II(S^3,L\cup\bigcup_{i=1}^{n-1}m_i,\sum_{i=1}^{n-1}u_i;\Gamma_{K_n})  \\
&=\rank_\mathcal{R} \II(S^3,L'\cup\bigcup_{i=1}^{n-1}m_i\cup U,\sum_{i=1}^{n-1}u_i;\Gamma_{U}) \\
&=2^{2n-3}.
\end{align*}
The above inequality together with \eqref{eq_upper_bound_I_L_earring} implies \eqref{eq_rank_I_L_earring}.
\end{proof}

Choose $j\in\{1,2\}$ such that the algebraic intersection number of $S_j$ and $K_n$ is zero. Choose an orientation of $S_j$, and endow $L'$ with the boundary orientation. For each $i=1,\cdots,n-1$, let $N(K_i)$ be a sufficiently small tubular neighborhood of $K_i$ that is disjoint from $m_i$, and apply a Dehn surgery on $N(K_i)$ that glues the meridian of $N(K_i)$ to $S_j\cap \partial N(K_i)$.  
Since $S^3-L'$ is fibered over $S^1$ with fiber $S_j$, the manifold obtained from the Dehn surgeries is fibered over $S^1$ with fiber $S^2$. Since the orientation-preserving mapping class group of $S^2$ is trivial, the resulting manifold is diffeomorphic to $S^1\times S^2$ with the product fibration. Let $\hat K_1,\cdots, \hat K_n, \hat m_1,\cdots, \hat m_{n-1},\hat U$ be the images of $K_1,\cdots,K_n,m_1,\cdots,m_{n-1},U$ respectively after the surgery. Let $\hat L' := \hat K_1\cup\cdots\cup \hat K_{n-1}$ be the image of $L'$, let $m:=m_1\cup\cdots \cup m_{n-1}$ be the union of the earrings, and let $\hat m:=\hat m_1\cup\cdots\cup \hat m_{n-1}$ be the image of $m$. We further require that the surgery on $N(K_i)$ fixes $u_i\cap \partial N(K_i)$ for all $i=1,\cdots,n-1$, hence the image of $u_i$ is an arc connecting $\hat K_i$ and $\hat m_i$, and we denote the image of $u_i$ by $\hat u_i$. 

Given a $\bC$-vector space $V$, a linear map $f:V\to V$, and $\lambda\in \mathbb{C}$, we will use
$E(V,f,\lambda)$ to denote the generalized eigenspace of $f$ with eigenvalue $\lambda$. 
\begin{Lemma}\label{lem_eigenspace_BKn=BU}
For all $\lambda\in \bC$,
we have
\begin{align*}
& \dim_\bC E(\II(S^1\times S^2, \hat L'\cup \hat m \cup \hat K_n,\sum_{i=1}^{n-1}\hat u_i),\muu(S^2),\lambda)\\
\cong ~
& \dim_\bC E(\II(S^1\times S^2, \hat L' \cup  \hat m \cup \hat U,\sum_{i=1}^{n-1}\hat u_i),\muu(S^2),\lambda),
\end{align*}
where the operator $\muu(S^2)$ is the $\mu$-map defined by $\{p\}\times S^2\subset S^1\times S^2$ for an arbitrary $p\in S^1$.
\end{Lemma}
\begin{proof}
By the torus excision theorem
and Lemma \ref{lem_dim_inst_L'_earrings}, we have
\begin{equation}\label{eq_rank_I_BKn}
\II(S^1\times S^2, \hat L'\cup \hat m \cup \hat K_n,\sum_{i=1}^{n-1}\hat u_i)\cong \II(S^3,L\cup m,\sum_{i=1}^{n-1}u_i)\cong \bC^{2^{2n-3}},
\end{equation}
and
\begin{equation}\label{eq_Rrank_BU}
\II(S^1\times S^2, \hat L' \cup  \hat m  \cup \hat U,\sum_{i=1}^{n-1}\hat u_i;\Gamma_{\hat U})\cong 
\II(S^3,L'\cup m \cup U,\sum_{i=1}^{n-1}u_i;\Gamma_U)\cong \mathcal{R}^{ 2^{2n-3}},
\end{equation}
where $\Gamma_{\hat U}$ is the local system associated with $\hat U$, and $\Gamma_U$ is the local system associated with $ U$.
Since the algebraic intersection number of $K_n$ and $S_j$ is zero, we conclude that $\hat K_n$ is homotopic to $\hat U$ in $S^1\times S^2-\bigcup_{i=1}^{n-1} \hat u_i$. Let $\Gamma_{\hat K_n}$ be the local system associated with $\hat K_n$. By Proposition \ref{Prop-Gammah-homotopy_invariance},
we have
\begin{equation}
\label{eqn_BKn=BU_with_local_system}
\II(S^1\times S^2, \hat L'\cup \hat m \cup \hat K_n,\sum_{i=1}^{n-1}\hat u_i; \Gamma_{\hat K_n}(h))\cong \II(S^1\times S^2, \hat L' \cup  \hat m  \cup \hat U,\sum_{i=1}^{n-1}\hat u_i; \Gamma_{\hat U}(h))
\end{equation} 
for every $h\in \bC-\{0\}$ satisfying $(1-h^2)\theta(h)\neq 0$, and this isomorphism commutes with $\muu(S^2)$. As a consequence, for every $\lambda\in \bC$ and $h\in \bC-\{0\}$ satisfying $(1-h^2)\theta(h)\neq 0$, we have
\begin{align}
& \dim_\bC E(\II(S^1\times S^2, \hat L'\cup \hat m \cup \hat K_n,\sum_{i=1}^{n-1}\hat u_i; \hat \Gamma_{\hat K_n}(h)),\muu(S^2),\lambda)\nonumber \\
\cong ~
& \dim_\bC E(\II(S^1\times S^2, \hat L' \cup  \hat m \cup \hat U,\sum_{i=1}^{n-1}\hat u_i;\Gamma_{\hat U}(h)),\muu(S^2),\lambda). \label{eqn_eigenspace_BKn=BU_with_local_system}
\end{align}
When $h(1-h^2)\theta(h)\neq 0$, the universal coefficient theorem and \eqref{eq_Rrank_BU}, \eqref{eqn_BKn=BU_with_local_system} imply
\begin{equation*}
\II(S^1\times S^2, \hat L'\cup \hat m \cup \hat K_n,\sum_{i=1}^{n-1}\hat u_i;\Gamma_{\hat K_n}(h))\cong \II(S^1\times S^2, \hat L' \cup  \hat m  \cup \hat U,\sum_{i=1}^{n-1}\hat u_i; \Gamma_{\hat U}(h))
\cong \bC^{2^{2n-3}}.
\end{equation*} 
On the other hand, notice that when $h=1$, the local systems $\Gamma_{\hat K_n}(h)$ and $\Gamma_{\hat U}(h)$ become the trivial system with coefficient $\bC$, hence by \eqref{eq_rank_I_BKn} and \eqref{eq_Rrank_BU}, 
\begin{equation*}
\II(S^1\times S^2, \hat L'\cup \hat m \cup \hat K_n,\sum_{i=1}^{n-1}\hat u_i;\Gamma_{\hat K_n}(1))\cong \II(S^1\times S^2, \hat L' \cup  \hat m  \cup \hat U,\sum_{i=1}^{n-1}\hat u_i; \Gamma_{\hat U}(1))
\cong \bC^{2^{2n-3}}.
\end{equation*} 
Therefore the desired result follows from \eqref{eqn_eigenspace_BKn=BU_with_local_system} by taking the limit $h\to 1$ and invoking Part (3) of Proposition \ref{chain-complex-limit} .
\end{proof}

Notices that $\hat L'\cup \hat m$ is a braid in $S^1\times S^2$. In fact, the projection of $\hat L'\cup \hat m$ to $S^1$ is a diffeomorphism on each component. Therefore, after an isotopy, we may write $S^1\times S^2$ as $A_0\cup_{S^1\times S^1} A_1$, where $A_0,A_1$ are diffeomorphic to $S^1\times D^2$, such that
\begin{enumerate}
\item $\hat K_1$, $\hat m_1$ are included in $A_0$ and are given by $S^1\times \{p_1\}$ and $S^2\times \{p_2\}$ with $p_1,p_2\in D^2$,
\item $\hat u_1$ is an arc connecting $\hat K_1$ and $\hat m_1$, and $\hat u_1$ is included in $A_0$,
\item $\hat K_2,\cdots,\hat K_n,\hat m_2,\cdots,\hat m_{n-1},\hat U$ are included in $A_1$.
\end{enumerate}

 Let 
\begin{equation}\label{eqn_define_L0_with_earrings}
\mathcal{L}_0:= \bigcup_{i=2}^{n-1} \hat K_i\cup \bigcup_{i=2}^{n-1} \hat m_i \cup \hat K_n,
\end{equation}
\begin{equation}\label{eqn_define_L1_with_earrings}
\mathcal{L}_1:= \bigcup_{i=2}^{n-1} \hat K_i\cup \bigcup_{i=2}^{n-1} \hat m_i \cup \hat U,
\end{equation}
then $\mathcal{L}_0$ and $\mathcal{L}_1$ are two annular links in $A_1$. By the definition of annular instanton Floer homology, we have
\begin{align}
\AHI(\mathcal{L}_0) & \cong \II(S^1\times S^2,\hat L'\cup \hat m \cup \hat K_n,\hat u_1), 
\label{eqn_AHI_L0}
\\
\AHI(\mathcal{L}_1) & \cong \II(S^1\times S^2,\hat L'\cup \hat m \cup \hat U,\hat u_1).
\label{eqn_AHI_L1}
\end{align}

\begin{Lemma}\label{lem_remove_w2_from_u2_to_un_L1}
Assume there exists a connected oriented Seifert surface $S\subset S^3$ of $L'$, such that $S$ is compatible with the orientation of $L'$, and $S$ has genus $g$ and is disjoint from $K_n$. Then we have
\begin{align*}
& \dim_\bC E(\II(S^1\times S^2, \hat L'\cup \hat m \cup \hat K_n,\sum_{i=1}^{n-1}\hat u_i),\muu(S^2),2g+2n-4) \\
= ~ 
& \dim_\bC E(\II(S^1\times S^2, \hat L'\cup \hat m \cup \hat K_n, \hat u_1),\muu(S^2),2g+2n-4),
\end{align*}
and
$$
E(\II(S^1\times S^2, \hat L'\cup \hat m \cup \hat K_n,\sum_{i=1}^{n-1}\hat u_i),\muu(S^2),i)=0 
$$
for all integers $i>2g+2n-4$.
\end{Lemma}

\begin{proof}
After an isotopy, we may assume that $S$ intersects each $m_i$ transversely at one point. 
The image of $S-\bigcup_{i=1}^{n-1}N(K_i)$ in $S^1\times S^2$ is a connected surface with $n-1$ boundary components, where the boundary components are give by the meridians of $\hat K_1,\cdots,\hat K_{n-1}$. Therefore we can glue disks to the boundary of the image of $S-\bigcup_{i=1}^{n-1}N(K_i)$ and obtain a connected closed surface in $S^1\times S^2$ with genus $g$, that is disjoint from $\hat K_n$ and intersects each of $\hat K_1,\cdots,\hat K_{n-1},\hat m_1,\cdots,\hat m_{n-1}$ transversely at one point. Denote this surface by $\hat S$. After a further isotopy, 
we may assume that 
the arcs $\hat m_1,\cdots,\hat m_{n-1}$ lie on $\hat S$. 

Recall that  $\hat K_1$ and $\hat m_1$ are contained in $A_0\cong S^1\times D^2$ and are given by $S^1\times \{p_1\}$ and $S^2\times \{p_2\}$ for $p_1,p_2\in D^2$. Take a point $p_0\in D^2-\{p_1,p_2\}$, and let $\hat K_0\subset A_0$ be the knot $S^1\times \{p_0\}$. After a further isotopy, we may assume that $\hat S$ intersects $\hat K_0$ transversely at one point.
Let $c$ be a simple closed curve on $D^2$ such that $p_0,p_1$ are inside of $c$ and $p_2$ is outside. let $T_1\subset A_0$ be the torus given by $T_1 := S^1\times c$.

Notice that $\hat S$ is homologous to the slice of $S^2$ in $S^1\times S^2$, therefore we have $\muu(\hat S) = \muu(S^2)$. The surface $\hat S$ intersects $\hat L'\cup \hat m \cup \hat K_n\cup \hat K_0$ transversely at $2n-1$ points.
Apply \cite[Theorem 6.1]{XZ:excision} to the surface $\hat S$, we deduce that the set of eigenvalues of $ \muu(S^2)$ on 
$$\II(S^1\times S^2, \hat L'\cup \hat m \cup \hat K_n\cup \hat K_1,\sum_{i=1}^{n-1}\hat u_i)$$ 
is included in 
$$
\{-(2g+2n-3), -(2g+2n-5),\cdots,(2g+2n-5),(2g+2n-3)\}.
$$

Consider the triple $(S^1\times S^2,S^1\times \{q_1,q_2\}, v)$, where $q_1,q_2\in S^2$ and $v$ is an arc connecting $S^1\times \{q_1\}$ and $S^1\times \{q_2\}$. Let $T_2$ be a torus given by the boundary of a tubular neighborhood of $S^1\times \{q_1\}$. Recall that $T_1\subset A_0$ is the torus $S^1\times c$ as defined above.
Applying the torus excision on the triple
$$(S^1\times S^2, \hat L'\cup \hat m \cup \hat K_n\cup \hat K_0,\sum_{i=1}^{n-1}\hat u_i) \sqcup (S^1\times S^2,S^1\times \{q_1,q_2\}, v)$$ along $T_1\cup T_2$ yields
 \begin{align*}
& \dim_\bC E(\II(S^1\times S^2, \hat L'\cup \hat m \cup \hat K_n\cup \hat K_0,\sum_{i=1}^{n-1}\hat u_i),\muu(S^2),\lambda) \\
= ~ 
& \dim_\bC E(\II(S^1\times S^2, \hat L'\cup \hat m \cup \hat K_n,\sum_{i=1}^{n-1}\hat u_i),\muu(S^2),\lambda-1)
\\
& + \dim_\bC E(\II(S^1\times S^2, \hat L'\cup \hat m \cup \hat K_n,\sum_{i=1}^{n-1}\hat u_i),\muu(S^2),\lambda+1)
\end{align*}
for all $\lambda\in \bC$.
Therefore we have
 \begin{align}
& \dim_\bC E(\II(S^1\times S^2, \hat L'\cup \hat m \cup \hat K_n,\sum_{i=1}^{n-1}\hat u_i),\muu(S^2),2g+2n-4)\nonumber \\
= ~ 
& \dim_\bC E(\II(S^1\times S^2, \hat L'\cup \hat m \cup \hat K_n\cup \hat K_0,\sum_{i=1}^{n-1}\hat u_i),\muu(S^2),2g+2n-3).
\label{eqn_remove_w2_from_u2_to_un_1}
\end{align}
and
\begin{equation}\label{eqn_Seifert-surface-bound-eigenvalue}
\dim_\bC E(\II(S^1\times S^2, \hat L'\cup \hat m \cup \hat K_n,\sum_{i=1}^{n-1}\hat u_i),\muu(S^2),i)=0.
\end{equation}
for all integers $i>2g+2n-4$.

Similarly, applying torus excision on the triple
$$(S^1\times S^2, \hat L'\cup \hat m \cup \hat K_n\cup \hat K_0,\hat u_1) \sqcup (S^1\times S^2,S^1\times \{q_1,q_2\}, v)$$ along $T_1\cup T_2$ yields
\begin{align}
& \dim_\bC E(\II(S^1\times S^2, \hat L'\cup \hat m \cup \hat K_n\cup \hat K_0,\hat u_1),\muu(S^2),2g+2n-3) 
\nonumber
\\
= ~ 
& \dim_\bC E(\II(S^1\times S^2, \hat L'\cup \hat m \cup \hat K_n,\hat u_1),\muu(S^2),2g+2n-4).\label{eqn_remove_w2_from_u2_to_un_2}
\end{align}

Let $\{z_1,\cdots,z_{2n-1}\}\subset \hat S$ be the intersection of $\hat S$ with $\hat L'\cup \hat m \cup \hat K_n\cup \hat K_0$. Apply the singular excision theorem \cite[Theorem 6.4]{XZ:excision} on the following triple
$$
(S^1\times S^2, \hat L'\cup \hat m \cup \hat K_n,\hat u_1)  \sqcup (S^1\times\hat S, S^1\times \{z_1,\cdots,z_{2n-1}\}, \sum_{i=2}^{n-1} \hat u_i)
$$
along 
$\hat S$ in the first component, and a slice of $\hat S$ in the second component, and invoke \cite[Proposition 6.7]{XZ:excision}, we obtain
\begin{align}
& \dim_\bC E(\II(S^1\times S^2, \hat L'\cup \hat m \cup \hat K_n\cup \hat K_0,\sum_{i=1}^{n-1}\hat u_i),\muu(\hat S_0),2g+2n-3) 
\nonumber
\\
= ~ 
& \dim_\bC E(\II(S^1\times S^2, \hat L'\cup \hat m \cup \hat K_n\cup \hat K_0,\hat u_1),\muu(\hat S_0),2g+2n-3).
\label{eqn_remove_w2_from_u2_to_un_3}
\end{align}
Since $\muu(\hat S_0) = \muu(S^2)$, the first part of the lemma is proved by \eqref{eqn_remove_w2_from_u2_to_un_1}, \eqref{eqn_remove_w2_from_u2_to_un_2}, and \eqref{eqn_remove_w2_from_u2_to_un_3}. The second part of the lemma is proved by \eqref{eqn_Seifert-surface-bound-eigenvalue}.
\end{proof}

\begin{Lemma}\label{lem_two_versions_of_minimal_genus}
Let $\mathcal{L}_0\subset A_1$ be the annular link defined by \eqref{eqn_define_L0_with_earrings}. Suppose there exists a meridional surface $S$ (cf. Definition \ref{def_meridional_surface}) with genus $g$ such that $S$ intersects $\mathcal{L}_0$ transversely at $m$ points. Then there exsits a connected Seifert surface $\hat S$ of $L'$ in $S^3$, such that $\hat S$ is compatible with the orientation of $L'$ and is disjoint from $K_n$, and the genus of $\hat S$ is equal to $g+m/2-n+2$.
\end{Lemma}

\begin{proof}
Suppose there is a component $K$ of $\mathcal{L}_0$ whose intersection with $S$ has different signs, then we can attach a tube to $S$ along a segment of $K$ to decrease the value of $m$ by $2$ and increase the value of $g$ by $1$. Repeating this process until the number of intersection points of $S$ with each component of $\mathcal{L}_1$ equals the absolute value of their algebraic intersection number, we obtain a new meridional surface $S'\subset A_1$, such that 
\begin{enumerate}
\item the genus of $S'$ equals $g+(m-2n+4)/2$,
\item $S'$ intersects each of  $\hat K_2,\cdots,\hat K_{n-1},\hat m_2,\cdots,\hat m_{n-1}$ transversely at one point,
\item $S'$ is disjoint from $\hat K_n$.
\end{enumerate}

Since $S'$ is a meridional surface, by attaching a disk in $A_0$, we can complete the surface $S'$  to a closed surface with the same genus that intersects each of $\hat K_1,\cdots,\hat K_{n-1}$ transversely at one point and is disjoint from $\hat K_n$, therefore it gives rise to a Seifert surface of $L'$ in $S^3$ with the same genus that is disjoint from $K_n$, hence the lemma is proved. 
\end{proof}

\begin{Corollary}\label{cor_instanton_detects_seifert_genus_Kn}
Let $g_0$ be the smallest integer with the following property. There exists a connected oriented Seifert surface $S\subset S^3$ of $L'$ that is compatible with the orientation of $L'$, such that $S$ has genus $g_0$ and is disjoint from $K_n$. Then we have
$$
 \dim_\bC E(\II(S^1\times S^2, \hat L'\cup \hat m \cup \hat K_n,\sum_{i=1}^{n-1}\hat u_i),\muu(S^2),2g_0+2n-4)\neq 0
$$
and
$$
 E(\II(S^1\times S^2, \hat L'\cup \hat m \cup \hat K_n,\sum_{i=1}^{n-1}\hat u_i),\muu(S^2),i)=0. 
$$
for all integers $i>2g_0+2n-4$.
\end{Corollary}
\begin{proof}
By Theorem \ref{Theorem-2g+n} and \eqref{eqn_AHI_L0}, there are integers $g,m$, such that there exists a meridional surface in $A_1$ with genus $g$ and intersects $\mathcal{L}_1$ transversely at $m$ points, such that 
\begin{equation}\label{eqn_dim_positive_w2=u1_gm}
\dim_\bC E(\II(S^1\times S^2, \hat L'\cup \hat m \cup \hat K_n, \hat u_1),\muu(S^2),2g+m)> 0.
\end{equation}
Let $g':=g+m/2-n+2$.
By Lemma \ref{lem_two_versions_of_minimal_genus}, we may choose $g,m$ such that there exists a connected oriented Seifert surface of $L'$ in $S^3$ that is compatible with the orientation of $L'$, has genus $g'$, and is disjoint from $K_n$.
Since $2g+m=2g'+2n-4$, 
by Lemma \ref{lem_remove_w2_from_u2_to_un_L1} and \eqref{eqn_dim_positive_w2=u1_gm}, we have
\begin{align}
& \dim_\bC E(\II(S^1\times S^2, \hat L'\cup \hat m \cup \hat K_n,\sum_{i=1}^{n-1}\hat u_i),\muu(S^2),2g'+2n-4) 
\nonumber
\\
= ~ 
& \dim_\bC E(\II(S^1\times S^2, \hat L'\cup \hat m \cup \hat K_n, \hat u_1),\muu(S^2),2g+m)>0.
\label{eqn_dim_positive_w2=u1_g'n}
\end{align}
By the definition of $g_0$, we have $g_0\le g'$.
On the other hand, the second part of Lemma \ref{lem_remove_w2_from_u2_to_un_L1} implies that
\begin{equation}\label{eqn_dim_equal_zero_w2=u1_g0}
\dim_\bC E(\II(S^1\times S^2, \hat L'\cup \hat m \cup \hat K_n,\sum_{i=1}^{n-1}\hat u_i),\muu(S^2),i)  = 0
\end{equation}
for all integers $i > 2g_0+2n-4$. Therefore by \eqref{eqn_dim_positive_w2=u1_g'n} and \eqref{eqn_dim_equal_zero_w2=u1_g0}, we have $g_0\ge g'$. In conclusion, we must have $g=g_0$, and the lemma follows from \eqref{eqn_dim_positive_w2=u1_g'n} and \eqref{eqn_dim_equal_zero_w2=u1_g0}.
\end{proof}

Replacing $\hat K_n$ with $\hat U$ in the previous arguments, we also have the following lemma.
\begin{Lemma}\label{lem_instanton_detects_seifert_genus_U}
Let $g_1$ be the smallest integer with the following property. There exists a connected oriented Seifert surface $S\subset S^3$ of $L'$ that is compatible with the orientation of $L'$, such that $S$ has genus $g_1$ and is disjoint from $U$. Then we have
$$
 \dim_\bC E(\II(S^1\times S^2, \hat L'\cup \hat m \cup \hat U,\sum_{i=1}^{n-1}\hat u_i),\muu(S^2),2g_1+2n-4)\neq 0
$$
and
$$
E(\II(S^1\times S^2, \hat L'\cup \hat m \cup \hat U,\sum_{i=1}^{n-1}\hat u_i),\muu(S^2),i)=0. 
$$
for all integers $i>2g_1+2n-4$. \qed
\end{Lemma}

\begin{proof}[Proof of Proposition \ref{Prop_disjoint_Seifert_surface}]
It is obvious that the minimal genus $g_1$ in Lemma \ref{lem_instanton_detects_seifert_genus_U} is zero, therefore by Lemma \ref{lem_eigenspace_BKn=BU}, Corollary \ref{cor_instanton_detects_seifert_genus_Kn}, and Lemma \ref{lem_instanton_detects_seifert_genus_U}, the genus $g_0$ in Corollary \ref{cor_instanton_detects_seifert_genus_Kn} is also zero.
As a result, there exists a connected oriented Seifert surface $S\subset S^3$ for $L'$ with genus zero that is disjoint from $K_n$ and is compatible with the orientation of $L'$. 
Since the minimal-genus Seifert surface for an oriented fibered link is unique up to isotopy, we conclude that there exists an ambient isotopy of $S^3$ that fixes $L'$ and takes $S$ to $S_j$. This ambient isotopy gives the desired isotopy from $K_n$ to $K_n'$.
\end{proof}

\section{The fundamental group of $\bR^3 - L'$}
\label{sec_fundamental_group}

This section takes a detour to study the properties of $\pi_1(\bR^3-L')$. The results in this section will be used in the proof of the non-existence of $L$. 

By the Wirtinger presentation, $\pi_1(\bR^3-L')$ is generated by 
$$g_1,\cdots,g_{n-1}, g_2',\cdots,g_{n-2}'$$ as shown in Figure \ref{fig_pi1}, where the base point is taken to be a point above and far away from the diagram. 
\begin{figure}  
  \includegraphics[width=0.8\linewidth]{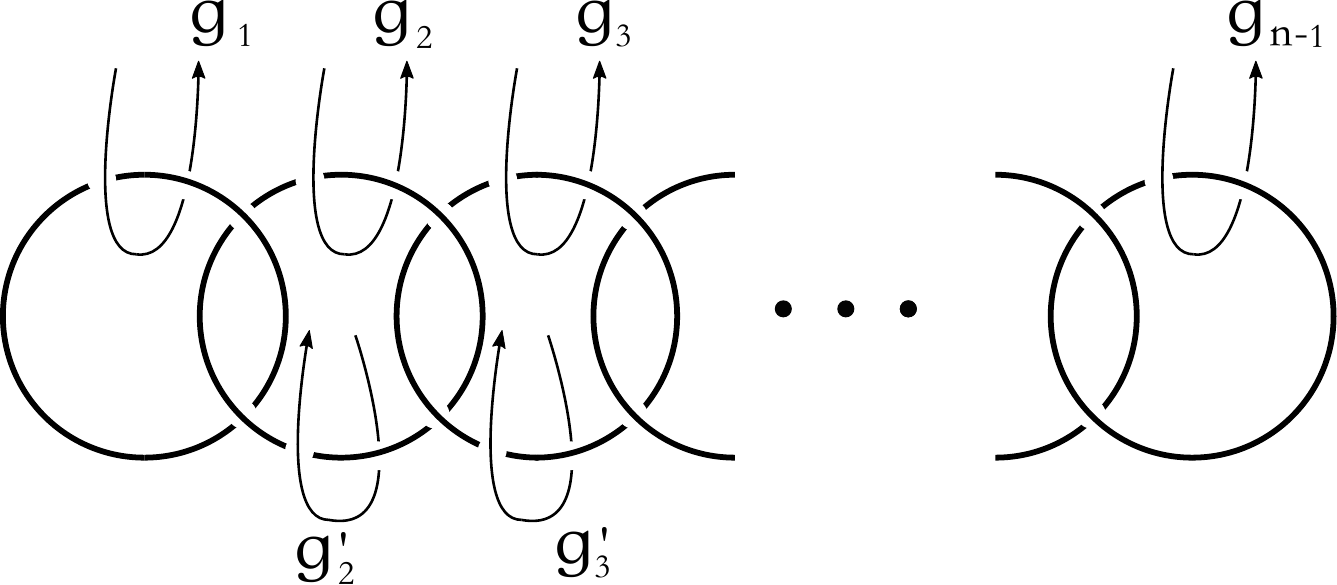}
  \caption{Generators of $\pi_1(\bR^3-L')$.}
  \label{fig_pi1}
\end{figure}
Notice that $g_i'$ and $g_i$ are homotopic relative to the base point because one can shrink $K_1\cup\cdots\cup K_{i-1}$ into a small neighborhood of $K_i$. Therefore $\pi_1(\bR^3-L')$ is generated by $g_1,\cdots , g_{n-1}$, and the Wirtinger presentation gives
$$
\pi_1(\bR^3-L') = \langle g_1,\cdots,g_{n-1} | [g_i,g_{i+1}]=1 \text{ for } i=1,\cdots,n-2\rangle.
$$

To simplify the notation, for the rest of this section we will use $m$ to denote $n-1$, and use $G$ to denote the group $\pi_1(\bR^3-L')$. For $i=1,\cdots,m$, define the set
$$
C_i :=\begin{cases}
\{g_1,g_2,g_1^{-1},g_2^{-1}\} & \text{ if } i=1, \\
\{g_{i-1},g_i,g_{i+1},g_{i-1}^{-1},g_i^{-1},g_{i+1}^{-1}\} & \text{ if } i=2,\cdots,m-1, \\
\{g_{m-1},g_m,g_{m-1}^{-1},g_m^{-1}\} & \text{ if } i=m.
\end{cases}
$$

The first part of this section solves the word problem for $G$.

\begin{Definition}
A \emph{word} is a sequence $(x_1,x_2,\cdots, x_N)$, such that 
$$x_i\in\{g_1,\cdots,g_m,g_1^{-1},\cdots,g_m^{-1}\}\text{ for all }i.$$ 
We call $x_1,\cdots, x_N$ the \emph{letters} of the word $(x_1,x_2,\cdots, x_N)$.
\end{Definition}
\begin{Definition}
The word $(x_1,x_2,\cdots, x_N)$ is called \emph{reduced}, if 
for every pair $u< v$ with $(x_u,x_v)=(g_i,g_i^{-1})$ or $(g_i^{-1},g_i)$, there exists $w\in(u,v)$, such that $x_w\notin C_i$.
\end{Definition}

\begin{Definition}
Define an equivalent relation $\sim$ on the set of words using the following relations as generators:
\begin{align*}
(x_1,\cdots,x_k,g_i,g_{i+1},x_{k+3},\cdots,x_N) &\sim (x_1,\cdots,x_k,g_{i+1},g_{i},x_{k+3},\cdots,x_N)
\\
(x_1,\cdots,x_k,g_i^{-1},g_{i+1},x_{k+3},\cdots,x_N) &\sim (x_1,\cdots,x_k,g_{i+1},g_{i}^{-1},x_{k+3},\cdots,x_N)
\\
(x_1,\cdots,x_k,g_i,g_{i+1}^{-1},x_{k+3},\cdots,x_N) &\sim (x_1,\cdots,x_k,g_{i+1}^{-1},g_{i},x_{k+3},\cdots,x_N)
\\
(x_1,\cdots,x_k,g_i^{-1},g_{i+1}^{-1},x_{k+3},\cdots,x_N) &\sim (x_1,\cdots,x_k,g_{i+1}^{-1},g_{i}^{-1},x_{k+3},\cdots,x_N).
\end{align*}
\end{Definition}
It is straightforward to verify that if two words are equivalent and one of them is reduced, then the other is also reduced. Therefore $\sim$ defines an equivalence relation on the set of reduced words.

Every word $(x_1,x_2,\cdots, x_N)$ represents an element of $G$ by taking the product $x_1x_2\cdots x_N$. By the definition of $G$, equivalent words represent the same element.

\begin{Proposition}
\label{prop_word_prob}
Every element of $G$ is represented by a reduced word. Two reduced words represent the same element if and only if they are equivalent.
\end{Proposition}

\begin{proof}
Define another group $\widetilde G$ as follows. The elements of $\widetilde G$ are the equivalence classes of reduced words. If $(x_1,\cdots,x_N)$ is a reduced word, we use $[x_1,\cdots,x_N]\in \widetilde G$ to denote the equivalence class of $(x_1,\cdots,x_N)$. Let $(x_1,\cdots,x_N)$ be a reduced word, let $y\in \{g_1,\cdots,g_m,g_1^{-1},\cdots,g_m^{-1}\}$. If the word $(x_1,\cdots,x_N,y)$ is reduced, define 
\begin{equation}
\label{eqn_def_prod_one_letter_1}
[x_1,\cdots,x_N]\cdot [y] := [x_1,\cdots,x_N,y].
\end{equation}
If the word $[x_1,\cdots,x_N,y]$ is not reduced, then there exists $u$ such that $x_uy=1$ and every letter in $(x_{u+1},\cdots,x_N)$ is commutative with both $x_u$ and $y$. In this case, define 
\begin{equation}
\label{eqn_def_prod_one_letter_2}
[x_1,\cdots,x_N]\cdot [y] := [x_1,\cdots,x_{u-1},x_{u+1},\cdots,x_N].
\end{equation}
For different choices of $x_u$, the right-hand side of \eqref{eqn_def_prod_one_letter_2} gives the same equivalence class. Moreover, if we take a different representative of $[x_1,\cdots,x_N]$, the right-hand sides of \eqref{eqn_def_prod_one_letter_1} and \eqref{eqn_def_prod_one_letter_2} remain the same. It is also straightforward to verify that if $y_1$ and $y_2$ are commutative generators of $G$, then
$$
[[x_1,\cdots,x_N]\cdot y_1]\cdot y_2= [[x_1,\cdots,x_N]\cdot y_2]\cdot y_1.
$$
Therefore, we obtain a well-defined product operator on $\widetilde G$ defined inductively by 
$$
[x_1,\cdots,x_N]\cdot [y_1,\cdots,y_M] := [[x_1,\cdots,x_N]\cdot [y_1,\cdots,y_{M-1}]]\cdot y_M.
$$
The associativity of the product operator is clear from the definition. For an element $[x_1,\cdots,x_N]\in \widetilde G$, we have $[x_1,\cdots,x_N]\cdot [x_N^{-1},\cdots,x_1^{-1}]=1$, hence every element in $\widetilde G$ has an inverse, therefore $\widetilde G$ is a group.
By the universal property, there is a unique homomorphism $\varphi$ from $G$ to $\widetilde G$ defined by $\varphi(g_i):=[g_i]$. We also have a map $\psi$ from $\widetilde G$ to $G$ defined by 
$$
\psi ( [x_1,\cdots,x_N]) := x_1\cdots x_N.
$$
Since
$$
\psi ( [x_1,\cdots,x_N]\cdot [y_1,\cdots,y_M]) = \psi ( [x_1,\cdots,x_N])\cdot \psi ([y_1,\cdots,y_M]),
$$
the map $\psi$ is a group homomorphism. It is obvious from the definitions that $\varphi$ and $\psi$ are inverse to each other, therefore $\varphi$ and $\psi$ are isomorphisms, hence the proposition is proved.
\end{proof}

\begin{Definition}
If $(x_1,\cdots,x_N)$ is a reduced word and $w=x_1\cdots x_N\in G$, we call $x_1\cdots x_N$ a \emph{reduced presentation} of $w$.
\end{Definition}

\begin{Definition}
For $w\in G$, define $\len(w)$ to be the length of a reduced presentation of $w$. 
\end{Definition}
By Proposition \ref{prop_word_prob}, $\len(\cdot)$ does not depend on the choice of the reduced presentation, hence it is well-defined.

\begin{Lemma}
\label{lem_generator_G}
For $C\subset \{g_1,\cdots,g_m,g_1^{-1},\cdots,g_m^{-1}\}$, let $G_C$ be the subgroup of $G$ generated by $C$. Suppose $(x_1,\cdots,x_N)$ is a reduced word, then $x_1\cdots x_N\in G_C$ if and only if $x_i\in C \cup C^{-1}$ for all $i$. 
\end{Lemma}
\begin{proof}
Let $w=x_1\cdots x_N$. Suppose $w\in G_C$, then $w=y_1\cdots y_M$ with $y_i\in C\cup C^{-1}$ for all $i$.  If $(y_1,\cdots,y_M)$ is not reduced, there exist letters $y_u$ and $y_v$ such that $y_uy_v=1$ and every letter between $y_u$ and $y_v$ are commutative with both $y_u$ and $y_v$. Removing $y_u$ and $y_v$ from the word yields a shorter word representing the same element of $G$. Repeating this process, we obtain a reduced word representing $w$ which is a sub-word of $(y_1,\cdots,y_M)$. By Proposition \ref{prop_word_prob}, this word is equivalent to $(x_1,\cdots,x_N)$, hence $x_i\in C\cup C^{-1}$ for all $i$.

The other direction of the lemma is obvious.
\end{proof}

\begin{Lemma}
\label{lem_com_gi}
For each $i$, the centralizer of $g_i$ in $G$ is generated by $C_i$.
\end{Lemma}

\begin{proof}
Suppose there exists an element $w$ in the centralizer of $g_i$ that is not generated by $C_i$, choose such a $w$ such that $N:=\len(w)$ is as small as possible. Let $w=x_1\cdots x_N$ be a reduced presentation of $w$, then there exists $u$ such that $x_u\notin C_i$. If $x_1\in C_i$, then $x_2\cdots x_N$ is an element in the centralizer of $g_i$, and by Lemma \ref{lem_generator_G}, the element $x_2\cdots x_N$ is not generated by $C_i$, which contradicts the minimality of $N$. Therefore $x_1\notin C_i$. Similarly, $x_N\notin C_i$. Moreover, the same property holds for every reduced word that is equivalent to $(x_1,\cdots,x_N)$. Therefore $(x_1,\cdots, x_N, g_i, x_N^{-1},\cdots, x_1^{-1})$ is a reduced word. By Proposition \ref{prop_word_prob}, $x_1\cdots x_Ng_ix_N^{-1}\cdots x_1^{-1} \neq g_i$, hence $w$ is not in the centralizer of $g_i$, contradicting the assumption.
\end{proof}

\begin{Lemma}
\label{lem_intersect_centralizer}
Suppose $m\ge 4$, then the only element that is commutative to both $g_1$ and $g_m$ is $1$.
\end{Lemma}

\begin{proof}
The lemma is an immediate consequence of Proposition \ref{prop_word_prob}, Lemma \ref{lem_generator_G} and Lemma \ref{lem_com_gi}.
\end{proof}

\begin{Lemma}
\label{lem_diff_heading}
Suppose $(x_1,\cdots,x_N)$ and $(y_1,\cdots,y_N)$ are reduced words such that $x_1$, $y_1$ are not commutative, then $x_1\cdots x_N\neq y_1\cdots y_N$. 
\end{Lemma}

\begin{proof}
By Proposition \ref{prop_word_prob}, we only need to show that the two words $(x_1,\cdots,x_N)$ and $(y_1,\cdots,y_N)$ are not equivalent. Assume the contrary, let $w_x$ and $w_y$ be the sub-words of $(x_1,\cdots,x_N)$ and $(y_1,\cdots,y_N)$ respectively, consisting of all the letters in $\{x_1,y_1\}$. 
Since $x_1$ and $y_1$ are not commmutative, $w_x$ has to be equal to $w_y$ if $(x_1,\cdots,x_N)$ and $(y_1,\cdots,y_N)$ are equivalent.
On the other hand, $w_x$ starts with $x_1$, and $w_y$ starts with $y_1$, hence $w_x\neq w_y$, which is a contradiction.
\end{proof}

\begin{Lemma}
\label{lem_com_g1gm}
Suppose $m\ge 4$, then the centralizer of $g_1g_m$ is generated by $g_1g_m$.
\end{Lemma}

\begin{proof}
Let $w=x_1\cdots x_N$ be an element in the centralizer of $g_1g_m$, and assume $(x_1,\cdots,x_N)$ is a reduced word.  We use induction on $N$ to show that $w$ is a power of $g_1g_m$. If $N=0$, then $w=1$ and the property is trivial. From now, assume $N>0$, and assume that the claim is proved when $\len(w)<N$.

By the assumptions on $w$, 
we have $g_1g_mwg_m^{-1}g_1^{-1}=w,$
hence the word 
$$(g_1,g_m,x_1,\cdots,x_N,g_m^{-1},g_1^{-1})$$ is not reduced. Therefore there are three possibilities: 
\begin{enumerate}
\item[(a1)] $g_m^{-1}$ is a letter in $(x_1,\cdots,x_N)$, and every letter before the first appearance of $g_m^{-1}$ in $(x_1,\cdots,x_N)$ is in $C_m$; 
\item[(a2)] $g_m$ is a letter in $(x_1,\cdots,x_N)$, and every letter after the last appearance of $g_m$ in $(x_1,\cdots,x_N)$ is in $C_m$; 
\item[(a3)] every letter in $(x_1,\cdots,x_N)$ is contained in $C_m$.
\end{enumerate}
 Case (a3) implies $[w,g_m]=1$, therefore $[w,g_1]=1$, and by Lemma \ref{lem_intersect_centralizer}, $w=1$. Since we are assuming $N>0$, Case (a3) is impossible.

Similarly, since $g_m^{-1}g_1^{-1}wg_1g_m=w$, the word
$$(g_m^{-1},g_1^{-1},x_1,\cdots,x_N,g_1,g_m)$$ is not reduced. Applying the same argument as before, we conclude that there are two possibilities: 
\begin{enumerate}
\item[(b1)] $g_1$ is a letter in $(x_1,\cdots,x_N)$, and every letter before the first appearance of $g_1$ in $(x_1,\cdots,x_N)$ is in $C_1$; 
\item[(b2)] $g_1^{-1}$ is a letter in $(x_1,\cdots,x_N)$, and every letter after the last appearance of $g_1^{-1}$ in $(x_1,\cdots,x_N)$ is in $C_1$; 
\end{enumerate}

Since $m\ge 4$, we have $C_1\cap C_m=\emptyset$, hence (a1) and (b1) are exclusive, and (a2) and (b2) are exclusive. Therefore either (a2) and (b1) hold, or (a1) and (b2) hold. 

If (a2) and (b1) hold, $(x_1,\cdots,x_N)$ is equivalent to a reduced word of the form $(g_1,x_2',\cdots,x_{N-1}',g_m)$. 
Let $w':=x_2'\cdots x_{N-1}'$, then $[g_1w'g_m,g_1g_m]=[w,g_1g_m]=1$, hence $[w',g_mg_1]=1$. Let $\sigma:G\to G$ be the isomorphism of $G$ defined by $\sigma(g_k):=g_{m+1-k}$, then $[\sigma(w'),g_1g_m]=[\sigma(w'),\sigma(g_mg_1)]=1$. By the induction hypothesis, $\sigma(w')$ is a power of $g_1g_m$, therefore $w'$ is a power of $g_mg_1$, hence $w=g_1w'g_m$ is a power of $g_1g_m$. 

If (a1) and (b2) hold, $(x_1,\cdots,x_N)$ 
is equivalent to a reduced word of the form $(g_m^{-1},x_2',\cdots,x_{N-1}',g_1^{-1})$, and the result follows from a similar argument.
\end{proof}

\begin{Lemma}
\label{lem_sol_eqn}
Suppose $m\ge 4$.
The solutions to the equation
\begin{equation}
\label{eqn_conjugation}
u\,g_1 u^{-1} \cdot v\, g_m\, v^{-1}= g_1 g_m
\end{equation}
for $u,v\in G$ are given by
\begin{align}
u &= (g_1g_m)^k u',  \label{eqn_conjugation_u}\\
v &= (g_1 g_m)^k v',  \label{eqn_conjugation_v}
\end{align}
where $k\in\bZ$, $u'$ is in the centralizer of $g_1$, and $v'$ is in the centralizer of $g_m$.
\end{Lemma}

\begin{remark}
The expressions on the right-hand side of \eqref{eqn_conjugation_u} and \eqref{eqn_conjugation_v} are not required to be reduced. For example, we may have $k=1$, $u'=1$, $v'=g_m^{-1}$. 
\end{remark}

\begin{proof}
It is clear that every pair $(u,v)$ given by \eqref{eqn_conjugation_u} and \eqref{eqn_conjugation_v} is a solution to \eqref{eqn_conjugation}. To prove the reverse, we use induction on $\len(u)+\len(v)$. If $\len(u)+\len(v)=0$, then $u=v=1$, and the result is obvious. 

Suppose $\len(u)+\len(v)=N>0$, and assume the result is proved when $\len(u)+\len(v)<N$. 
We can write $u$ as $u=u_1u_2$ with the following properties:
\begin{enumerate}
\item $\len(u)=\len(u_1)+\len(u_2)$,
\item $u_2$ is in the centralizer of $g_1$,
\item $u_1$ does not have a reduced presentation that ends with a letter in $C_1$.
\end{enumerate}
Notice that $(u_1,v)$ is also a solution to \eqref{eqn_conjugation}, therefore the result is proved by the induction hypothesis if $u_2\neq 1$. 

Similarly, we can write $v$ as $v=v_1v_2$ with the following properties:
\begin{enumerate}
\item $\len(v)=\len(v_1)+\len(v_2)$,
\item $v_2$ is in the centralizer of $g_m$,
\item $v_1$ does not have a reduced presentation that ends with a letter in $C_m$.
\end{enumerate}
Since $(u,v_2)$ is also a solution to \eqref{eqn_conjugation}, the result is proved by the induction hypothesis if $v_2\neq 1$.

From now on, we assume $u_2=v_2=1$.
This implies $u=u_1$, $v=v_1$, hence both $ug_1u^{-1}$ and $v g_m^{-1} v^{-1}$ are reduced presentations, and we have 
$$\len(ug_1u^{-1}) = 2\len(u)+1,$$
$$\len(v g_m^{-1} v^{-1})=2\len(v)+1.$$
As a result,
$$\len(g_1^{-1}ug_1u^{-1}) = 2\len(u)+2\quad\text{or}\quad  2\len(u),$$
$$\len(g_m v g_m^{-1} v^{-1})=2\len(v)+2\quad\text{or} \quad 2\len(v).$$
By \eqref{eqn_conjugation}, 
$$
g_1^{-1}\,u\,g_1\,u^{-1} = g_m \, v\, g_m^{-1}\,v^{-1},
$$
therefore there are four possibilities;

Case 1. $\len(u)=\len(v)+1$, the expression $g_m v g_m^{-1} v^{-1}$ is reduced, and $g_1^{-1}ug_1u^{-1}$ is not reduced. By the previous assumption on $u$, the element $u$ cannot be represented by a reduced word that ends with a letter in $C_1$, therefore for $g_1^{-1}ug_1u^{-1}$ to be not reduced, $u$ must have a presentation of the form $u=g_1 \hat u$, where $\len(\hat u) = \len(u)-1$. Thus we have
\begin{equation}
\label{eqn_hat_u_conj}
g_m v g_m^{-1} v^{-1} = \hat u g_1 \hat u^{-1} g_1^{-1}.
\end{equation}

Since the left-hand side of \eqref{eqn_hat_u_conj} is reduced, and the right-hand side of \eqref{eqn_hat_u_conj} is given by a word with the same length, the right-hand side of \eqref{eqn_hat_u_conj} is also reduced. Therefore by Proposition \ref{prop_word_prob}, the corresponding words given by the two sides of \eqref{eqn_hat_u_conj} are equivalent. By the assumption $v=v_1$, every reduced presentation of $g_m v g_m^{-1} v^{-1}$ has the property that the product of the first $\len(v)+1$ terms is $g_mv$. Similarly, by $u=u_1$, every reduced presentation of $\hat u g_1 \hat u^{-1} g_1^{-1}$ has the property that the product of the first $\len(\hat u)+1$ terms is $\hat ug_1$. Therefore
\begin{align*}
g_m v &= \hat u g_1, \\
g_m^{-1} v^{-1} &= \hat u^{-1} g_1^{-1}.
\end{align*}
Hence
$g_1 \hat u = v g_m$, and
$$
g_1g_mv = g_1\hat u g_1 = v g_m g_1,
$$
hence $[g_1g_m,vg_m]=1$.
By Lemma \ref{lem_com_g1gm}, we have $v=(g_1g_m)^kg_m^{-1}$ for some integer $k$. By the previous equations, $\hat u= g_1^{-1}(g_mg_1)^k$, and $u=g_1\hat u = (g_1g_m)^k$, therefore the desired result is proved. 

Case 2. $\len(v)=\len(u)+1$, and $g_1^{-1}ug_1u^{-1}$ is reduced,  $g_m v g_m^{-1} v^{-1}$ is not reduced. This case follows from the same argument as Case 1.

Case 3. $\len(u)=\len(v)$, both $g_1^{-1}ug_1u^{-1}$ and $g_m v g_m^{-1} v^{-1}$  are reduced. This is impossible by Lemma \ref{lem_diff_heading}. 

Case 4. $\len(u)=\len(v)$, neither $g_1^{-1}ug_1u^{-1}$ nor $g_m v g_m^{-1} v^{-1}$  is reduced. By the previous assumption that $u=u_1$, for $g_1^{-1}ug_1u^{-1}$ to be not reduced, $u$ must have a presentation of the form $u=g_1 \hat u$, where $\len(\hat u) = \len(u)-1$. Similarly by the assumption that $v=v_1$, there is a presentation of $v$ given by $v=g_m^{-1} \hat v$, where $\len(\hat v)=\len(v)-1$. Equation \eqref{eqn_conjugation} gives
$$
\hat u g_1 \hat u^{-1} g_1^{-1} g_m^{-1} \hat v g_m \hat v^{-1} =1,
$$
therefore
$$
\hat v g_m \hat v^{-1} \hat u g_1 \hat u = g_m g_1.
$$
Let $\sigma:G\to G$ be the isomorphism of $G$ defined by $\sigma(g_k):=g_{m+1-k}$, then 
$$
\sigma(\hat v) g_1 \sigma(\hat v)^{-1} \sigma(\hat u) g_m \sigma(\hat u)^{-1} = g_1 g_m.
$$
By the induction hypothesis, $\sigma(\hat v) = (g_1g_m)^k \hat v'$,  $\sigma(\hat u) = (g_1g_m)^k \hat u'$, where $k\in \bZ$, $\hat v'$ is in the centralizer of $g_1$, and $\hat u'$ is in the centralizer of $g_m$. Therefore 
\begin{align*}
u = g_1 \hat u &= g_1(g_mg_1)^k \sigma(\hat u')=(g_1g_m)^k g_1\sigma(\hat u')\\
v=g_m^{-1}\hat v &= g_m^{-1}(g_mg_1)^k \sigma(\hat v')
=(g_1g_m)^k g_m^{-1}\sigma(\hat v').
\end{align*}
Since $g_1\sigma(\hat u')$ is in the centralizer of $g_1$, and $g_m^{-1}\sigma(\hat v')$ is in the centralizer of $g_m$, the desired result is proved.

In conclusion, every solution of \eqref{eqn_conjugation} can be written as \eqref{eqn_conjugation_u} and \eqref{eqn_conjugation_v}.
\end{proof}

\begin{Corollary}
\label{cor_sol_eqn}
Suppose $m\ge 4$.
The solutions to the equation
\begin{equation}
\label{eqn_conjugation'}
u\,g_1 u^{-1} \cdot v\, g_m^{-1}\, v^{-1}= g_1 g_m^{-1}
\end{equation}
for $u,v\in G$ are given by
\begin{align}
u &= (g_1g_m^{-1})^k u',  \label{eqn_conjugation_u'}\\
v &= (g_1 g_m^{-1})^k v',  \label{eqn_conjugation_v'}
\end{align}
where $k\in\bZ$, $u'$ is in the centralizer of $g_1$, and $v'$ is in the centralizer of $g_m$.
\end{Corollary}

\begin{proof}
Notice that there is an isomorphism $\sigma:G\to G$ defined by $\sigma(g_i):=g_i$ for $i<m$, and $\sigma(g_m):=g_m^{-1}$. Apply $\sigma$ to Lemma \ref{lem_sol_eqn} yields this result.
\end{proof}

\section{Arcs on compact surfaces}
\label{sec_arc}
This section collects several results about arcs on surfaces that will be used later. We first recall the following result of Feustel.

\begin{Proposition}[\cite{feustel1966homotopic}]
\label{prop_homotopy_isotopy}
Let $S$ be a smooth compact surface with boundary, let $\gamma_1$, $\gamma_2$ be two smoothly embedded arcs in $S$ such that $\gamma_i\cap \partial S=\partial \gamma_i$ and $\gamma_i$ is transverse to $\partial S$ for $i=1,2$. 
Suppose $\gamma_1$ and $\gamma_2$ are homotopic to each other in $S$ relative to $\partial S$, then $\gamma_1$ and $\gamma_2$ are isotopic to each other in $S$ relative to $\partial S$.
\end{Proposition}

We also need the following result in Section \ref{sec_nonexistence}.

\begin{Lemma}
\label{lem_annulus_remove_disk}
Let $p\in S^1$, let $D$ be a closed disk in $(S^1-\{p\})\times [0,1]$. Let $S:=S^1\times[0,1]-D$, let $\gamma_0 := \{p\}\times [0,1]$. Let $f_1:S\to S$ be the Dehn twist along a curve parallel to $S^1\times\{0\}$, and let $f_2:S\to S$ be the Dehn twist along a curve parallel to $S^1\times\{1\}$. Suppose $\gamma$ is an arc on $S$ from $(p,0)$ to $(p,1)$, then there exist integers $u$ and $v$ such that $\gamma$ is isotopic to $f_1^u f_2^v(\gamma_0)$ in $S$ relative to $\partial S$.
\end{Lemma}

\begin{proof}
Notice that $S$ can be embedded in $\bR^2$.
By the Jordan curve theorem, cutting $S$ open along $\gamma$ yields a closed annulus with corners, which is diffeomorphic to the manifold obtained by cutting $S$ open along $\gamma_0$. Hence there exists an orientation-preserving diffeomorphism $\varphi:S\to S$ such that $\varphi(\gamma_0) = \gamma$ and $\varphi|_{\partial S}=\id$.  Since 
\begin{enumerate}
\item the mapping class group of $S$ is generated by Dehn twists (see, for example, \cite{au2010lectures} or \cite[Corollary 4.16]{farb2011primer}), 
\item  every simple closed curve on $S$ is parallel to the boundary, 
\item Dehn twists along curves parallel to $\partial D$ preserve the isotopy class of $\gamma_0$,
\end{enumerate}
 the desired result is proved.
\end{proof}

The rest of this section studies arcs on Seifert surfaces.

\begin{Definition}
\label{def_K(S,gamma)}
Let $L_0$ be a link in $\bR^3$, let $S$ be a Seifert surface of $L_0$, and let $\gamma$ be an arc on $S$ such that $\gamma$ intersects $\partial S$ transversely in $S$ and $\gamma\cap S=\partial \gamma$. Define $K(S,\gamma)$ to be the knot in $\bR-L_0$ which satisfies the the property that $K(S,\gamma)$ bounds an embedded disk $D$ in $\bR^3$ that intersects $S$ transversely at $\gamma$.
\end{Definition}

\begin{remark}
Since $K(S,\gamma)$ can be isotoped to a neighborhood of $\gamma$ in $D-\gamma$, the knot $K(S,\gamma)$ satisfying  Definition \ref{def_K(S,gamma)} is unique up to isotopy in $\bR^3-L_0$.
 An example of $K(S,\gamma)$ can be constructed as follows.
Let $S'$ be an extension of $S$ to a slightly bigger embedded surface such that $S$ is in the interior of $S'$. Let $N(S')\subset \bR^3$ be a small neighborhood of the zero section of the normal bundle of $S'$, then $N(S')$ is a neighborhood of $S$. Let $\pi:N(S')\to S'$ be the bundle projection. Let $\gamma'$ be an extension of $\gamma$ in $S'$. Then $K(S,\gamma)$ can be taken to be the boundary of a neighborhood of $\gamma$ in 
$\pi^{-1}(\gamma')$. 

By definition, $K(S,\gamma)$ is always an unknot in $\bR^3$.
\end{remark}

\begin{Lemma}
\label{lem_disk_intersect_arc}
Let $S$, $\gamma$ be as in Definition \ref{def_K(S,gamma)}. Suppose $K$ is a knot in $\bR^3-L_0$ such that $K$ bounds an embedded disk $D$ in $\bR^3$. Moreover, assume $D$ intersects $S$ transversely, and that $D\cap S$ is the disjoint union of $\gamma$ and a family of circles. Then $K$ is isotopic to $K(S,\gamma)$ in $\bR^3-L_0$.
\end{Lemma}
\begin{proof}
By the assumptions, $\gamma$ is an arc in the interior of $D$, and $D\cap L_0 = \partial (D\cap S) = \partial \gamma$. Let $K'$ be the boundary of a small neighborhood of $\gamma$ in $D$, then $K$ is isotopic to $K'$ in $D-\gamma$. Since $(D-\gamma) \cap L_0 = \emptyset$, the isotopy remains in $\bR^3-L_0$. By the definition of $K(S,\gamma)$, the knot $K'$ is isotopic to $K(S,\gamma)$ in $\bR^3-L_0$, hence the lemma is proved.
\end{proof}

\begin{Lemma}
\label{lem_monodromy_arc}
Let $S$, $\gamma$ be as in Definition \ref{def_K(S,gamma)}. Suppose $L_0$ is a fibered link with respect to the Seifert surface $S$ and with monodromy $f:S\to S$. Then $K(S,\gamma)$ is isotopic to $K(S,f(\gamma))$ in $\bR^3-L_0$.
\end{Lemma}

\begin{proof}
By the definition of monodromy, there exists an isotopy of $S$ given by 
$$\tau: S\times [0,1]\to \bR^3,$$ such that 
\begin{enumerate}
 \item $\tau(x,t)$ is independent of $t$ for $x\in \partial S$, 
 \item $\tau(x,0)=x$ for all $x\in S$,
 \item $\tau(x,1)=f(x)$ for all $x\in S$.
 \end{enumerate}
 The map $\tau$ induces an isotopy from $K(S,\gamma)$ to $K(S,f(\gamma))$ in $\bR^3-L_0$ by the family of knots $K(\tau(S,t),\tau(\gamma,t))$.
\end{proof}

\section{The link $L_{u,v}$}
\label{sec_Luv}


This section defines a family of links $L_{u,v}$ and computes their Jones polynomials at $t=-1$. The computation will be used in the proof of the non-existence of the hypothetical link $L$ satisfying Condition \ref{cond_cycle}.

\begin{Definition}
\label{def_Luv}
For a pair of integers $(u,v)$ with $u\ge 3$, we define a link $L_{u,v}$ as follows. If $v\ge 0$, define $L_{u,v}$ to be the link given by Figure \ref{fig_Luv1} with $u$ components such that there are $v$ crossings in the dotted rectangle. If $v<0$, define $L_{u,v}$ to be the link given by Figure \ref{fig_Luv2} with $u$ components  such that there are $|v|$ crossings in the dotted rectangle.
\end{Definition}
\begin{figure}  
  \includegraphics[width=0.7\linewidth]{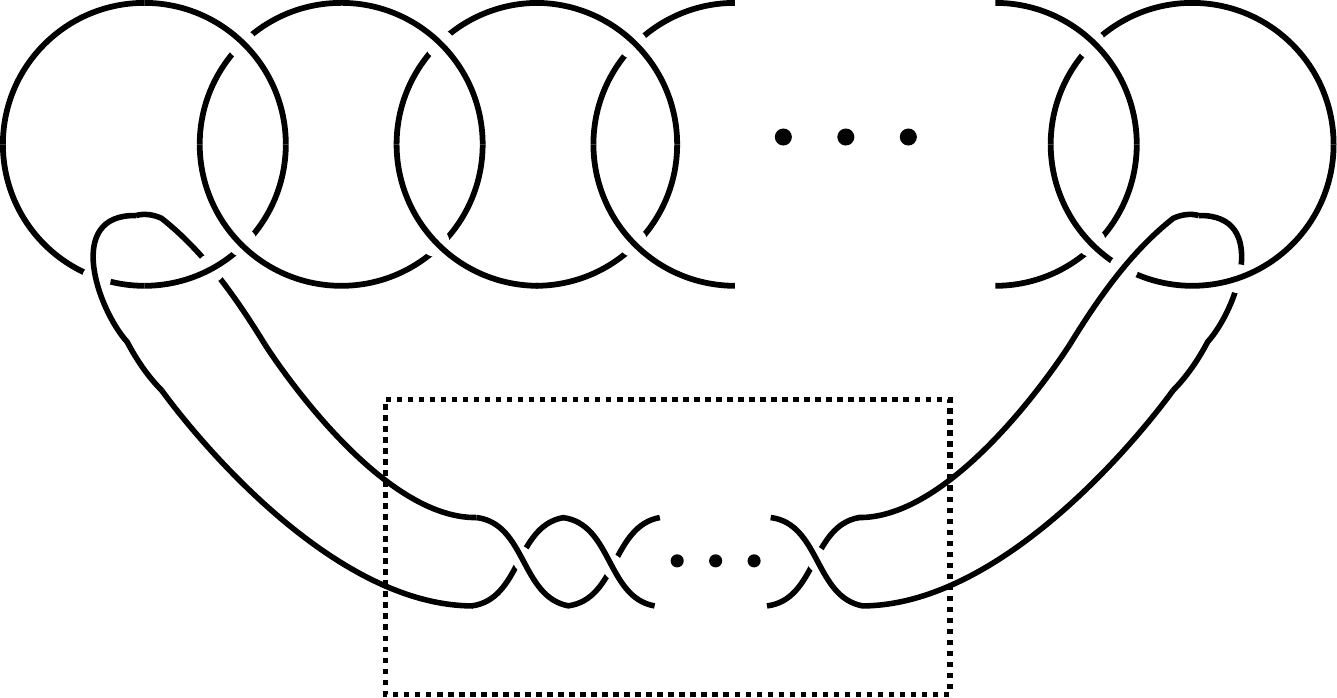}
  \caption{$L_{u,v}$ when $v\ge 0$.}
  \label{fig_Luv1}
  \vspace{\baselineskip}
  \includegraphics[width=0.7\linewidth]{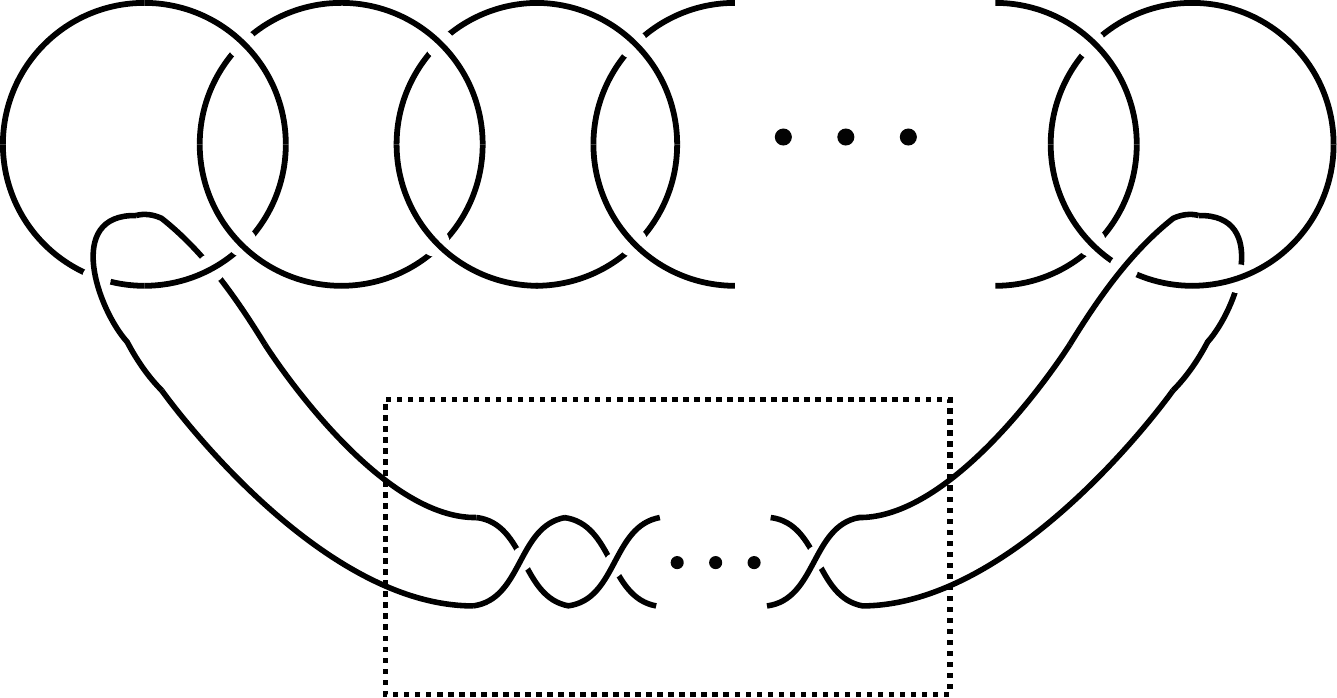}
  \caption{$L_{u,v}$ when $v< 0$.}
  \label{fig_Luv2}
\end{figure}

The only difference between Figure \ref{fig_Luv1} and Figure \ref{fig_Luv2} is that the crossings in the dotted rectangles are reversed. Notice that $L_{u,v}$ is alternating if $v\ge 0$. 

Let $V(L_{u,v})$ be the (reduced) Jones polynomial of $L_{u,v}$ with the orientation given by Figure \ref{fig_Luv_orientation}. Let $V_{u,v}$ be the value of $V(L_{u,v})$ when plugging in $t^{1/2} = -i$. 
\begin{figure}  
  \includegraphics[width=0.7\linewidth]{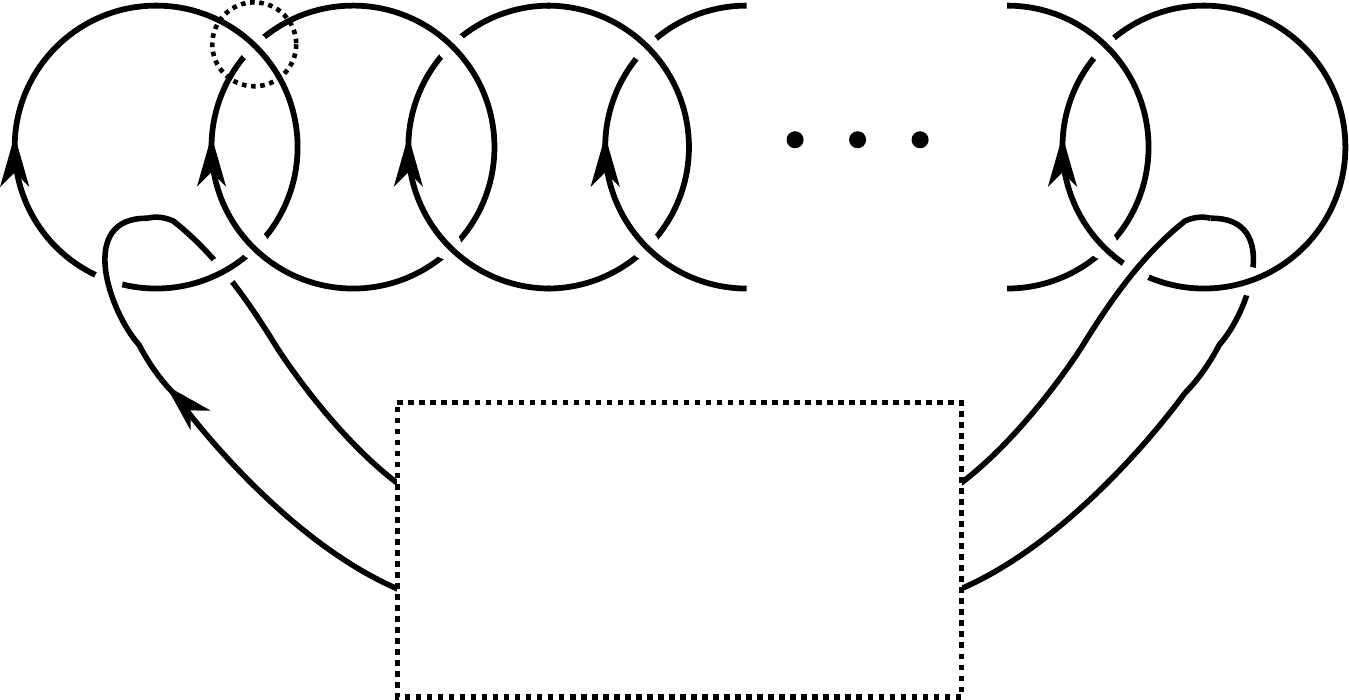}
  \caption{An orientation of $L_{u,v}$.}
  \label{fig_Luv_orientation}
\end{figure}

Notice that the Hopf link with linking number $1$ has Jones polynomial $-t^{1/2}-t^{5/2}$; and the Hopf link with linking number $-1$ has Jones polynomial $-t^{-1/2}-t^{-5/2}$. Moreover, the Jones polynomial of the connected sum of links is the product of the Jones polynomial of each summand. 
Therefore, if $v$ is even, 
by the skein relation at the dotted circle in Figure \ref{fig_Luv_orientation}, we have
$$
(t^{1/2}-t^{-1/2})V(L_{u-1,v}) = t^{-1} V(L_{u,v}) - t (-t^{1/2}-t^{5/2})^{u-1}.
$$
If $v$ is odd, then the skein relation gives
$$
(t^{1/2}-t^{-1/2})V(L_{u-1,v}) = t^{-1} V(L_{u,v}) - t (-t^{1/2}-t^{5/2})^{u-2} (-t^{-1/2}-t^{-5/2}).
$$
Hence 
$$
V_{u,v} = \begin{cases*}
 (2i)V_{u-1,v} + (2i)^{u-1} &\text{if} v \text{is even,}\\
 (2i)V_{u-1,v} - (2i)^{u-1} &\text{if} v \text{is odd.}
 \end{cases*}
$$
On the other hand, if $v$ is even, the skein relation at a crossing in the dotted box in Figure \ref{fig_Luv_orientation} yields
$$
(t^{1/2}-t^{-1/2})(-t^{1/2}-t^{5/2})^u = t^{-1} V(L_{u,v-2}) - t V(L_{u,v}).
$$
If $v$ is odd, then the skein relation gives
$$
(t^{1/2}-t^{-1/2})(-t^{1/2}-t^{5/2})^{u-1}(-t^{-1/2}-t^{-5/2}) = t^{-1} V(L_{u,v-2}) - t V(L_{u,v}).
$$
Therefore
$$
V_{u,v} = \begin{cases*}
 V_{u,v-2} - (2i)^{u+1} &\text{if} v \text{is even,}\\
 V_{u,v-2} + (2i)^{u+1} &\text{if} v \text{is odd.}
 \end{cases*}
$$

It can be directly computed that
\begin{align*}
&V(L_{3,-1}) = 2 + t^2 + t^4, \\
&V(L_{3,0}) = t^7-t^6+3t^5-t^4+3t^3-2t^2+t,
\end{align*}
hence
$$
V_{3,-1} = 4, V_{3,0} = -12.
$$

Combining the computations above, we have
\begin{equation}\label{eqn_V_u,v}
V_{u,v} = (-1)^v(2i)^{u-1}(u+2v).
\end{equation}

As a consequence, we have the following result.

\begin{Corollary}
\label{cor_Jones_Luv}
If $|u+2v|>1$, then $\rank_{\bZ/2} \Kh(L_{u,v};{\mathbb{Z}/2}) > 2^u$.
\end{Corollary}
\begin{proof}
Since the coefficients of the Jones polynomial $V(L_{u,v})$ are the Euler characteristics of $\Khr(L_{u,v})$ at different $q$-gradings, we have
$$
\rank_{\bZ/2} \Khr(L_{u,v};{\mathbb{Z}/2}) \ge |V_{u,v}| = |2^{u-1} (u+2v)|.
$$
If $|u+2v|>1$, then
$$\rank_{\bZ/2} \Kh(L_{u,v};{\mathbb{Z}/2}) = 2\cdot \rank_{\bZ/2}\Khr(L_{u,v}; {\mathbb{Z}/2}) > 2^u.\phantom\qedhere\makeatletter\displaymath@qed$$ 
\end{proof} 

\section{The non-existence of $L$}
\label{sec_nonexistence}
This section combines the results from Sections \ref{sec_top_properties}, \ref{sec_fundamental_group}, \ref{sec_arc} and \ref{sec_Luv} to prove that the hypothetical link $L$ satisfying Condition \ref{cond_cycle} does not exist. We will proceed by showing more properties of $L$ and eventually deduce a contradiction. By Lemma \ref{thm_modulo_cond_cycle}, this will finish the proof of Theorem \ref{The_main_theorem}.

Recall that the components of $L$ are $K_1,\cdots,K_n$, and $L'=K_1\cup\cdots\cup K_{n-1}$. We have defined $S_1$ and $S_2$ to be the Seifert surfaces of $L'$ given by Figure \ref{fig_fiber1} and Figure \ref{fig_fiber2} respectively. By the conditions on the linking numbers of $L$, there are two possibilities:

{\bf Case 1.} The algebraic intersection number of $S_1$ and $K_n$ is zero;

{\bf Case 2.} The algebraic intersection number of $S_2$ and $K_n$ is zero.

By Proposition \ref{Prop_disjoint_Seifert_surface}, for $j\in\{1,2\}$, if the algebraic intersection number of $S_j$ and $K_n$ is zero, then $K_n$ can be isotopically deformed in $\bR^3-L'$ into $\bR^3-S_j$. The first half of this section will focus on Case 1. The argument for Case 2 is similar and will be sketched afterwards.

Let $\gamma_0$ be the arc on $S_1$ as shown in Figure \ref{fig_standard_arc}, where $\gamma_0$ starts from a point $p_1\in K_1$ and travels from left to right, goes through the crossings of $L'$ in an alternating way, and ends at a point $p_2\in K_{n-1}$.
\begin{figure}  
  \includegraphics[width=0.95\linewidth]{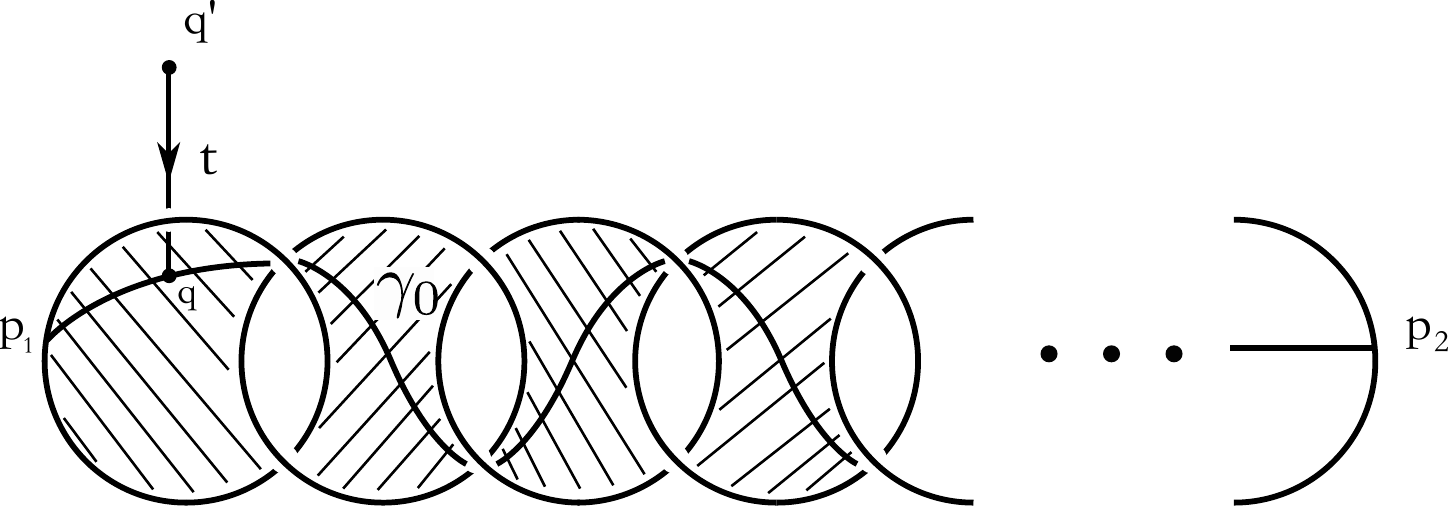}
  \caption{The arc $\gamma_0$ on $S_1$.}
  \label{fig_standard_arc}
\end{figure}

\begin{Lemma}
\label{lem_existence_gamma}
Suppose Case 1 holds,
then there exists an arc $\gamma\subset S_1$ from $p_1$ to $p_2$ such that $K_n$ is isotopic to $K(S_1,\gamma)$ in $\bR^3-L'$.
\end{Lemma}

\begin{proof}
By Proposition \ref{Prop_disjoint_Seifert_surface}, there exists a knot $K_n'\subset \bR^3-S_1$ such that $K_n$ is isotopic to $K_n'$ in $\bR^3-L'$.
By Proposition \ref{prop_embedded_disks}, $K_n'$ bounds a disk $D_n$ such that $D_n$ intersects $K_1$ and $K_{n-1}$  respectively at one point, and is disjoint from $K_2\cup\cdots\cup K_{n-2}$. After a further isotopy, we may assume that $D_n\cap L' =\{p_1,p_2\}$, and that $D_n$ intersects $S_1$ transversely. Therefore $D_n\cap S_1$ consists of an arc $\gamma\subset S_1$ from $p_1$ to $p_2$ and a union of circles. By Lemma \ref{lem_disk_intersect_arc}, $K_n'$ is isotopic to $K(S_1,\gamma)$ in $\bR^3-L'$.
\end{proof}

\begin{Lemma}
\label{lem_relation_gamma_gamma0}
Suppose Case 1 holds.
Fix an orientation on $S_1$,
let $f_1$, $f_2$ be the Dehn twists on $S_1$ along an oriented curve parallel to $K_1$ and an oriented curve parallel to $K_{n-1}$ respectively, and let $f_3:S_1\to S_1$ be the monodromy of the fibered structure of $L'$. Let $\gamma$ be the arc given by Lemma \ref{lem_existence_gamma}, then there exist integers $a,b,c$ such that $\gamma$ is isotopic to $f_1^af_2^bf_3^c(\gamma_0)$ relative to $\{p_1,p_2\}$ on $S_1$. 
\end{Lemma}

\begin{proof}
If $n=3$, then $S_1$ is an annulus, and every arc from $p_1$ to $p_2$ is isotopic to $f_1^a\gamma_0$ for some integer $a$. If $n=4$, then $S_1$ is an annulus with a disk removed, and the result follows from Lemma \ref{lem_annulus_remove_disk} with $c=0$. From now we assume $n\ge 5$. 

Fix a point $q$ in the interior of $\gamma_0$ as shown in Figure \ref{fig_standard_arc}. Let $\gamma_1$ be the sub-arc of $\gamma_0$ from $p_1$ to $q$, and let $\gamma_2$ be the sub-arc of $\gamma_0$ from $q$ to $p_2$. Then  there exists a closed curve $w$ in the interior of $S_1$ bases on $q$, such that $\gamma$ is homotopic to $\gamma_1\cdot w\cdot \gamma_2$ relative to $\{p_1,p_2\}$ on $S_1$. The loop $w$ is not necessarily simple.

Let $g_1,\cdots,g_{n-1}$ be the generators of $\pi_1(\bR^3-L',q')$ defined in Section \ref{sec_fundamental_group}, where $q'$ is their base point. Fix an arc $t$ from $q'$ to $q$ as given by Figure \ref{fig_standard_arc}, let $ t^{-1}$ be the same arc with the reversed orientation, let $[w]\in \pi_1(\bR^3-L',q')$ be the homotopy class of $t\cdot w\cdot t^{-1}$.  

Every oriented knot in $\bR^3-L'$ defines a conjugation class in $\pi_1(\bR^3-L',q')$. By Corollary \ref{cor_cable}, the conjugation class defined by $K_n$ has the form $g_1^u g_{n-1}^v$, where $u,v\in\{-1,1\}$ depend on the signs of the linking numbers and the orientation of $K_n$. On the other hand, under a suitable orientation, the conjugation class defined by $K(S_1,\gamma)$ is given by $g_1 [w] g_{n-1}^{b'} [w]^{-1}$, where $b'=(-1)^{n+1}$. Therefore there exists $r\in \pi_1(\bR^3-L',q')$ such that
$$
r\,g_1\,[w]\, g_{n-1}^{b'}\, [w]^{-1}\, r^{-1} = g_1^a \,g_{n-1}^b.
$$
Comparing the images of both sides in $H_1(\bR^3-L';\bZ)$ yields $a=1$, $b=b'$, thus the equation can be rewritten as
$$
r\,g_1\,r^{-1}\cdot (r[w])\,g_{n-1}^{b'}\,(r[w])^{-1} = g_1 \,g_{n-1}^{b'}.
$$
Apply Lemma \ref{lem_sol_eqn} and Corollary \ref{cor_sol_eqn} for $u=r$, $v=r[w]$, and invoke Lemma \ref{lem_com_gi}, we have
$$
[w] = g_1^\alpha \, g_2^\beta \, g_{n-2}^\delta \, g_{n-1}^\eta,
$$
for $\alpha,\beta,\delta,\eta\in\bZ$. Notice that the image of $H_1(\interior(S_1);\bZ)$ in $H_1(\bR^3-L';\bZ)$ is generated by $[g_1]+[g_2], [g_2]+[g_3], [g_3]+[g_4], \cdots, [g_{n-2}]+[g_{n-1}]$, therefore we have 
$$
\alpha - \beta + (-1)^{n-1}\delta + (-1)^n \eta = 0,
$$
hence 
\begin{equation}
\label{eqn_alpha_theta_eta}
[w] = (g_1g_2)^{\beta} \, g_1^{\theta}\, (g_{n-1}^{(-1)^{n-1}})^{\theta} \,(g_{n-2}g_{n-1})^{\delta},
\end{equation}
where $\theta := \alpha-\beta = (-1)^{n-1}(\eta-\delta)$.

We construct a set of generators of $\pi_1(\interior(S_1),q)$ as follows. Let $u_1,\cdots,u_{n-2}$ be the oriented simple closed curves on $S_1$ as given by Figure \ref{fig_pi1_S1}. Each $u_i$ intersects $\gamma_0$ at one point near one of the crossings of $L'$. Let $q_i$ be the intersection point of $u_i$ and $\gamma_0$, let $v_i$ be the sub-arc of $\gamma_0$ from $q$ to $q_i$, let $v_i^{-1}$ be the same arc with the reversed orientation. Let $u_i'$ be the loop based at $q$ defined by $v_i\cdot u_i\cdot v_i^{-1}$. Then $\pi_1(\interior(S_1),q)$ is a free group generated by $[u_1'],\cdots,[u_{n-2}']$. Equation \eqref{eqn_alpha_theta_eta} implies that $w$ is based homotopic to  
\begin{equation}\label{eqn_u_beta_theta_delta}
u_1'^\beta\cdot (u_1'^\theta u_2'^{-\theta} u_3'^{\theta}\cdots u_{n-2}'^{(-1)^{n-1}\theta})\cdot u_{n-2}'^\delta
\end{equation}
in $\pi_1(\bR^3-L',q)$.
\begin{figure}  
  \includegraphics[width=0.95\linewidth]{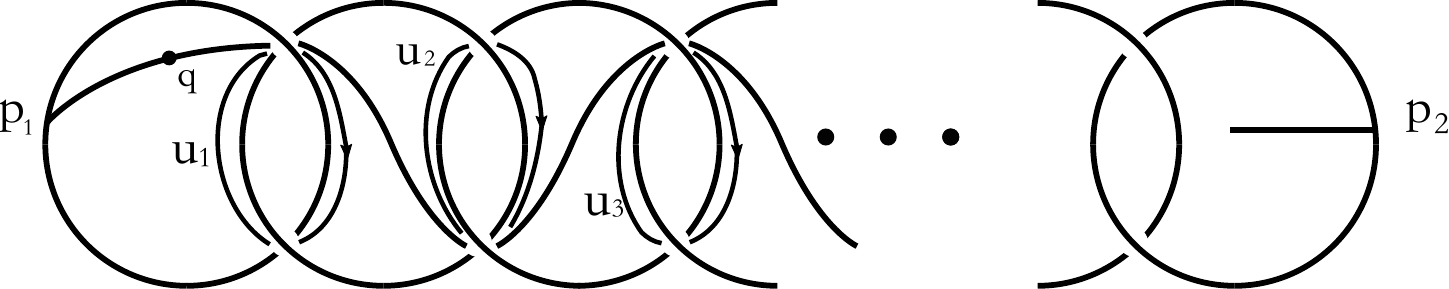}
  \caption{The generators of $\pi_1(\interior(S_1),q)$.}
  \label{fig_pi1_S1}
\end{figure}

Since $\bR^3-L'$ is a fiber bundle over $S^1$ with a fiber being the interior of $S_1$, the map from $\pi_1(\interior(S_1),q)$ to $\pi_1(\bR^3-L',q)$ is injective, hence $w$ is based homotopic to \eqref{eqn_u_beta_theta_delta} in $S_1$. By Lemma \ref{lem_connect_sum_fiber}, the monodromy $f_3$ is given by the composition of the Dehn twists along $u_1,\cdots,u_{n-2}$. 
Therefore, under a suitable choice of the orientation for the monodromy $f_3$, the image of $\gamma_0$ under $f_3^c$ is homotopic to $\gamma_1\cdot(u_1'^c u_2'^{-c} u_3'^{c}\cdots u_{n-2}'^{(-1)^{n-1}c}) \cdot \gamma_2 $ relative to $\{p_1,p_2\}$, where the alternating signs in front of $c$ come from the fact that the normal vector field of $S_1$ switches directions at each crossing of the diagram. 
As a consequence, $\gamma$ is homotopic to $f_1^af_2^bf_3^c(\gamma_0)$ relative to $\{p_1,p_2\}$ on $S_1$ with $a=\pm \beta$, $b=\pm \delta$, $c=\theta$, where the signs depend on the orientations of the Dehn twists in the definitions of $f_1,f_2$. By Proposition \ref{prop_homotopy_isotopy}, $\gamma_0$ is isotopic to $f_1^af_2^bf_3^c(\gamma_0)$ relative to $\{p_1,p_2\}$ on $S_1$.
\end{proof}

\begin{Corollary}
\label{cor_isotopic_to_standard_Kn}
Under the condition of Case 1, the knot $K_n$ is isotopic to $K(S_1,\gamma_0)$ in $\bR^3-L'$.
\end{Corollary}
\begin{proof}
Let $f_1,f_2,f_3$ be as in Lemma \ref{lem_relation_gamma_gamma0}. 
By Lemma \ref{lem_existence_gamma} and Lemma \ref{lem_relation_gamma_gamma0}, there exist integers $a,b,c$ such that $K_n$ is isotopic to $K(S_1,\gamma)$ in $\bR^3-L'$, where $\gamma$ is an arc on $S_1$ that is isotopic to $f_1^af_2^bf_3^c(\gamma_0)$ relative to $\{p_1,p_2\}$. Therefore $\gamma$ is isotopic to $f_3^c(\gamma_0)$ on $S_1$ if we allow its boundary points to move on $\partial S_1$. Hence $K(S_1,\gamma)$ is isotopic to $K(S_1,f_3^c(\gamma_0))$ in $\bR^3-L'$. By Lemma \ref{lem_monodromy_arc}, $K(S_1,f_3^c(\gamma_0))$ is isotopic to $K(S_1,\gamma_0)$ in $\bR^3-L'$, hence the result is proved.
\end{proof}

Recall that for a pair of integers $u,v$ with $u\ge 3$, the link $L_{u,v}$ is defined by Definition \ref{def_Luv}.

\begin{Lemma}
\label{lem_exist_Cn}
Under the condition of Case 1, there exists an integer $C(n)$ depending on $n$, such that $L$ is isotopic to $L_{n,C(n)}$. 
\end{Lemma}

\begin{proof}
Notice that $(S_1,\gamma_0)$ is isotopic to Figure \ref{fig_twist_S1}.
\begin{figure}
  \includegraphics[width=\linewidth]{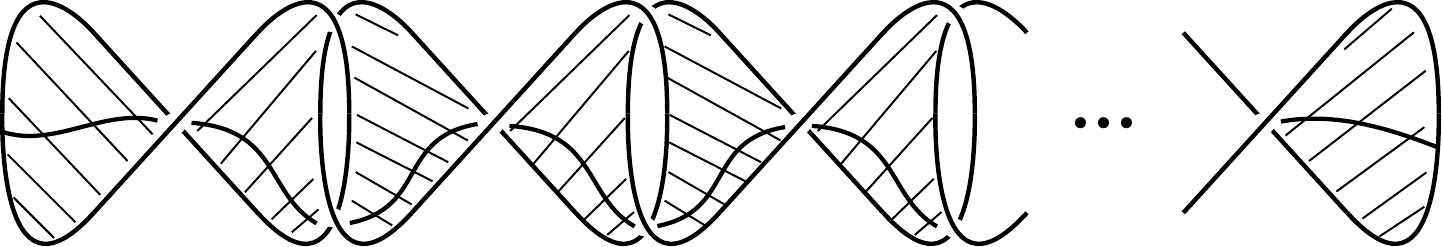}
  \caption{Another diagram for $S_1$ and $\gamma_0$.}
  \label{fig_twist_S1}
  \vspace{\baselineskip}
  \includegraphics[width=\linewidth]{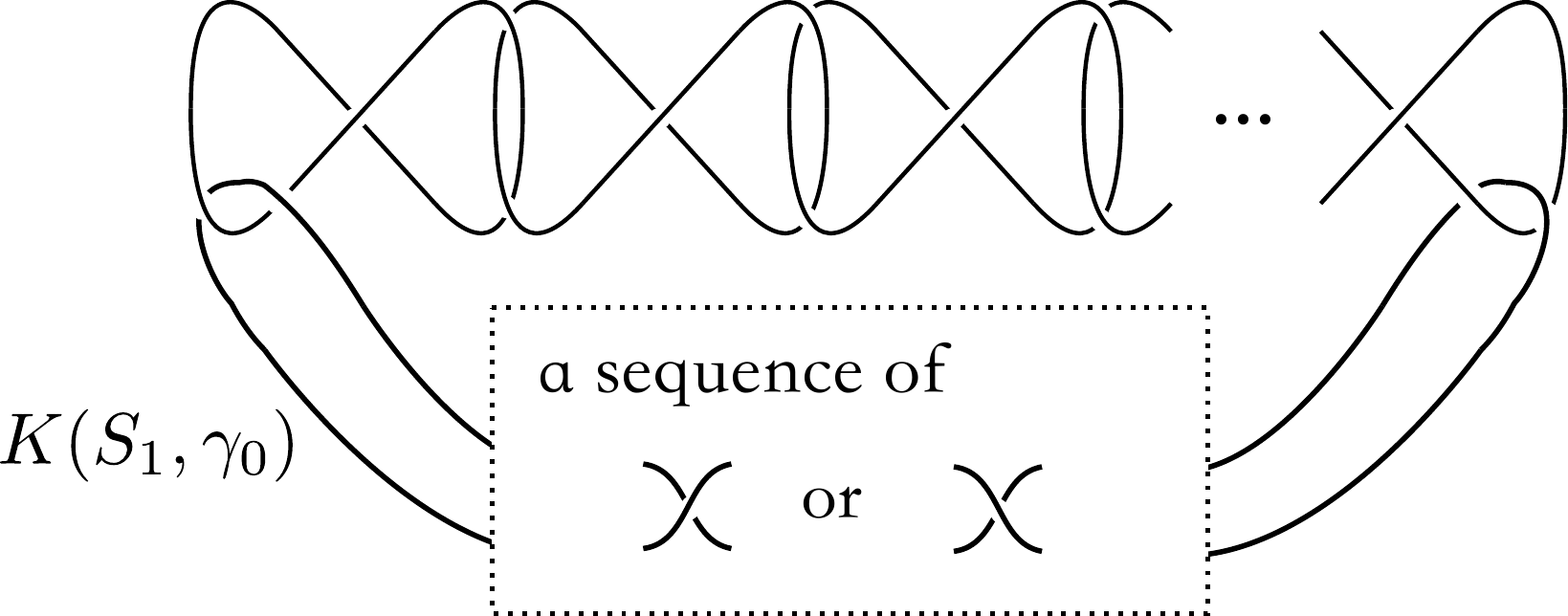}
  \caption{$K(S_1,\gamma_0)$.}
  \label{fig_pull_down_gamma}
\end{figure}
Therefore, we can shrink  $K(S_1,\gamma_0)$ to a neighborhood of $\gamma_0$, then pull down the knot and obtain a link that is given by Figure \ref{fig_pull_down_gamma}, hence the result is proved.
\end{proof}

A priori, the function $C(n)$ satisfying the statement of Lemma \ref{lem_exist_Cn} may not be unique. However, since the proof of the lemma is constructive, 
we will define $C(n)$ to be the function given by the proof of Lemma \ref{lem_exist_Cn}.
It is possible to write down a formula for $C(n)$ directly by tracing the crossings changes from Figure \ref{fig_twist_S1} to Figure \ref{fig_pull_down_gamma}. We take a slightly different approach that is less prone to mistakes.

\begin{Lemma}
The braid in the dotted rectangle of Figure \ref{fig_pull_down_gamma} which is given by the proof of Lemma \ref{lem_exist_Cn} is independent of $n$.
\end{Lemma}

\begin{proof}
When $n$ is increased by $1$, the arc $\gamma_0$ goes through two more crossings in Figure \ref{fig_twist_S1}. The two crossings are reverse to each other therefore they do not change the braid type in the dotted rectangle of Figure \ref{fig_pull_down_gamma}. 
\end{proof}

\begin{Corollary}
\label{cor_C(n+1)}
The function $C(n)$ satisfies $C(n+1)=C(n)-1$. \qed
\end{Corollary}

\begin{Lemma}
Under the condition of Case 1, the link $L$ is isotopic to $L_{n,1-n}$.
\end{Lemma}

\begin{proof}
It is straightforward to verify that when $n=3$, the link $L'\cup K(S_1,\gamma_0)$ is isotopic to $L_{3,-2}$, see Figure \ref{fig_isotopy1}. Therefore $L_{3,C(3)}$ is isotopic to $L_{3,-2}$. By equation \eqref{eqn_V_u,v}, we have $C(3) = -2$ or $-1$. On the other hand, the condition of Case 1 implies that $C(n)$ has the same parity as $n+1$, therefore $C(3) = -2$, hence by Corollary \ref{cor_C(n+1)}, we have $C(n) = 1-n$.
\begin{figure}
  \includegraphics[width=0.8\linewidth]{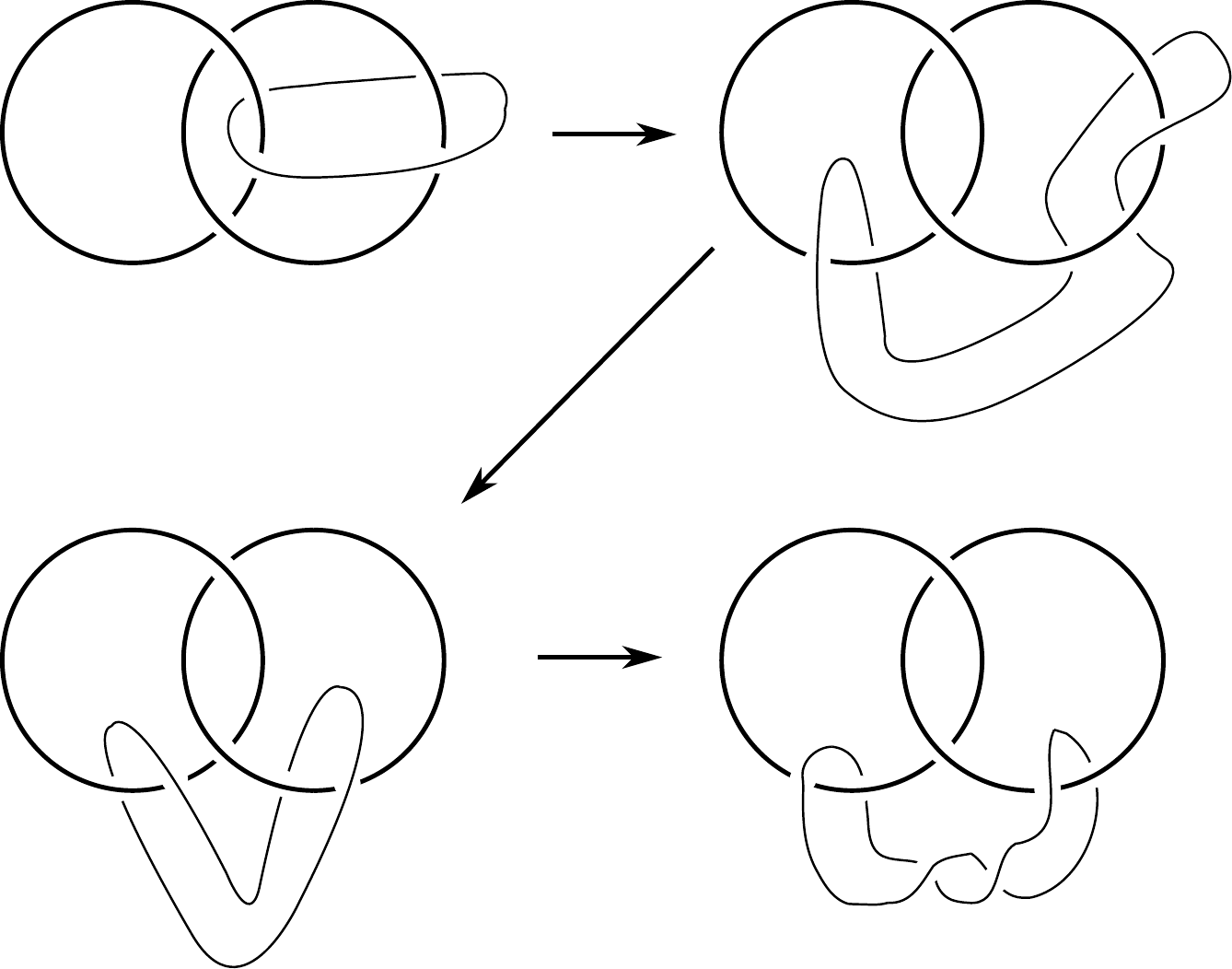}
  \caption{Isotopy from $L'\cup K(S_1,\gamma_0)$ to $L_{3,-2}$ when $n=3$.}
  \label{fig_isotopy1}
\end{figure}
\end{proof}

By Corollary \ref{cor_Jones_Luv}, if $n\ge 4$ then $\rank_{\bZ/2}\Kh(L_{n,1-n};{\mathbb{Z}/2})>2^n$.
It can be directly verified that $L_{3,-2}$ is isotopic (up to mirror image) to the link $L6n1$ in the Thistlethwaite link table, and the rank of $\Kh(L_{3,-2};{\mathbb{Z}/2})$ equals $12$. Therefore all the links $L_{n,1-n}$ fail to satisfy Part (2) of Condition \ref{cond_cycle}. This proves the non-existence of $L$ for Case 1.

To prove the result for Case 2, let $\gamma_0$ be the arc on $S_2$ given by Figure \ref{fig_standard_arc_2}. Then the same argument as in Lemma \ref{lem_relation_gamma_gamma0} and Corollary \ref{cor_isotopic_to_standard_Kn} shows that $K_n$ is isotopic to $K(S_2,\gamma_0)$ in $\bR^3-L'$. 
\begin{figure}
  \includegraphics[width=0.9\linewidth]{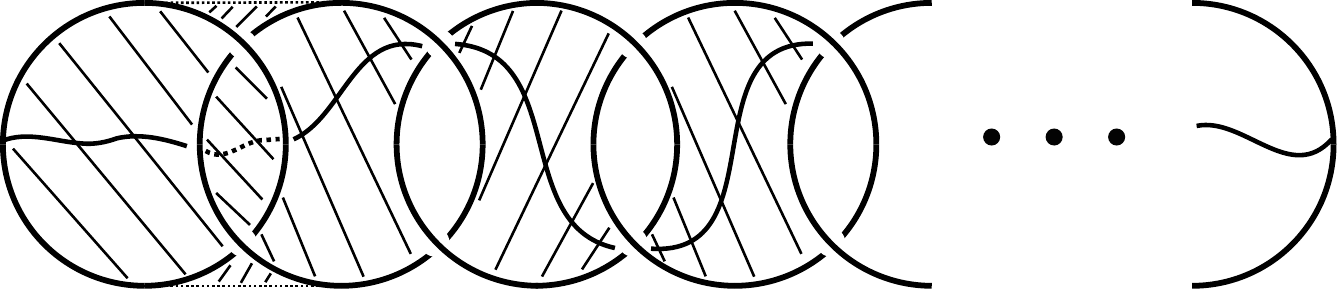}
  \caption{The arc $\gamma_0$ on $S_2$.}
  \label{fig_standard_arc_2}
\end{figure}
To see that $L$ is isotopic to $L_{n,C'(n)}$ for some function $C'(n)$, we use the diagram of $(S_2,\gamma_0)$ given by Figure \ref{fig_twist_S2}. 
\begin{figure}
  \includegraphics[width=\linewidth]{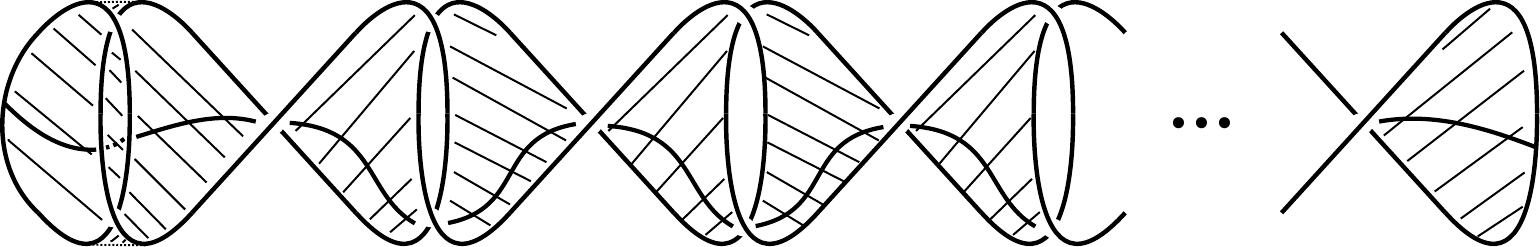}
  \caption{Another diagram for $S_2$ and $\gamma_0$.}
  \label{fig_twist_S2}
\end{figure}
Similar to Corollary \ref{cor_C(n+1)}, the function $C'(n)$ satisfies $C'(n+1)=C'(n)-1$ . Moreover, when $n=3$, the link $L'\cup K(S_2,\gamma_0)$ is isotopic to $L_{3,-1}$, see Figure \ref{fig_isotopy2}. Therefore equation \eqref{eqn_V_u,v} implies $C'(3) = -1$ or $-2$. On the other hand, the condition of Case 2 implies that $C'(n)$ has the same parity as $n$, therefore $C'(3) = -1$, hence $C'(n) = 2-n$.
\begin{figure}
  \includegraphics[width=\linewidth]{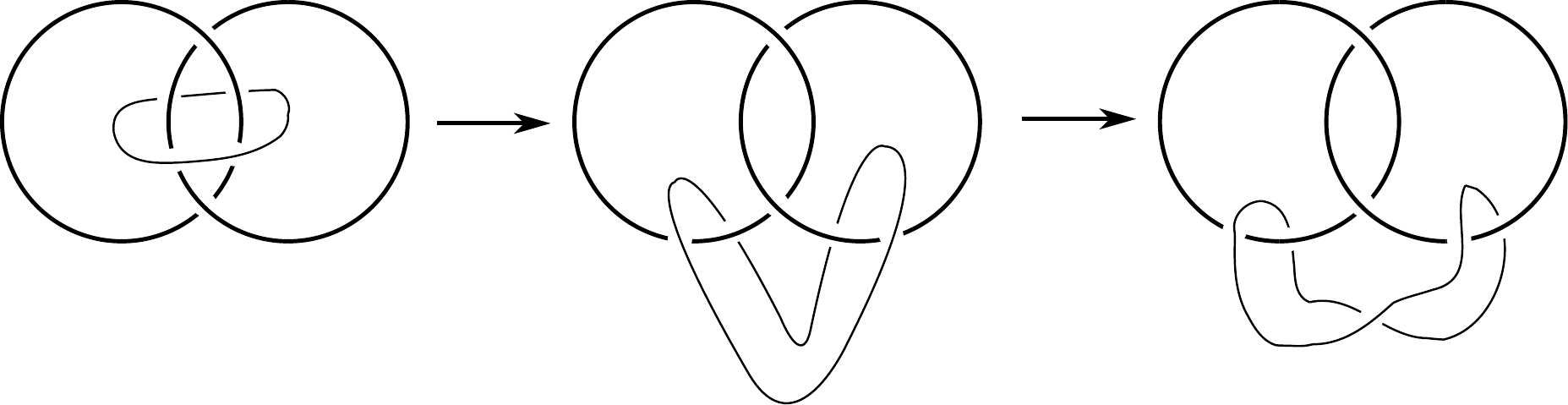}
  \caption{Isotopy from $L'\cup K(S_2,\gamma_0)$ to $L_{3,-1}$ when $n=3$.}
  \label{fig_isotopy2}
\end{figure}
By Corollary \ref{cor_Jones_Luv}, when $n\ge 6$ we have $\rank_{\bZ/2}\Kh(L_{n,2-n};{\mathbb{Z}/2})>2^n$. The link $L_{3,-1}$ is isotopic up to mirror image to the link L6n1 in the Thistlethwaite link table, and the rank of $\Kh(L_{3,-1};{\mathbb{Z}/2})$ equals $12$. The link $L_{4,-2}$ is isotopic up to mirror image to L8n8, and the rank of $\Kh(L_{4,-2};{\mathbb{Z}/2})$ equals $24$. The link $L_{5,-3}$ is isotopic up to mirror image to L10n113, and the rank of $\Kh(L_{5,-3};{\mathbb{Z}/2})$ equals $60$. This proves the desired result for Case 2.

In conclusion, we have proved that the link $L$ satisfying Condition \ref{cond_cycle} does not exist, therefore Theorem \ref{The_main_theorem} follows from Lemma \ref{thm_modulo_cond_cycle}.

\bibliographystyle{amsalpha}
\bibliography{references}

\begin{bibdiv}
\begin{biblist}

\bib{au2010lectures}{article}{
      author={Au, Thomas Kwok-Keung},
      author={Luo, Feng},
      author={Yang, Tian},
       title={Lectures on the mapping class group of a surface},
        date={2010},
     journal={Transformation groups and moduli spaces of curves},
      volume={16},
       pages={21\ndash 61},
}

\bib{AP:Kh_torsion}{incollection}{
      author={Asaeda, Marta~M.},
      author={Przytycki, J\'{o}zef~H.},
       title={Khovanov homology: torsion and thickness},
        date={2004},
   booktitle={Advances in topological quantum field theory},
      series={NATO Sci. Ser. II Math. Phys. Chem.},
      volume={179},
   publisher={Kluwer Acad. Publ., Dordrecht},
       pages={135\ndash 166},
         url={https://doi.org/10.1007/978-1-4020-2772-7_6},
      review={\MR{2147419}},
}

\bib{Birman}{book}{
      author={Birman, Joan~S.},
       title={Braids, links, and mapping class groups},
   publisher={Princeton University Press, Princeton, N.J.; University of Tokyo
  Press, Tokyo},
        date={1974},
        note={Annals of Mathematics Studies, No. 82},
      review={\MR{0375281}},
}

\bib{Kh-unlink}{article}{
      author={Batson, Joshua},
      author={Seed, Cotton},
       title={A link-splitting spectral sequence in {K}hovanov homology},
        date={2015},
        ISSN={0012-7094},
     journal={Duke Math. J.},
      volume={164},
      number={5},
       pages={801\ndash 841},
         url={https://doi.org/10.1215/00127094-2881374},
      review={\MR{3332892}},
}

\bib{BS}{article}{
      author={Baldwin, John~A},
      author={Sivek, Steven},
       title={Khovanov homology detects the trefoils},
        date={2018},
     journal={arXiv preprint arXiv:1801.07634},
}

\bib{BSX}{article}{
      author={Baldwin, John~A},
      author={Sivek, Steven},
      author={Xie, Yi},
       title={Khovanov homology detects the {H}opf links},
        date={2018},
     journal={to appear in Math. Res. Lett.; arXiv preprint, arXiv:1810.05040},
}

\bib{Dowlin}{article}{
      author={Dowlin, Nathan},
       title={A spectral sequence from {K}hovanov homology to knot {F}loer
  homology},
        date={2018},
     journal={arXiv preprint arXiv:1811.07848},
}

\bib{EKT:trivial_Jones}{article}{
      author={Eliahou, Shalom},
      author={Kauffman, Louis~H.},
      author={Thistlethwaite, Morwen~B.},
       title={Infinite families of links with trivial {J}ones polynomial},
        date={2003},
        ISSN={0040-9383},
     journal={Topology},
      volume={42},
      number={1},
       pages={155\ndash 169},
         url={https://doi.org/10.1016/S0040-9383(02)00012-5},
      review={\MR{1928648}},
}

\bib{feustel1966homotopic}{article}{
      author={Feustel, C.~D.},
       title={Homotopic arcs are isotopic},
        date={1966},
        ISSN={0002-9939},
     journal={Proc. Amer. Math. Soc.},
      volume={17},
       pages={891\ndash 896},
         url={https://doi.org/10.2307/2036278},
      review={\MR{196724}},
}

\bib{farb2011primer}{book}{
      author={Farb, Benson},
      author={Margalit, Dan},
       title={A primer on mapping class groups (pms-49)},
   publisher={Princeton University Press},
        date={2011},
}

\bib{freedman1990Topology}{book}{
      author={Freedman, Michael~H.},
      author={Quinn, Frank},
       title={Topology of 4-manifolds},
      series={Princeton Mathematical Series},
   publisher={Princeton University Press, Princeton, NJ},
        date={1990},
      volume={39},
        ISBN={0-691-08577-3},
      review={\MR{1201584}},
}

\bib{HN-unlink}{article}{
      author={Hedden, Matthew},
      author={Ni, Yi},
       title={Khovanov module and the detection of unlinks},
        date={2013},
        ISSN={1465-3060},
     journal={Geom. Topol.},
      volume={17},
      number={5},
       pages={3027\ndash 3076},
         url={https://doi.org/10.2140/gt.2013.17.3027},
      review={\MR{3190305}},
}

\bib{Kh-Jones}{article}{
      author={Khovanov, Mikhail},
       title={A categorification of the {J}ones polynomial},
        date={2000},
        ISSN={0012-7094},
     journal={Duke Math. J.},
      volume={101},
      number={3},
       pages={359\ndash 426},
         url={https://doi.org/10.1215/S0012-7094-00-10131-7},
      review={\MR{1740682}},
}

\bib{Khovanov-pattern}{article}{
      author={Khovanov, Mikhail},
       title={Patterns in knot cohomology. {I}},
        date={2003},
        ISSN={1058-6458},
     journal={Experiment. Math.},
      volume={12},
      number={3},
       pages={365\ndash 374},
         url={http://projecteuclid.org/euclid.em/1087329238},
      review={\MR{2034399}},
}

\bib{KM:Alexander}{article}{
      author={Kronheimer, P.~B.},
      author={Mrowka, T.~S.},
       title={Instanton {F}loer homology and the {A}lexander polynomial},
        date={2010},
        ISSN={1472-2747},
     journal={Algebr. Geom. Topol.},
      volume={10},
      number={3},
       pages={1715\ndash 1738},
         url={https://doi.org/10.2140/agt.2010.10.1715},
      review={\MR{2683750}},
}

\bib{KM:Kh-unknot}{article}{
      author={Kronheimer, P.~B.},
      author={Mrowka, T.~S.},
       title={Khovanov homology is an unknot-detector},
        date={2011},
        ISSN={0073-8301},
     journal={Publ. Math. Inst. Hautes \'Etudes Sci.},
      number={113},
       pages={97\ndash 208},
         url={http://dx.doi.org/10.1007/s10240-010-0030-y},
      review={\MR{2805599}},
}

\bib{KM:YAFT}{article}{
      author={Kronheimer, P.~B.},
      author={Mrowka, T.~S.},
       title={Knot homology groups from instantons},
        date={2011},
        ISSN={1753-8416},
     journal={J. Topol.},
      volume={4},
      number={4},
       pages={835\ndash 918},
         url={https://doi.org/10.1112/jtopol/jtr024},
      review={\MR{2860345}},
}

\bib{KM-Ras}{article}{
      author={Kronheimer, P.~B.},
      author={Mrowka, T.~S.},
       title={Gauge theory and {R}asmussen's invariant},
        date={2013},
        ISSN={1753-8416},
     journal={J. Topol.},
      volume={6},
      number={3},
       pages={659\ndash 674},
         url={https://doi.org/10.1112/jtopol/jtt008},
      review={\MR{3100886}},
}

\bib{KM:embedded-surfaces-II}{article}{
      author={Kronheimer, P.~B.},
      author={Mrowka, T.~S.},
       title={Gauge theory for embedded surfaces. {II}},
        date={1995},
        ISSN={0040-9383},
     journal={Topology},
      volume={34},
      number={1},
       pages={37\ndash 97},
         url={https://doi.org/10.1016/0040-9383(94)E0003-3},
      review={\MR{1308489}},
}

\bib{Kr-ob}{article}{
      author={Kronheimer, Peter~B},
       title={An obstruction to removing intersection points in immersed
  surfaces},
        date={1997},
     journal={Topology},
      volume={36},
      number={4},
       pages={931\ndash 962},
}

\bib{LS-split}{article}{
      author={Lipshitz, Robert},
      author={Sarkar, Sucharit},
       title={Khovanov homology also detects split links},
        date={2019},
     journal={arXiv preprint, arXiv:1910.04246},
}

\bib{Mor-braid}{incollection}{
      author={Morton, H.~R.},
       title={Exchangeable braids},
        date={1985},
   booktitle={Low-dimensional topology ({C}helwood {G}ate, 1982)},
      series={London Math. Soc. Lecture Note Ser.},
      volume={95},
   publisher={Cambridge Univ. Press, Cambridge},
       pages={86\ndash 105},
         url={https://doi.org/10.1017/CBO9780511662744.003},
      review={\MR{827298}},
}

\bib{OS:HFK}{article}{
      author={Ozsv\'{a}th, Peter},
      author={Szab\'{o}, Zolt\'{a}n},
       title={Holomorphic disks and knot invariants},
        date={2004},
        ISSN={0001-8708},
     journal={Adv. Math.},
      volume={186},
      number={1},
       pages={58\ndash 116},
         url={https://doi.org/10.1016/j.aim.2003.05.001},
      review={\MR{2065507}},
}

\bib{OS:HFL}{article}{
      author={Ozsv\'{a}th, Peter},
      author={Szab\'{o}, Zolt\'{a}n},
       title={Holomorphic disks, link invariants and the multi-variable
  {A}lexander polynomial},
        date={2008},
        ISSN={1472-2747},
     journal={Algebr. Geom. Topol.},
      volume={8},
      number={2},
       pages={615\ndash 692},
         url={https://doi.org/10.2140/agt.2008.8.615},
      review={\MR{2443092}},
}

\bib{Ras:HFK}{book}{
      author={Rasmussen, Jacob~Andrew},
       title={Floer homology and knot complements},
   publisher={ProQuest LLC, Ann Arbor, MI},
        date={2003},
        ISBN={978-0496-39374-9},
  url={http://gateway.proquest.com/openurl?url_ver=Z39.88-2004&rft_val_fmt=info:ofi/fmt:kev:mtx:dissertation&res_dat=xri:pqdiss&rft_dat=xri:pqdiss:3091665},
        note={Thesis (Ph.D.)--Harvard University},
      review={\MR{2704683}},
}

\bib{Sar-HFL}{article}{
      author={Sarkar, Sucharit},
       title={A note on sign conventions in link {F}loer homology},
        date={2011},
        ISSN={1663-487X},
     journal={Quantum Topol.},
      volume={2},
      number={3},
       pages={217\ndash 239},
         url={https://doi.org/10.4171/QT/20},
      review={\MR{2812456}},
}

\bib{Shu:torsion_Kh}{article}{
      author={Shumakovitch, Alexander~N.},
       title={Torsion of {K}hovanov homology},
        date={2014},
        ISSN={0016-2736},
     journal={Fund. Math.},
      volume={225},
      number={1},
       pages={343\ndash 364},
         url={https://doi.org/10.4064/fm225-1-16},
      review={\MR{3205577}},
}

\bib{Xie-earring}{article}{
      author={Xie, Yi},
       title={Earrings, sutures and pointed links},
        date={2018},
     journal={arXiv preprint, arXiv:1809.09254},
}

\bib{AHI}{article}{
      author={Xie, Yi},
       title={Instantons and annular {K}hovanov homology},
        date={2018},
     journal={arXiv preprint, arXiv:1809.01568},
}

\bib{XZ:excision}{article}{
      author={Xie, Yi},
      author={Zhang, Boyu},
       title={Instanton {F}loer homology for sutured manifolds with tangles},
        date={2019},
     journal={arXiv preprint, arXiv:1907.00547},
}

\end{biblist}
\end{bibdiv}

\end{document}